\documentclass[a4paper]{amsart}
\usepackage{amsmath}
\usepackage{amsthm}
\usepackage{amsfonts}
\usepackage{amssymb}
\usepackage{mathabx}
\usepackage{mathtools}
\usepackage{hyperref}
\usepackage[all]{xy}
\usepackage{graphics,graphicx}
\usepackage{mathdots}
\usepackage[english]{babel}

\usepackage[inner=2cm,outer=2cm]{geometry}
\usepackage{xcolor}

\newcommand{\Z}{\mathbb Z}

\newcommand{\R}{\mathbb R}

\newcommand{\Ax}{\mathrm{Axis}}
\newcommand{\Aut}{\mathrm{Aut}}
\newcommand{\Stab}{\mathrm{Stab}}

\newcommand{\BS}{\mathrm{BS}}
\newcommand{\supp}{\mathrm{supp}}
\newcommand{\link}{\mathrm{link}}
\newcommand{\diam}{\mathrm{diam}}
\newcommand{\cont}{\text{cont}}
\newcommand{\im}{\mathrm{im}}
\newcommand{\id}{\mathrm{id}}

\newcommand{\lbr}{\{\!\!\{}
\newcommand{\rbr}{\}\!\!\}}

\newtheorem{thmx}{Theorem}
\newtheorem{corx}[thmx]{Corollary}
\newtheorem{quex}[thmx]{Question}

\newtheorem*{thmA}{Theorem \ref{mainT}}
\newtheorem*{thmB}{Theorem \ref{thmB}}
\newtheorem*{proof42}{\emph{Proof of Corollary \ref{mainC}}}
\newtheorem*{thmGP}{Theorem \ref{mainGraphProducts}}
\newtheorem*{cor52}{Corollary \ref{mainC}}
\newtheorem*{distance_formula}{Distance Formula for hierarchically hyperbolic spaces}

\newtheorem{theorem}{Theorem}[section]

\newtheorem{proposition}[theorem]{Proposition}

\newtheorem{lemma}[theorem]{Lemma}

\newtheorem{corollary}[theorem]{Corollary}

\theoremstyle{definition}
\newtheorem{definition}[theorem]{Definition}

\newtheorem{example}[theorem]{Example}
\newtheorem{remark}[theorem]{Remark}
\newtheorem*{convention}{Convention}

\title{A refined combination theorem for hierarchically hyperbolic groups}
\author{Federico Berlai}
\author{Bruno Robbio}
\address{Department of Mathematics, UPV/EHU, Sarriena s/n, 48940, Leioa - Bizkaia, Spain}
\email[Federico Berlai]{federico.berlai@ehu.eus}
\email[Bruno Robbio]{bruno.robbio@ehu.eus}
\date{\today}

\begin{document}

\begin{abstract}
In this work, we are concerned with hierarchically hyperbolic spaces and hierarchically hyperbolic groups. Our main result is a wide generalization of a combination theorem of Behrstock, Hagen, and Sisto. In particular, as a consequence,  we show that  any finite graph product of hierarchically hyperbolic groups is again a hierarchically hyperbolic group, thereby answering \cite[Question D]{BHS2} posed by Behrstock, Hagen, and Sisto.
In order to operate in such a general setting, we establish a number of structural results for hierarchically hyperbolic spaces and hieromorphisms (that is, morphisms between such spaces),
and we introduce two new notions for hierarchical hyperbolicity, that is \emph{concreteness} and the \emph{intersection property}, proving that they are satisfied in all known examples.

\end{abstract}
\maketitle

\section{Introduction}
Many seemingly different classes of groups, including hyperbolic $3$-manifold groups, surface groups, small cancellation groups, share some common behavior and properties that were identified by Gromov
(and others) and synthesized into the class of hyperbolic groups \cite{Gr}. 
Gromov's breakthrough consisted in considering groups purely as geometric objects, abstracting the properties shared by the above-mentioned classes. This geometric approach implies strong algebraic, asymptotic,
and growth properties: hyperbolic groups are finitely presented, they have exponential growth (except the virtually cyclic ones), they satisfy a strong form of Tits' alternative and a linear isoperimetric inequality.
As already stressed by Gromov, some natural groups of geometric origin do not fit into this picture: fundamental groups of $3$-manifolds with cusps and mapping class groups are in general not hyperbolic. What is more, the class of hyperbolic
groups is closed under taking free products, but not direct products. Therefore, as proved by Meier \cite{Me}, a graph product of hyperbolic groups is again a hyperbolic group (if and) only if some strong conditions are met.

To overcome these limitations, several generalizations of hyperbolic groups have been introduced over the years. The notion of relative hyperbolicity \cite{Bo,Os1} recovers fundamental groups of $3$-manifolds with cusps, whereas
mapping class groups are examples of acylindrically hyperbolic groups \cite{Os2}, and raags (that is right-angled Artin groups) are among the groups acting properly and cocompactly on CAT(0) cube complexes, that is cubulable groups~\cite{Sa1,Wi}. 
On the one hand, mapping class groups are not relatively hyperbolic (unless they are already hyperbolic \cite[Theorem 1.2]{BDM}). On the other one, the class of acylindrically hyperbolic groups is extensive, and acylindrical hyperbolicity does not imply finite presentability, or any kind of bound on the isoperimetric inequality. Therefore, 
one is brought to find a set of properties that would generalize hyperbolicity, include mapping class groups, be preserved by direct products, and still have strong algebraic consequences for groups satisfying~them. 

These conditions have been identified by Behrstock, Hagen, and Sisto, who isolated the notions of \emph{hierarchically hyperbolic spaces} and of \emph{hierarchically hyperbolic groups} \cite{BHS1,BHS2}. Again, the geometric approach 
that is undertaken reflects into strong algebraic and asymptotic properties: 
hierarchically hyperbolic groups are finitely presented \cite[Corollary 7.5]{BHS2}, they satisfy a quadratic isoperimetric inequality \cite[Corollary 7.5]{BHS2}, they are coarse median \cite[Theorem 7.3]{BHS2}, 
and they have finite asymptotic dimension \cite{asdim}.

The key insight was to axiomatize the Masur-Minsky machinery for mapping class groups, to be able to apply it to its full extent. Although not being in general hyperbolic, a mapping class group 
$\mathcal{M}\mathcal{C}\mathcal{G}(S)$ (of a surface $S$ of finite complexity) can be studied, using the tools developed by Masur and Minsky \cite{MM1, MM2}, through a family of hyperbolic spaces, the curve complexes $\mathcal{C}V$ associated to 
subsurfaces $V\subseteq S$. In a similar manner, hierarchically hyperbolic spaces and groups are the ones for which an analogous approach can be undertaken.

Hierarchically hyperbolic groups provide a common framework to work with hyperbolic groups, mapping class groups, and raags. Other examples comprise all known cubulable groups~\cite{HaSu}, 
toral relatively hyperbolic groups~\cite[Theorem 9.1]{BHS2}, fundamental groups of many $3$-manifolds~\cite[Theorem~10.1]{BHS2}, free and direct products~\cite{BHS2} of these. 
Hierarchically hyperbolic spaces include all hierarchically hyperbolic groups, the Teichm\"{u}ller space with either the Thurston or the Weil-Petersson metric, 
almost all separating curve graphs (and other graphs of multicurves) of surfaces \cite{KateVokes}, universal covers of compact special cube complexes, and any space quasi-isometric to a hierarchically hyperbolic space.

Hierarchical hyperbolicity has been used to prove several new results, and to uniformize results previously known only for certain subclasses of hierarchically hyperbolic spaces.
In \cite{quasiflats}, Behrstock, Hagen and Sisto show that, in a hierarchically hyperbolic space, any top-dimensional quasiflat is uniformly close to a union of standard orthants. This strengthened the known results \cite{BKS,huang} 
in the cubulable setting, and resolved conjectures of Farb for mapping class groups and Brock for Teichm\"{u}ller spaces.
In \cite{asdim}, as already mentioned, the same authors show that every hierarchically hyperbolic space has finite asymptotic dimension, and obtain the sharpest known bound on the asymptotic dimension of mapping class groups. 

\smallskip
The definition of hierarchical hyperbolicity is rather technical and we postpone it until Section \ref{subsection_definition_hhs}.
For the time being, it is enough to know that a hierarchically hyperbolic space $(\mathcal{X},\mathfrak{S})$ is a metric space $(\mathcal{X},d_{\mathcal{X}})$ equipped with a collection of $\delta$-hyperbolic spaces 
$\{\mathcal{C}V\mid V\in\mathfrak{S}\}$, and projections $\pi_V$ from $\mathcal{X}$ onto the various hyperbolic spaces $\mathcal{C}V$, for all $V\in\mathfrak{S}$. 
This index set $\mathfrak{S}$ is equipped with a partial order called nesting, a symmetric and anti-reflexive relation called orthogonality, and if $V, U \in\mathfrak{S}$ are neither nested nor orthogonal, then they are transverse. 
These relations are mutually exclusive, and in the mapping class group scenario their role is respectively taken by nesting, disjointness and overlapping of subsurfaces.
The projections onto hyperbolic spaces and the nesting, orthogonality and transversality satisfy, in addition, several axioms (see Definition \ref{HHS_definition}), which again are evocative of mapping class groups, 
and assure that the coarse geometry of the space $\mathcal{X}$ can be reconstructed from the hierarchically hyperbolic structure.

This leads to one of the most salient features of hierarchical hyperbolicity: a distance formula that generalizes the celebrated distance formula for mapping class groups of Masur and Minsky \cite{MM2}.
In other words, distances in a hierarchically hyperbolic space $(\mathcal{X},\mathfrak{S})$ can be (uniformly) coarsely computed by projecting onto the various hyperbolic spaces associated to $\mathfrak{S}$, determining distances there,
and then sum.
This is made precise by the following theorem:
\begin{distance_formula}[{\cite[Theorem 4.5]{BHS2}}]
Let $(\mathcal{X},\mathfrak{S})$ be a hierarchically hyperbolic space. There exists $s_0$ such that for all $s\geq s_0$ there exist constants $k,c>0$ such that
\[d_{\mathcal{X}}(x,y)\asymp_{(k,c)} \sum_{V\in\mathfrak{S}}\lbr d_V(\pi_V(x),\pi_V(y))\rbr_s,\qquad\qquad \forall\, x,y\in\mathcal{X},\]
where the symbol $\lbr a\rbr_s$ means that $a$ is added to the sum only if $a\geq s$, and $a\asymp_{(k,c)}b$ stands for $\frac{b}{k}-c\leq a\leq kb+c$. 
\end{distance_formula} 

For what concerns hierarchically hyperbolic groups, at this time let us just mention that there exist groups that are hierarchically hyperbolic spaces, but fail to be hierarchically hyperbolic \emph{groups}, and therefore being a hierarchically
hyperbolic group is a stronger condition than having a Cayley graph which is a hierarchically hyperbolic space (see Definition \ref{def:hhg}).

\medskip
Given a class of groups $\mathcal{C}$, it is natural to investigate under which group constructions the class is preserved.
On the one hand, the fact that $\mathcal{C}$ is closed under certain operations gives information on the nature of the class, and, on the other, it provides methods to construct new groups in the class from known examples. 

A construction that generalizes free products 
(with amalgamation, and HNN extensions) is the fundamental group of a graph of groups, and results in this direction are usually referred to as \emph{combination theorems}. 
The Bestvina-Feighn combination theorem \cite{BestvinaFeighn} for hyperbolic groups is such an example: given a finite graph $\mathcal{G}$ of hyperbolic groups satisfying certain conditions, the resulting fundamental group is again hyperbolic.
Their strategy of proof was to consider a metric space (more precisely, a \emph{tree of metric spaces} obtained from the Bass-Serre tree of the graph and the vertex/edge groups of $\mathcal{G}$) and study the action of the fundamental group on such space.
This approach turned out to be very successful, and was later applied in several other related contexts. This is the case for the combination theorem of \cite{StrongRelGps} in the class of strongly relatively hyperbolic groups, 
or for the Hsu-Wise combination theorem in the context of groups acting on cube complexes \cite{WiHs}, or Alibegovi\'{c}'s combination theorem for relatively hyperbolic groups \cite{EminaAlibegovic}. 
On the other hand, a more dynamical approach is undertaken  by Dahmani \cite{Dahm} to obtain another combination theorem for relatively hyperbolic groups.

Also in the context of hierarchically hyperbolic groups and spaces, there have been efforts in establishing such combination theorems. In \cite[Section 8]{BHS2}, Behrstock, Hagen and Sisto impose strict conditions on a tree of 
hierarchically hyperbolic spaces (something completely analogous to the trees of hyperbolic groups considered by Bestvina and Feighn, and mentioned previously - see Definition~\ref{equiv_class_def}) that ensure that the resulting 
space is again hierarchically hyperbolic. From this, they deduce \cite[Corollary 8.24]{BHS2} the hierarchical hyperbolicity of fundamental groups of finite graph of groups satisfying related strict conditions.
In \cite[Theorem 4.17]{Sp2}, Spriano shows that certain amalgamated products of hierarchically hyperbolic groups are hierarchically hyperbolic, building on results from his previous work~\cite{Sp1}. 

\smallskip
In this work we provide a new combination theorem for hierarchically hyperbolic spaces (see Theorem \ref{mainT}) and groups (see Corollary \ref{mainC}). 
To do so, we introduce several new tools for the study of hierarchical hyperbolicity, which are of independent interest. 
The first one is the \emph{intersection property} (see Definition~\ref{intersectionproperty_definition}, and the discussion after the statement of Theorem \ref{mainGraphProducts}), which in
turn leads to the notion of \emph{concreteness}. We introduce the latter notion to exclude artificial examples of hierarchically hyperbolic spaces that carry some undesirable features. As we will see in this Introduction, the intersection property has a very natural definition, and we conjecture that all hyperbolic spaces
admit a hierarchically hyperbolic structure with the intersection property (see Question~\ref{conjectu}). On the other hand, concreteness is more technical, but 
nevertheless we prove in Proposition \ref{concreteness_making} that any hierarchically hyperbolic space with the intersection property can be supposed to be concrete.

These properties are of independent interest, and we expect them to be of further use. They allow us to assume much weaker hypotheses for our combination theorem than the ones used by
Behrstock, Hagen, and Sisto \cite[Theorem 8.6]{BHS2}. 

\smallskip
The first result of this paper is the following combination theorem.
After having stated it, we will briefly comment on terminology and some concepts related to hierarchical hyperbolicity, relegating their full and precise introduction to Section \ref{SectionHH}. 

\begin{thmx}\label{mainT}
Let $\mathcal{T}$ be a tree of hierarchically hyperbolic spaces. Suppose that: 
\begin{enumerate}
\item each edge-hieromorphism is hierarchically quasiconvex, uniformly coarsely lipschitz and full;
\item\label{hyp3} comparison maps are uniform quasi isometries;
\item the hierarchically hyperbolic spaces of $\mathcal{T}$ have the intersection property and clean containers.
\end{enumerate} 
Then the metric space $\mathcal{X}(\mathcal{T})$ associated to $\mathcal{T}$ is a hierarchically hyperbolic space with clean containers and the intersection property.
\end{thmx}
As already mentioned, $\mathcal{X}(\mathcal{T})$ is a metric space associated to $\mathcal{T}$, and it is built from a tree, replacing vertices and edges with hierarchically hyperbolic spaces, with embeddings of edge spaces into vertex spaces 
(see Definition~\ref{equiv_class_def}). These embeddings are given by hieromorphisms, which are morphisms between hierarchically hyperbolic spaces that agree with the hierarchical structure (see Definition~\ref{HHS_hieromorphism}, and 
Definition~\ref{fullness_definition} for the notion of full hieromorphism). 
Theorem~\ref{mainT} then has three hypotheses: the first two are metric conditions, one of them imposing constraints on how the edge groups in $\mathcal{T}$ are embedded into vertex groups, and the other one requiring certain natural 
maps (compare Definition \ref{comparison_maps}) at the level of the hyperbolic spaces $\mathcal{C}V$ to be isometries. 
These are the two fundamental hypotheses of the theorem, and, as we will see, neither of the two can be dropped or relaxed.

The third hypothesis invokes two conditions that, in view of the motivating examples, are very natural. These two properties are known to persist under all known operations that preserve the classes of hierarchically hyperbolic spaces and groups, and currently they are
satisfied in all examples of hierarchically hyperbolic spaces.

\smallskip
As a consequence of Theorem \ref{mainT}, we obtain a combination theorem for hierarchically hyperbolic groups:
\begin{corx}\label{mainC}
Let $\mathcal{G}$ be a finite graph of hierarchically
hyperbolic groups. Suppose that:
\begin{enumerate}
\item each edge-hieromorphism is hierarchically quasiconvex, uniformly coarsely lipschitz and full;
\item comparison maps are isometries;
\item the hierarchically hyperbolic spaces of $\mathcal{G}$ have the intersection property and clean containers.
\end{enumerate}
Then the group associated to $\mathcal{G}$ is itself a hierarchically hyperbolic group.
\end{corx}
Notice that, in contrast with Theorem \ref{mainT}, in Corollary \ref{mainC} comparison maps are required to be \emph{isometries}, and not just uniform quasi isometries. This is needed in the proof that the hierarchical structure given to the fundamental group of $\mathcal{G}$ provided by Theorem \ref{mainT} is a hierarchically hyperbolic \emph{group} structure, and not just a hierarchically hyperbolic space structure. We direct the interested reader to Remark \ref{remark_qi} and Corollary \ref{corollary_qi} for more regarding this.

\smallskip
A more involved application of Theorem \ref{mainT} and Corollary \ref{mainC} is the following Theorem \ref{mainGraphProducts}, which is concerned with persistence of hierarchical hyperbolicity under taking graph products.
Theorem \ref{mainGraphProducts} answers in the positive a question posed by Behrstock, Hagen, and Sisto \cite[Question D]{BHS2}. As a byproduct of Theorem \ref{mainGraphProducts}, we extend
the results of \cite{ABD} to show that clean containers are not only preserved by taking free and direct products, but also by graph products.

\begin{thmx}\label{mainGraphProducts}
Let $\Gamma$ be a finite simplicial graph, $\mathcal{G}=\{G_v\}_{v\in V}$ be a family of hierarchically hyperbolic groups with the intersection property and clean containers. Then 
the graph product $\Gamma\mathcal{G}$ is a hierarchically hyperbolic group with the intersection property and clean containers.
\end{thmx}

\emph{Clean containers} (see Remark \ref{container_remark}), a notion introduced originally by Abbott, Behrstock, and Durham \cite{ABD},
is a technical condition that in the mapping class group setting translates into the following: if $V\subseteq S$ is a subsurface of the surface $S$, then $V$ and $S\setminus V$ are disjoint, and \emph{any} subsurface disjoint from $V$ is contained into~$S\setminus V$.
On the other hand, the intersection property is a condition that we introduce, and in the mapping class group setting means that, given two subsurfaces
$V,U\subseteq S$, the subsurface $V\cap U$ is the biggest subsurface of 
$S$ that is contained in both $V$ and $U$. The intersection property gives to the index set $\mathfrak{S}$ the structure of a lattice.
At this point, it is instructive to notice that both $V\cap U$ and $S\setminus V$ could be \emph{non-connected} subsurfaces of $S$, and indeed the hierarchically hyperbolic 
structure with clean containers and the intersection property of a mapping class group $\mathcal{M}\mathcal{C}\mathcal{G}(S)$ is obtained considering all, possibly non-connected, subsurfaces of~$S$.

Both properties are satisfied in all known examples of hierarchically hyperbolic spaces, in the sense that given a hierarchically hyperbolic space $\mathcal{X}$ (respectively: hierarchically hyperbolic 
group $G$), there exists a hierarchical structure $\mathfrak{S}$ such that $(\mathcal{X},\mathfrak{S})$ is a hierarchically hyperbolic space (respectively: $(G,\mathfrak{S})$ is a hierarchically hyperbolic group) with the intersection property and clean containers.

\smallskip
We are inclined to believe that any hierarchically hyperbolic space admits a hierarchically hyperbolic structure with the intersection property and clean containers:
\begin{quex}\label{conjectu}
Let $(\mathcal{X},d_{\mathcal{X}})$ be a hierarchically hyperbolic space. Does there exist a hierarchically hyperbolic structure~$\mathfrak{S}$ such that
$(\mathcal{X},\mathfrak{S})$ is a hierarchically hyperbolic space with the intersection property and clean containers?
\end{quex}

All the stated theorems rely on the following fundamental result, Theorem \ref{thmB}, which is of independent interest. It provides equivalent conditions for a (full) hieromorphism $\phi\colon(\mathcal{X},\mathfrak{S})\to(\mathcal{X}',\mathfrak{S}')$
with hierarchically quasiconvex image, to be a coarsely lipschitz map. 

An interesting feature of Theorem \ref{thmB} is the following. On the one hand, its first two conditions are purely metric conditions on the hieromorphism, 
whereas the third is a metric condition on
certain natural maps (that is \emph{gate maps}, see Remark \ref{def_gatemaps_nontree}) between hierarchically quasiconvex subspaces of the hierarchically hyperbolic structure of 
$(\mathcal{X}',\mathfrak{S}')$. On the other hand, after the image of the hieromorphism $\phi$ is understood, the fourth and the fifth conditions can be detected in $(\mathcal{X}',\mathfrak{S}')$.

Therefore, Theorem \ref{thmB} reveals that a seemingly mild condition on $\phi$ 
(being coarsely lipschitz) already guarantees that the hieromorphism is a quasi-isometric embedding, and has implications on the hierarchically hyperbolic structure of $(\mathcal{X}',\mathfrak{S}')$.

\begin{thmx}\label{thmB}
Let $\phi\colon(\mathcal{X},\mathfrak{S})\to(\mathcal{X}',\mathfrak{S}')$ be a full hieromorphism with hierarchically quasiconvex image, and let 
$S$ be the $\sqsubseteq$-maximal element of $\mathfrak{S}$. 
The following are equivalent:
\begin{enumerate}
 \item $\phi$ is coarsely lipschitz;
 \item $\phi$ is a quasi-isometric embedding;
  \item the maps $\mathfrak{g}_{\phi(\mathcal{X})}\colon{\textbf F}_{\phi(S)}\to\phi(\mathcal{X})$ and $\mathfrak{g}_{{\textbf F}_{\phi(S)}}\colon\phi(\mathcal{X})\to\textbf{F}_{\phi(S)}$ are quasi-inverses of each other, and in particular quasi isometries;
 \item the subspace $\phi(\mathcal{X})\subseteq \mathcal{X}'$, endowed with the subspace metric, admits a hierarchically hyperbolic structure obtained from the 
 one of $\mathcal{X}$ by composition with the map $\phi$ (and its induced maps at the level of hyperbolic spaces);
 \item $\pi_{W}(\phi(\mathcal{X}))$ is uniformly bounded for every $W\in\mathfrak{S}'\setminus\phi(\mathfrak{S})$.
\end{enumerate}
\end{thmx}

\subsection*{Organisation of the paper}
The paper is organized as follows. In Section~\ref{SectionHH} we fix the notation, and recall all the necessary definitions and facts concerning 
hierarchically hyperbolic spaces and groups. In Section~\ref{section3} we introduce the notions of \emph{intersection property}, of $\varepsilon$-\emph{support}, and of \emph{concreteness} of a hierarchically hyperbolic space 
(see Definition \ref{intersectionproperty_definition}, Definition \ref{supports}, and Definition \ref{concreteness}).
As already mentioned, we conjecture that all hierarchically hyperbolic spaces satisfy the intersection property. On the other hand, concreteness is a technical condition
that will play a pivotal role in the proofs of Theorem~\ref{winning_theorem} and of Theorem~\ref{mainT}.
In Proposition~\ref{concreteness_making} we prove that, without loss of generality, any hierarchically hyperbolic space with the intersection property can be assumed to be concrete (and this is why 
concreteness does not appear as an hypothesis in Theorem~\ref{mainT}). 

In Section \ref{Section4} we prove Theorem~\ref{thmB} of the Introduction, which is then used in the proofs of Theorem~\ref{winning_theorem} and 
Lemma~\ref{consequence_winning_theorem}. These results will be applied repeatedly in Section~\ref{Section5}, which is devoted to the proof of Theorem~\ref{mainT}. In Subsection \ref{trees_with_decorations}
we introduce a trick, which we call the \emph{decoration} of a tree of hierarchically hyperbolic spaces~$\mathcal{T}$, which is fundamental for our approach to prove Theorem~\ref{mainT}.

In Subsection \ref{nes_ort_tra}, Subsection \ref{unif_hyp}, and Subsection \ref{proj_hyp} we built the index set needed for Theorem~\ref{mainT}, and describe projections onto the hyperbolic 
spaces associated to this index set. Finally, in Subsection \ref{proof_mainT} we prove Theorem \ref{mainT}.
We conclude with Section \ref{last_section}, where we discuss the connections of our result to the combination theorem of Bestvina and Feighn (see Subsection \ref{BFsubsection}), and where 
the applications of Theorem~\ref{mainT} can be found, that is, where we prove Corollary \ref{mainC} and Theorem \ref{mainGraphProducts}.

\subsubsection*{Acknowledgements}
Bruno Robbio is, and Federico Berlai was, supported by the ERC grant PCG-336983, Basque Government Grant IT974-16, and Ministry of Economy, Industry and Competitiveness of the Spanish Government Grant MTM2017-86802-P. Federico Berlai is now supported by the Austrian Science Foundation FWF, grant no.~J4194. Bruno Robbio also acknowledges the support of the University of the Basque Country through the grant PIF17/241.
We are very grateful to Montserrat Casals Ruiz, Mark Hagen, and Ilya Kazachkov, for fruitful, inspiring, and continuous conversations over the topic of this work. The first-named author is
grateful to Alessandro Sisto for stimulating discussions over a previous version of the paper during the workshop \textquotedblleft Non-positively curved groups and spaces\textquotedblright\ that took place in Regensburg, September 2017. We are grateful to the anonymous reviewer for their many suggestions that helped improve the exposition of the paper.

\section{Preliminaries}\label{SectionHH}
In this section, after setting the notation, we will recall the notions of hierarchically hyperbolic spaces and groups, and of morphisms - that is 
\emph{hieromorphisms} - between these spaces, following \cite{BHS2}.

\subsection*{Notation} 
For real-valued functions $A$ and $B$, we write $A\asymp_{(K,C)}B$ if there exist constants $C$ and $K$ such that 
\[K^{-1}B(x)-C\leq A(x)\leq KB(x)+C\]
for all $x$ in the domain of the functions.
With $A\asymp B$ we intend that there exist real numbers $C$ and $K$ such that $A\asymp_{(K,C)}B$.

Moreover, for real numbers $a,b$ we define
\begin{equation*}
\lbr a \rbr_{b} :=
\begin{cases} a,\quad\text{ if } a\geqslant b;\\
0,\quad\text{ if } a<b.
\end{cases}
\end{equation*}

\begin{definition}
A map $\phi\colon (\mathcal X,d_{\mathcal X})\to(\mathcal Y,d_{\mathcal Y})$ between metric spaces is \emph{coarsely lipschitz} if there exists constants $K\geq 1$ and $C\geq 0$ such that
\[d_{\mathcal Y}\bigl(f(x),f(y)\bigr)\leqslant Kd_{\mathcal X}(x,y)+C, \qquad \forall x,y\in \mathcal X.\]
In this case we call $\phi$ a $(K,C)$-coarsely lipschitz map.

The map $\phi$
is a \emph{quasi-isometric embedding} if there exist constants $K\geq 1$ and $C\geq 0$ so that
\[K^{-1}d_{\mathcal X}(x,y)-C\leq d_{\mathcal Y}(\phi(x),\phi(y))\leq Kd_{\mathcal X}(x,y) + C,\qquad \forall x,y\in\mathcal X.\]
In this case we call $\phi$ a $(K,C)$-quasi-isometric embedding. If, in addition, there exists a constant $M$ such that $\mathcal Y=\mathcal{N}_M(\phi(\mathcal X))$ then $\phi$ is a \emph{quasi-isometry}. If $X$ is a connected subset of $\mathbb{R}$ then we call $\phi$ a $(K,C)$-\emph{quasigeodesic}. 
\end{definition}

\begin{definition}[\bf{Quasigeodesic metric space}]
A metric space $(\mathcal{X},d_{\mathcal{X}})$ is $(K,C)$-\emph{quasigeodesic} if for every $x,y\in \mathcal{X}$ there exist a $(K,C)$-quasigeodesic $\gamma:[0,1]\to X$ such that $\gamma(0)=x$ and $\gamma(1)=y$. We call the metric space $K$-quasigeodesic if it is $(K,K)$-quasigeodesic, and we call the metric space \emph{quasigeodesic} if it is $(K,C)$-quasigeodesic for some $K\geq 1$ and $C\geq 0$. 
\end{definition}

Any geodesic metric space is a $(K,C)$-quasigeodesic metric space.

\subsection{Hierarchically hyperbolic spaces and groups}\label{subsection_definition_hhs}
We start this subsection with the definition of hierarchically hyperbolic spaces.

\begin{definition}[\bf{Hierarchically hyperbolic space}]\label{HHS_definition}
A $q$-quasigeodesic metric space $(\mathcal{X},d_\mathcal{X})$ is \emph{hierarchically hyperbolic} if there exist $\delta\geqslant 0$, an index set $\mathfrak{S}$, 
and a set $\{\mathcal{C}W\mid W\in\mathfrak{S}\}$ of $\delta$-hyperbolic spaces $(\mathcal{C}U,d_U)$, such that the following conditions are satisfied:
\begin{enumerate}
\item {\bf (Projections)}\label{axiom1} There is a set $\{\pi_W\colon\mathcal{X}\to 2^{\mathcal{C}W}\mid W\in\mathfrak{S}\}$ of projections that send points in $\mathcal{X}$ 
to sets of diameter bounded by some $\xi\geqslant 0$ in the hyperbolic spaces $\mathcal{C}W\in\mathfrak{S}$. Moreover, there exists $K$ so that all $W\in\mathfrak{S}$, the coarse map $\pi_{W}$ is $(K,K)$-coarsely lipschitz and $\pi_W(\mathcal{X})$\footnote{If $A\subseteq \mathcal{X}$, by $\pi_U(A)$ we mean $\bigcup_{a\in A}\pi_U(a)$.} is $K$-quasiconvex in $\mathcal{C}W$.
\item {\bf (Nesting)} The index set $\mathfrak{S}$ is equipped with a partial order $\sqsubseteq$ called \emph{nesting}, and either $\mathfrak{S}$ is empty or it contains a unique $\sqsubseteq$-maximal element. When $V\sqsubseteq W$, $V$ is nested into $W$. For each $W\in\mathfrak{S}$, $W\sqsubseteq W$, and with $\mathfrak{S}_W$
we denote the set of all $V\in\mathfrak{S}$ that are nested in $W$.
For all $V,W\in\mathfrak{S}$ such that $V\sqsubsetneq W$ there is a subset $\rho_W^V\subseteq \mathcal{C}W$ with diameter at most $\xi$, and a map 
$\rho_V^W\colon \mathcal{C}W\to 2^{\mathcal{C}V}$.
\item\label{A3} {\bf (Orthogonality)} The set $\mathfrak{S}$ has a symmetric and antireflexive relation $\perp$ called \emph{orthogonality}. Whenever $V\sqsubseteq W$ and 
$W\perp U$, then $V\perp U$ as well. For each $Z\in \mathfrak{S}$ and each $U\in\mathfrak{S}_Z$ for which $\{V\in\mathfrak{S}_Z\mid V\perp U\}\neq\emptyset$, 
there exists $\cont_\perp^ZU\in \mathfrak{S}_Z\setminus\{Z\}$ such that whenever $V\perp U$ and $V\sqsubseteq Z$, then $V\sqsubseteq \cont_\perp^ZU$.
\item {\bf (Transversality and Consistency)}\label{HHS_definition_4}
If $V,W\in\mathfrak{S}$ are not orthogonal and neither is nested into the other, then they are transverse: 
$V\pitchfork W$. There exists $\kappa_0\geqslant 0$ such that if $V\pitchfork W$, then there are sets $\rho_W^V\subseteq \mathcal{C}W$ and 
$\rho_V^W\subseteq \mathcal{C}V$, each of diameter at most $\xi$, satisfying
\begin{equation*}\label{consistent_transversal}
\min\bigl\{d_W(\pi_W(x),\rho_W^V),d_V(\pi_V(x),\rho_V^W)\bigr\}\leqslant \kappa_0,\qquad \forall\ x\in\mathcal{X}.
\end{equation*}
Moreover, for $V\sqsubseteq W$ and for all $x\in\mathcal{X}$ we have that
\begin{equation*}\label{consistent_nesting}
\min\bigl\{d_W(\pi_W(x),\rho_W^V),\diam_{\mathcal{C}V}(\pi_V(x)\cup\rho_V^W(\pi_W(x)))\bigr\}\leqslant \kappa_0.
\end{equation*}
In the case of $V\sqsubseteq W$, we have that $d_U(\rho^V_U,\rho^W_U)\leqslant \kappa_0$ whenever $U\in\mathfrak{S}$ is such that either 
$W\sqsubsetneq U$, or $W\pitchfork U$ and $U\not\perp V$.
\item {\bf (Finite complexity)} There is a natural number $n\geqslant0$, the complexity of $\mathcal{X}$ with respect to $\mathfrak{S}$, such that any set of 
pairwise $\sqsubseteq$-comparable elements of $\mathfrak{S}$ has cardinality at most $n$.
\item {\bf (Large links)} There exist $\lambda\geqslant 1$ and $E\geqslant \max\{\xi,\kappa_0\}$ such that, given any $W\in\mathfrak{S}$ and $x,x'\in\mathcal{X}$, 
there exists $\{T_i\}_{i=1,\dots,\lfloor N\rfloor}\subset \mathfrak{S}_W\setminus\{W\}$ such that for all $T\in\mathfrak{S}_W\setminus\{W\}$ either $T\in\mathfrak{S}_{T_i}$ for some $i$, or $d_T(\pi_T(x),\pi_T(x'))< E$, where $N=\lambda d_W(\pi_W(x),\pi_W(x'))+\lambda$. 
Moreover, $d_W(\pi_W(x),\rho_W^{T_i})\leqslant N$ for all $i$.
\item {\bf (Bounded geodesic image)} For all $W\in\mathfrak{S}$, all $V\in\mathfrak{S}_W \setminus \{W\}$ and all geodesics $\gamma$ of $\mathcal{C}W$, either $\diam_{\mathcal{C}V}(\rho_V^W(\gamma))\leqslant E$ or $\gamma\cap \mathcal{N}_E(\rho_W^V)\neq \emptyset$.
\item {\bf (Partial realization)} There is a constant $\alpha$ satisfying: let $\{V_j\}$ be a family of pairwise orthogonal elements of $\mathfrak{S}$, ad let
$p_j\in\pi_{V_j}(\mathcal{X})\subseteq \mathcal{C}V_j$. Then there exists $x\in\mathcal{X}$ such that
\begin{itemize}
\item $d_{V_j}\bigl(\pi_{V_j}(x),p_j\bigr)\leqslant \alpha$ for all $j$;
\item for each $j$ and each $V\in\mathfrak{S}$ with $V_j\sqsubseteq V$, we have $d_V\bigl(\pi_V(x),\rho_V^{V_j}\bigr)\leqslant \alpha$;
\item if $W\pitchfork V_j$ for some $j$, then $d_W\bigl(\pi_W(x),\rho_W^{V_j}\bigr)\leqslant \alpha$.
\end{itemize} 
\item {\bf (Uniqueness)} For each $\kappa\geqslant 0$ there exists $\theta_u=\theta_u(\kappa)$ such that if $x,y\in\mathcal{X}$ and $d(x,y)\geqslant \theta_u$, then there exists $V\in\mathfrak{S}$ such that $d_V(x,y)\geqslant\kappa$.
\end{enumerate}
The inequalities of the fourth axiom are called \emph{consistency inequalities}.
\end{definition}
Although most of the natural examples of hierarchically hyperbolic spaces are geodesic metric spaces, it is benefitial to work in the more general context of quasigeodesic metric spaces. This is because a quasiconvex subspace of a geodesic metric space might fail to be geodesic, but (hierarchically) quasiconvex subspaces of hierarchically hyperbolic spaces inherit naturally a hierarchically hyperbolic structure (compare Definition \ref{hierarchical_quasiconvexity} and \cite[Proposition 5.6]{BHS2}).

\begin{remark}\label{container_remark}
The element $\cont_\perp^ZU$ appearing in Axiom \eqref{A3} of Definition \ref{HHS_definition} is called the \emph{orthogonal container} 
(or the container of the orthogonal complement) of $U$ in $Z$.
If $Z$ is the $\sqsubseteq$-maximal element of $\mathfrak{S}$, then we might suppress it from the notation, write $\cont_\perp U$ and call it 
\emph{higher} container. If $Z$ is not the $\sqsubseteq$-maximal, then we will talk about \emph{lower} containers.

A hierarchically hyperbolic space has \emph{clean containers} if $U\perp \cont_\perp^Z U$ for all $U,Z\in\mathfrak{S}$, as originally defined in \cite[Definition 3.4]{ABD}.
\end{remark}

For a hierarchically hyperbolic space $(\mathcal{X},\mathfrak{S})$ and a subset $\mathfrak{U}\subseteq \mathfrak{S}$, we define 
\begin{equation}\label{equation_orthogonal_set}
\mathfrak{U}^{\perp}:=\{V\in\mathfrak{S}\mid V\perp U\text{ for every }U\in\mathfrak{U}\}.
\end{equation}

\begin{remark}\label{remark_coarsely_surjective}
By \cite[Remark 1.3]{BHS2}, the projections $\pi_U$ of a hierarchically hyperbolic space $(\mathcal{X},\mathfrak{S})$ can always be assumed to be uniformly coarsely surjective. Without loss of generality, we will always assume this.
\end{remark}

\begin{definition}[\bf Hierarchical quasiconvexity]\label{hierarchical_quasiconvexity}
Let $(\mathcal{X},\mathfrak{S})$ be a hierarchically hyperbolic space. A subspace $\mathcal{Y}\subseteq \mathcal{X}$ is $k$-\emph{hierarchically quasiconvex}, for 
some function $k\colon [0,+\infty)\to[0,+\infty)$, if:
\begin{enumerate}
\item for all $U\in\mathfrak{S}$ the image $\pi_U(\mathcal{Y})$ is a $k(0)$-quasiconvex subspace of the hyperbolic space $\mathcal{C}U$;
\item for all $\kappa\geqslant 0$, if $x\in\mathcal{X}$ is such that $d_U(\pi_U(x),\pi_U(\mathcal{Y}))\leqslant \kappa$ for all $U\in\mathfrak{S}$, then 
$d_{\mathcal{X}}(x,\mathcal{Y})\leqslant k(\kappa)$.
\end{enumerate}
\end{definition}
We now fix a notation for hyperbolic spaces. If $\mathcal{X}$ is a hyperbolic space and $\mathcal{Y}\subseteq\mathcal{X}$ a quasiconvex subspace, then we denote the closest-point projection by $p_{\mathcal{Y}}\colon \mathcal{X}\to \mathcal{Y}$. Whenever it is clear from the context to which quasiconvex subspace we are projecting, we might suppress it from the notation, and denote the closest-point projection by $p\colon \mathcal{X}\to \mathcal{Y}$.

\begin{remark}\label{remark:hierarchical_quasiconvexity}
It is extremely important to stress that, in \cite{BHS2}, a hieromorphism $\phi\colon (\mathcal{X},\mathfrak{S})\to(\mathcal{X}',\mathfrak{S}')$ is called \mbox{$k$-hierarchically} quasiconvex if $\phi(\mathcal{X})$ is a $k$-hierarchically quasiconvex subspace of $\mathcal{X}'$ - in the sense of Definition~\ref{hierarchical_quasiconvexity} - \emph{and} $\phi$ is already a quasi-isometric embedding (compare \cite[Definition 8.1]{BHS2}).

In this work, by $k$-hierarchically quasiconvex hieromorphism we just mean a hieromorphism whose image is a $k$-hierarchically quasiconvex
subspace. 

In practice, this will not produce diverging notions of hierarchical quasiconvexity: in this paper, whenever we consider a hierarchically quasiconvex hieromorphism $\phi$, this map $\phi$ is always also assumed to be coarsely lipschitz, and full. By what we will prove in Theorem \ref{thmB}, these hypotheses imply that $\phi$ is a quasi-isometric embedding. Therefore, a $k$-hierarchically quasiconvex hieromorphism in the sense of \cite{BHS2} is equivalent to a $k$-hierarchically quasiconvex full, coarsely lipschitz hieromorphism in the sense of this paper.

We elected to do this because, in previous ArXiv-versions of \cite{BHS2}, the assumption for $\phi$ to be a quasi-isometric embedding was not included in the notion of hierarchically quasiconvex hieromorphism, and because, doing so, we remember the reader that the hieromorphisms we consider are always assumed to be coarsely lipschitz (and equivalently quasi-isometric embeddings), something which is now hidden in \cite{BHS2}.
\end{remark}

\begin{remark}\label{def_gatemaps_nontree}
As for quasiconvexity in the hyperbolic setting, there exist coarse projections onto hierarchically quasiconvex subspaces. If $\mathcal{Y}\subseteq 
\mathcal{X}$ is a hierarchically quasiconvex subspace, then there exists a coarsely lipschitz map $\mathfrak{g}_{\mathcal{Y}}\colon \mathcal{X}\to 
\mathcal{Y}$, called \emph{gate map} \cite[Section 5]{BHS2}, with the following property: $\mathfrak{g}_{\mathcal{Y}}(x)\in\mathcal{Y}$ is such that for all 
$V\in\mathfrak{S}$ the set $\pi_V\bigl(\mathfrak{g}_{\mathcal{Y}}(x)\bigr)$ coarsely coincides (with uniform constants) with the projection of the element 
$\pi_V(x)\in\mathcal{C}V$ to the quasiconvex subspace $\pi_V(\mathcal{Y})$ of the hyperbolic space $\mathcal{C}V$.
\end{remark}

Important examples of hierarchically quasiconvex subspaces are \emph{standard product regions} \cite[Section 5]{BHS2}.
To define them, we need the notion of \emph{consistent tuple} \cite[Definition~1.16]{BHS2}.
\begin{definition}[{\bf $\kappa$-consistent tuple}]\label{product_region}
Fix $\kappa\geqslant 0$, and consider a tuple $\vec{b}=(b_U)_{U\in\mathfrak{S}}\in\prod_{U\in\mathfrak{S}}2^{\mathcal{C}U}$ such that for each coordinate 
$U\in\mathfrak{S}$ the coordinate $b_U$ is a subset of $\mathcal{C}U$ with diameter bounded by $\kappa$. The tuple $\vec{b}$ is $\kappa$-\emph{consistent} if 
whenever $V\pitchfork W$
\[\min\bigl\{d_W(b_W,\rho^V_W),d_V(b_V,\rho^W_V)\bigr\}\leqslant\kappa,\]
and whenever $V\sqsubseteq W$
\[\min\bigl\{d_W(b_W,\rho_W^V),\diam_{\mathcal{C}W}(b_V\cup \rho^W_V(b_W))\bigr\}\leqslant\kappa.\]
These inequalities generalize the consistency inequalities of the definition of hierarchically hyperbolic space.
\end{definition}
Let $(\mathcal{X},\mathfrak{S})$ be a hierarchically hyperbolic space. For a given $U\in\mathfrak{S}$, let
\[\mathfrak{S}_U\vcentcolon=\{V\in\mathfrak{S}\mid V\sqsubseteq U\}.\]
Given $\kappa\geqslant\kappa_0$, define ${\bf F}_U$ to be the set of $\kappa$-consistent tuples in $\prod_{V\in\mathfrak{S}_U}2^{\mathcal{C}V}$, and 
${\bf E}_U$ to be the set of $\kappa$-consistent tuples in $\prod_{V\in\mathfrak{S}_U^\perp\setminus\{A\}}2^{\mathcal{C}U}$, where
\[\mathfrak{S}_U^\perp=\{V\in\mathfrak{S}\mid V\sqsubseteq U\}\cup\{A\}\]
and $A$ is a $\sqsubseteq$-minimal element such that $V\sqsubseteq A$ for all $V\perp A$.

These sets ${\bf F}_U$ and ${\bf E}_U$ can be canonically identified as subspaces of $\mathcal{X}$. Indeed, by \cite[Construction 5.10]{BHS2} there are coarsely 
well-defined maps $\phi^\sqsubseteq\colon {\bf F}_U\to\mathcal{X}$ and $\phi^\perp\colon {\bf E}\to\mathcal{X}$ with hierarchically quasiconvex image, 
and by an abuse of notation we set that ${\bf F}_U=\im \phi^\sqsubseteq$ and ${\bf E}_U=\im\phi^\perp$.

Then, if ${\bf F}_U$ and ${\bf E}_U$ are endowed with the subspace metric, the spaces $({\bf F}_U,\mathfrak{S}_U)$ and $({\bf E}_U,\mathfrak{S}_U^\perp)$ are
hierarchically hyperbolic.
The maps $\phi^\sqsubseteq$ and $\phi^\perp$ extend to $\phi_U\colon {\bf F}_U\times {\bf E}_U\to \mathcal{X}$. Call ${\bf P}_U=\im\phi_U$ the \emph{standard product
region} in $\mathcal{X}$ associated to $U$ (compare \cite[Definition 5.14]{BHS2}). This space is coarsely equal to ${\bf F}_U\times {\bf E}_U$.
We direct the interested reader to \cite[Section 5]{BHS2} for more information on this.

\subsection{Morphisms between hierarchically hyperbolic spaces, and groups}
A hieromorphism is a morphism between hierarchically hyperbolic spaces that preserves the underlying structure. This statement is made precise by the following 
definition.

\begin{definition}[\bf Hieromorphism]\label{HHS_hieromorphism}
Let $(\mathcal{X},\mathfrak{S})$ and $(\mathcal{X}',\mathfrak{S}')$ be hierarchically hyperbolic spaces. A \emph{hieromorphism} is a triple 
$\phi=\bigl(\phi,\phi^\lozenge,\{\phi^*_U\}_{U\in\mathfrak{S}}\bigr)$, where $\phi\colon \mathcal{X}\to\mathcal{X}'$ is a map, 
$\phi^\lozenge\colon \mathfrak{S}\to\mathfrak{S}'$ is an injective map that preserves nesting, transversality and orthogonality,
and, for every $U\in\mathfrak{S}$, the maps $\phi^*_U\colon \mathcal{C}U\to\mathcal{C}\phi^\lozenge(U)$ are quasi-isometric embeddings with uniform constants. 

Moreover, the following two diagrams coarsely commute (again with uniform constants), for all non-orthogonal $U,V\in\mathfrak{S}$:
\begin{equation}\label{coarsely.commuting.diagrams}
\xymatrix{
\mathcal{X}\ar[r]^{\phi}\ar[d]_{\pi_U} & \mathcal{X}'\ar[d]^{\pi_{\phi^\lozenge(U)}}\\
\mathcal{C}U\ar[r]_{\phi^*_U}&\mathcal{C}\phi^\lozenge(U)
}\qquad\qquad\qquad 
\xymatrix{
\mathcal{C}U\ar[rr]^{\phi^*_U}\ar[d]_{\rho^U_V} && \mathcal{C}\phi^\lozenge(U)\ar[d]^{\rho^{\phi^\lozenge(U)}_{\phi^\lozenge(V)}}\\
\mathcal{C}V\ar[rr]_{\phi^*_V}&&\mathcal{C}\phi^\lozenge(V)
}
\end{equation}
\end{definition}
As hieromorphisms $\phi\colon (\mathcal{X},\mathfrak{S})\to(\mathcal{X}',\mathfrak{S}')$ between hierarchically hyperbolic spaces induce injective maps 
$\phi^\lozenge\colon \mathfrak{S}\to \mathfrak{S}'$ at the level of indexing sets, with a slight abuse of notation one can think of $\mathfrak{S}$ as a subset 
of $\mathfrak{S}'$.

\medskip
We will need the following strengthening of the notion of hieromorphism.
\begin{definition}\label{fullness_definition}
A hieromorphism $\phi\colon(\mathcal{X},\mathfrak{S})\to(\mathcal{X}',\mathfrak{S}')$ is \emph{full} if: 
\begin{enumerate} 
\item there exists $\xi$ such that the maps $\phi^*_U\colon\mathcal{C}U\to\mathcal{C}\phi^\lozenge(U)$ are $(\xi,\xi)$-quasi-isometries, for all $U\in\mathfrak{S}$;
\item if $S$ denotes the $\sqsubseteq$-maximal element of $\mathfrak{S}$, then for all $U'\in\mathfrak{S}'$ nested into $\phi^\lozenge(S)$ there exists 
$U\in\mathfrak{S}$ such that $U'=\phi^\lozenge(U)$.
\end{enumerate}
\end{definition}
Such hieromorphism is called full because its image coincides with (and it is not only contained in) the sublattice of $\mathfrak{S}'$ consisting of all $U'$ nested 
into $\phi^\lozenge(S)$.

Finally, we say that a hieromorphism $\phi\colon(\mathcal{X},\mathfrak{S})\to(\mathcal{X}',\mathfrak{S}')$ is $k$-\emph{hierarchically quasiconvex} 
if $\phi(\mathcal{X})$ is a $k$-hierarchically quasiconvex subspace of $\mathcal{X}'$, for some function $k\colon [0,+\infty)\to[0,+\infty)$.

\smallskip 
An \emph{automorphism} of a hierarchically hyperbolic space $(\mathcal{X},\mathfrak{S})$ is a hieromorphism $\phi\colon (\mathcal{X},\mathfrak{S})
\to(\mathcal{X},\mathfrak{S})$ such that $\phi^\lozenge$ is bijective and each $\phi^*_U$ is an isometry. 
The group of automorphisms of 
$(\mathcal{X},\mathfrak{S})$ is denoted by $\Aut(\mathfrak{S})$.

\begin{definition}[\bf{Hierarchically hyperbolic group}]\label{def:hhg}
A finitely generated group $G$ is \emph{hierarchically hyperbolic group} if there exists an action 
$G\to \Aut(\mathfrak{S})$ on a hierarchically hyperbolic space $(\mathcal{X},\mathfrak{S})$ such that the action of $G$ on $\mathcal{X}$ is
metrically proper, cobounded, and such that the induced action on $\mathfrak{S}$ is cofinite. 
This, in particular, means that the metric space $(G,d)$, where $d$ is the
word metric associated to any finite generating set of $G$, is hierarchically hyperbolic with respect to $\mathfrak{S}$.

For each $g\in G$, we denote its image in $\Aut(\mathfrak{S})$ by $\bigl(f_g,f_g^\lozenge,\{f^*_{g,U}\}_{U\in\mathfrak{S}}\bigr)$.
\end{definition}

Let $(G,\mathfrak{S})$ and $(G',\mathfrak{S}')$ be hierarchically hyperbolic groups. A hieromorphism $\phi\colon (G,\mathfrak{S})\to 
(G',\mathfrak{S}')$ is a \emph{homomorphism of hierarchically hyperbolic groups} if it is also a group-homomorphism $\phi\colon G\to G'$ that is $\phi$-equivariant, that is, for all $g\in G$ and $U\in\mathfrak{S}$ we have that 
\[\phi^\lozenge\bigl(f^\lozenge_g(U)\bigr)=f^\lozenge_{\phi(g)}(\phi^\lozenge(U))\]
and the following diagram uniformly coarsely commutes:
\[\xymatrix{
\mathcal{C}U\ar[rr]^{\phi^*_U}\ar[d]_{f_{g,U}^*} && \mathcal{C}\phi^\lozenge(U)\ar[d]^{f_{\phi(g),\phi^\lozenge(U)}^*}\\
\mathcal{C}f_g^\lozenge(U)\ar[rrd]_{\phi^*_{f_g^\lozenge(U)}}&&\mathcal{C}f_{\phi(g)}^\lozenge\bigl(\phi^\lozenge(U)\bigr)\ar@{}[d]|*=0[@]{=}\\
&&\mathcal{C}\phi^\lozenge\bigl(f^\lozenge_g(U)\bigr)
}\]

As a particular example, we now describe the hierarchically hyperbolic structure of a direct product of two hierarchically hyperbolic groups.
The hierarchical structure of the product of two hierarchically hyperbolic \emph{spaces} would be completely similar.

\begin{example}[{\bf Direct product of hierarchically hyperbolic groups}]\label{example_directproduct}
Let $(G_u,\mathfrak{S}_u)$ and $(G_w,\mathfrak{S}_w)$ be hierarchically hyperbolic groups. The direct product $G=G_u\times G_w$ is a hierarchically hyperbolic group \cite[Proposition~8.25]{BHS2}, and its hierarchical structure is described as follows.

The index set $\mathfrak{S}$ for $G$ is defined to be the disjoint union of $\mathfrak{S}_u$ with $\mathfrak{S}_w$, inheriting the associated hyperbolic spaces, along with the following elements whose associated hyperbolic spaces are defined to be points. For each $U\in\mathfrak{S}_u$ add an element $V_U$, into which every element of $\mathfrak{S}_u$ orthogonal to $U$, and every element of $\mathfrak{S}_w$, is nested. Analogously, for every $W\in\mathfrak{S}_w$ include an element $V_W$ into which every element of $\mathfrak{S}_w$ orthogonal to $W$, and every element of $\mathfrak{S}_u$, is nested. 
Finally, include a $\sqsubseteq$-maximal element $S$ into which each of the previous elements is nested. 

Nesting, orthogonality, and transversality agree with the ones of $(G_u,\mathfrak{S}_u)$ and $(G_w,\mathfrak{S}_w)$ on the subsets $\mathfrak{S}_u$ and $\mathfrak{S}_w$ of $\mathfrak{S}$, and any element of $\mathfrak{S}_u$ is orthogonal to any element of $\mathfrak{S}_w$. 
For any $A,B\in \mathfrak{S}_u\sqcup \mathfrak{S}_w$ we impose that
\begin{equation*}
\begin{cases}
A\sqsubsetneq V_B,&\quad\text{whenever }A\perp B;\\
A\perp V_B,&\quad\text{whenever }A\sqsubseteq B;\\
A\ \pitchfork\ V_B,&\quad\text{otherwise};
\end{cases}
\qquad\qquad\qquad
\begin{cases}
V_B\sqsubsetneq V_A,&\quad\text{whenever }A\sqsubsetneq B;\\
V_A\ \pitchfork\ V_B,&\quad\text{otherwise}.
\end{cases}
\end{equation*}
In particular, $A\perp V_A$ for any element $A\in\mathfrak{S}_u\sqcup \mathfrak{S}_w$.

Projections to the hyperbolic spaces are either defined to be trivial, for elements with trivial hyperbolic space, or defined as the compositions $\pi_U\circ p_u$  (respectively $\pi_W\circ p_w$) for every $U\in\mathfrak{S}_u$ (respectively for every $W\in\mathfrak{S}_w$), where $p_u\colon G\to G_u$ is the canonical projection on the first direct factor, and $\pi_U\colon G_u\to 2^{\mathcal{C}U}$ is the projection given in $(G_u,\mathfrak{S}_u)$. 

It follows that for every $U\in\mathfrak{S}_u$ the set $\pi_U(G_w)$ is uniformly bounded, and analogously for every $W\in\mathfrak{S}_w$ the set  $\pi_W(G_u)$ is uniformly bounded.
Moreover, the inclusions of the subgroups $G_u$ and $G_w$ into $G$ are full, hierarchically quasiconvex hieromorphisms that induce isometries at the level of hyperbolic spaces.

\end{example}

\subsection{Trees of hierarchically hyperbolic spaces}\label{section_trees}
\begin{definition}\label{equiv_class_def}
Let $T=(V,E)$ be a tree.
A \emph{tree of hierarchically hyperbolic spaces} is a quadruple $\mathcal{T}=\bigl(T,\{\mathcal{X}_v\}_{v\in V},\{\mathcal{X}_e\}_{e\in E},\{\phi_{e_\pm}\colon
\mathcal{X}_e\to\mathcal{X}_{e_\pm}\}_{ e\in E}\bigr)$ such that
\begin{enumerate}
\item $\{\mathcal{X}_v\}$ and $\{\mathcal{X}_e\}$ are families of uniformly hierarchically hyperbolic spaces with index sets $\{\mathfrak{S}_v\}$ and $\{\mathfrak{S}_e\}$ respectively;
\item all $\phi_{e_+}\colon (\mathcal{X}_e,\mathfrak{S}_e)\to (\mathcal{X}_{e_+},\mathfrak{S}_{e_+})$ and $\phi_{e_-}\colon (\mathcal{X}_e,\mathfrak{S}_e)\to (\mathcal{X}_{e_-},\mathfrak{S}_{e_-})$ are hieromorphisms with all constants bounded 
uniformly by some $\xi\geqslant 0$. 
\end{enumerate}
\end{definition}
To a tree of hierarchically hyperbolic spaces $\mathcal{T}$ we can associate the metric space
$\mathcal{X}(\mathcal{T})\vcentcolon=\bigsqcup_{v\in V}(\mathcal{X}_v,d)$ in the following way. If $x\in\mathcal{X}_e$, then add an edge between $\phi_{e_-}(x)$ and 
$\phi_{e_+}(x)$. Given $x,x'\in \mathcal{X}$ in the same vertex space $\mathcal{X}_v$, then define $d'(x,x')$ to be $d_{\mathcal{X}_v}(x,x')$. Given 
$x,x'\in \mathcal{X}$ joined by an edge, define $d'(x,x')=1$. If $x_0,x_1,\dots,x_m\in\mathcal{X}$ is a sequence with consecutive points either joined by 
an edge or in a common vertex space, then define 
\[d'(x_0,x_m)=\sum_{i=1}^{m}d'(x_{i-1},x_{i}).\]
Finally, given $x,x'\in\mathcal{X}$, define
\[d(x,x')=\inf\{d'(x,x')\mid \ x=x_0,\dots,x_m=x'\quad \text{a sequence}\}.\]
Following \cite[Section 8]{BHS2}, for each edge $e$ and each $W_{e_-}\in\mathfrak{S}_{e_-}$ and
$W_{e_+}\in\mathfrak{S}_{e_+}$, we write $W_{e_-}\sim_d W_{e_+}$ if there exists $W_e\in\mathfrak{S}_e$ such that $\phi_{e_-}^\lozenge(W_e)=W_{e_-}$ and 
$\phi_{e_+}^\lozenge(W_e)=W_{e_+}$. Then, the transitive closure of $\sim_d$ defines a equivalence relation in $\bigsqcup_{v}\mathfrak{S}_v$, denoted by $\sim$.

The \emph{support} of an $\sim$-equivalence class $[V]$ is
\begin{equation*}
T_{[V]}:=\bigl\{ v\in T\mid \text{there exists } V_v\in\mathfrak{S}_v \text{ such that }[V]=[V_v]\bigr\}.
\end{equation*}
By definition of the equivalence $\sim$, supports are trees.


\begin{definition}[{\bf Gate maps in trees of hierarchically hyperbolic spaces}]
Let $\mathcal{T}$ be a tree of hierarchically hyperbolic spaces and assume that the image of the hieromorphism $\phi_v:(\mathcal{X}_e,\mathfrak{S}_e)\to(\mathcal{X}_{v},\mathfrak{S}_v)$ is hierarchically quasiconvex (recall Definition \ref{hierarchical_quasiconvexity}) for every $e\in E$ and $v\in V$ connected to $e$.
The \emph{gate maps} $\mathfrak{g}_v\colon\mathcal{X}\to\mathcal{X}_v$ is defined as follows. 
Let $x\in\mathcal{X}$ be an arbitrary element. If $x\in\mathcal{X}_v$, then define $\mathfrak{g}_v(x):=x$. If $x\notin 
\mathcal{X}_v$, then we define $\mathfrak{g}_v(x)$ inductively. Let $w$ be the vertex such that $x\in \mathcal{X}_w$, suppose that 
$d_T(v,w)=n\geq1$, and that $\mathfrak{g}_v(-)$ is defined on all vertex spaces that are at distance strictly less than $n$ from $v$. 
Let $\gamma$ be the geodesic in $T$ connecting $w$ to $v$, let $e$ be its first edge, with $e^-=v$. It follows that $d_T(e^+,v)=n-1$.
Then
\[\mathfrak{g}_v(x)\vcentcolon=   \mathfrak{g}_v\Bigl( \phi_{e^+}\circ\bar{\phi}_{e^-}\bigl(\mathfrak{g}_{\phi_{e^-}(\mathcal{X}_e)}(x)\bigr)\Bigr),\]
where $\bar{\phi}_{e^-}\colon \mathcal{X}_{e^-}\to\mathcal{X}_e$ is a quasi-inverse of $\phi_{e^-}\colon \mathcal{X}_e\to\mathcal{X}_{e^-}$.
\end{definition}

\begin{definition}[{\bf Comparison maps}]\label{comparison_maps}
Let $\mathcal{T}$ be a tree of hierarchically hyperbolic groups, $[V]$ be an equivalence class, and let $u\neq v$ be two vertices in the 
support of $[V]$.
The \emph{comparison map} $\mathfrak{c}\colon \mathcal{C}V_u\to \mathcal{C}V_v$ between the hyperbolic spaces associated to the representatives
$V_u$ and $V_v$ of the class $[V]$ is defined as follows.

Assume first that $u$ and $v$ are vertices connected by a single edge $e$ such that $u=e^-$ and $v=e^+$. Then, the comparison map is defined as
\[\mathfrak{c}\vcentcolon=\phi_{e^+}^{\ast}\circ\overline{\phi^{\ast}_{e^-}}:\mathcal{C}V_u\to\mathcal{C}V_v.\] 
Where the maps $\phi^{\ast}_{e^+}:\mathcal{C}V_e\to\mathcal{C}V_{e^+}$ and $\phi^{\ast}_{e^-}:\mathcal{C}V_e\to\mathcal{C}V_{e^-}$ are the quasi-isometries induced by the hieromorphisms $\phi_{e^+}:\mathcal{X}_e\to\mathcal{X}_{e^+}$ and $\phi_{e^-}:\mathcal{X}_e\to\mathcal{X}_{e^-}$ respectively and $\overline{\phi^{\ast}_{e^-}}$ denotes a quasi inverse of $\phi^{\ast}_{e^-}$.

For the general case, let $\gamma$ be the geodesic in $T$ connecting $u$ to $v$, let $u_i$ be the $i$-th vertex of this geodesic (so that $u=u_0$ and $v=u_n$ for some natural number $n>0$), and let $e^i$ be the edge connecting $u_{i-1}$ to $u_{i}$.
For all $i=1,\dots,n$ consider the hieromorphisms $\phi_{e_i^-}\colon\mathcal{X}_{e_i}\to\mathcal{X}_{u_{i-1}}$ and 
$\phi_{e_i^+}\colon\mathcal{X}_{e_i}\to\mathcal{X}_{u_{i}}$, and the induced quasi-isometries $\phi_{e_i^-}^*\colon 
\mathcal{C}V_{e_i}\to\mathcal{C}V_{u_{i-1}}$ and $\phi_{e_i^+}^*\colon \mathcal{C}V_{e_i}\to\mathcal{C}V_{u_{i}}$ from the hyperbolic space
associated to the representative of $[V]$ in $\mathfrak{S}_{e_i}$ to the hyperbolic spaces associated to $V_{u_{i-1} }$ and $V_{u_i} $ 
respectively.
Finally, let $\overline{\phi_{e_i^-}^*}\colon \mathcal{C}V_{u_{i-1}}\to \mathcal{C}V_{e_i}$ be a quasi-inverse of the map $\phi_{e_i^-}^*$, for
all $i$.

Then, the comparison map $\mathfrak{c}$ is defined to be the composition of the previous quasi isometries:
\begin{equation}\label{definition_comparison}
\mathfrak{c}\vcentcolon=\phi_{e_n^-}^*\circ\overline{\phi_{e_n^-}^*}\dots \circ\phi_{e_1^-}^*\circ\overline{\phi_{e_1^-}^*}\colon \mathcal{C}V_{u_0}
\to\mathcal{C}V_{u_n}.
\end{equation}
\end{definition}

\begin{remark}\label{remark_comparison_maps}It is a fact \cite[Lemma 8.18]{BHS2} that if the cardinality of supports is uniformly bounded, then comparison maps are 
$(\xi,\xi)$-quasi-isometries, for some uniform (not depending on the two vertices $u$ and $v$) constant $\xi\geqslant1$.
\end{remark}

\begin{remark}\label{we_use_this_graphproduct}
If the edge hieromorphisms $\{\phi_{e^{\pm}}\}_{e\in E}$ of the tree of hierarchically hyperbolic spaces $\mathcal{T}$ induce isometries at the level of hyperbolic spaces, then we can choose inverse isometries for the maps $\phi_{e^{\pm}}^*$. Therefore, from Equation~\eqref{definition_comparison} it follows that comparison maps in this particular case are isometries.
\end{remark}

\smallskip 
We record now the following lemma, which is implicitly used in \cite{BHS2}. Its proof follows by applying repeatedly the (coarsely commutative) 
second diagram of Equation \eqref{coarsely.commuting.diagrams}.
\begin{lemma}\label{lemma.good.definition}
Let $\mathcal{T}$ be a tree of hierarchically hyperbolic spaces, and let $[U], [V]$ be two equivalence classes such that either 
$[U]\pitchfork [V]$ or $[U]\sqsubseteq [V]$.
If comparison maps are uniform quasi isometries, then for all vertices $u,v\in T_{[U]}\cap T_{[V]}$ the set $\mathfrak{c}(\rho^{U_u}_{V_u})$ is coarsely
equal to $\rho^{U_v}_{V_v}$, where $\mathfrak{c}  \colon\mathcal{C}V_u\to\mathcal{C}V_v$ is the comparison map.
\end{lemma}

\begin{definition}[\bf{Graph of hierarchically hyperbolic groups}]
Let $\Gamma=(V,E)$ be a finite simplicial graph.
A \emph{graph of hierarchically hyperbolic groups} is given by the quadruple $\bigl(\Gamma,\{G_v\}_{v\in V},\{G_e\}_{e\in E},
\{\phi_{e^\pm}\colon G_e\to G_{e^\pm}\}_{e\in E}\bigr)$, where vertex and edge groups are hierarchically hyperbolic groups, and the
$\phi_{e^\pm}$ are homomorphisms of hierarchically hyperbolic groups.

Let $F_E$ be the free group freely generated by the set $E$. 
The group $G$ associated to the graph of hierarchically hyperbolic groups is the quotient of $(\ast_{v\in V}G_v)\ast F_E$ obtained by
adding the relations
\begin{itemize}
\item $e=_G1$, if $e\in E$ belongs to a fixed spanning tree of $\Gamma$;
\item $\phi_{e^+}(g)=e\phi_{e^-}(g)e^{-1}$, for all $e\in E$ and $g\in G_e$.
\end{itemize}
\end{definition}
As described in \cite[Section 8.2]{BHS2}, a graph of hierarchically hyperbolic groups acts on a tree of hierarchically hyperbolic spaces whose
associated tree is the Bass-Serre tree $\tilde\Gamma$ of the graph $\Gamma$.


\section{Intersection property and concrete hierarchically hyperbolic spaces}\label{section3}

We now introduce a notion that will play a pivotal role in the proof of Theorem \ref{mainT}.

\begin{definition}[\bf Intersection property]\label{intersectionproperty_definition}
A hierarchically hyperbolic space $(\mathcal{X},\mathfrak{S})$ has the \emph{intersection property} if 
the index set admits an operation $\wedge\colon (\mathfrak{S}\cup\{\emptyset\})\times(\mathfrak{S}\cup\{\emptyset\})\to\mathfrak{S}\cup\{\emptyset\}$ satisfying
the following properties for all $U,V,W\in\mathfrak{S}$:
\begin{itemize}
\item[$(\wedge_1)$] $V\wedge \emptyset=\emptyset \wedge V=\emptyset$;
\item[$(\wedge_2)$] $U\wedge V=V\wedge U$;
\item[$(\wedge_3)$] $(U\wedge V)\wedge W=U\wedge (V\wedge W)$;
\item[$(\wedge_4)$] $U\wedge V\sqsubseteq U$ and $U\wedge V\sqsubseteq V$ whenever $U\wedge V\in\mathfrak{S}$;
\item[$(\wedge_5)$] if $W\sqsubseteq U$ and $W\sqsubseteq V$, then $W\sqsubseteq U\wedge V$.
\end{itemize}
\end{definition}
We call $U\wedge V$ the \emph{wedge} between $U$ and $V$.
Notice that $U\wedge V\in \mathfrak{S}_U\cap\mathfrak{S}_V$ as soon as $U\wedge V\neq \emptyset$, by property $(\wedge_4)$. Therefore, whenever $U\perp V$ it follows that $U\wedge V=\emptyset$, as 
the intersection $\mathfrak{S}_U\cap\mathfrak{S}_V$ is empty.
Moreover, it follows that $U\wedge V=V$ if and only if $V\sqsubseteq U$, and that
for all $U,V\in \mathfrak{S}$ the set $\mathfrak{S}_U\cap\mathfrak{S}_V$ either is empty or has a unique maximal element $U\wedge V$.

\smallskip
Hyperbolic groups satisfy the intersection property, since the index set consists of one element. 
Mapping class groups, raags, and the cubulable groups known to be hierarchically hyperbolic also satisfy the intersection property. In these cases, the operation $\wedge$ corresponds respectively to considering (the curve complex associated to) the intersection of two subsurfaces, the intersection of two parabolic subgroups, and the coarse projection (using gate maps \cite{HaSu}) of one hyperplane onto another.

Let $(\mathcal{X},\mathfrak{S})$ be a hierarchically hyperbolic space with the intersection property, let $U,V\in\mathfrak{S}$, and define
\begin{equation}\label{equation_vee}
U\vee V\vcentcolon=\bigwedge \bigl\{ W\in\mathfrak{S}\mid U\sqsubseteq W,\ V\sqsubseteq W\bigr\}.
\end{equation}
We call $U\vee V$ the \emph{join} between $U$ and $V$. The operations $\wedge$ and $\vee$ give to the set $\mathfrak{S}$ a lattice structure.

Notice that the set $\mathcal{W}=\{W\in\mathfrak{S}\mid U\sqsubseteq W,\ V\sqsubseteq W \}$ appearing in Equation \eqref{equation_vee} is never 
empty, because at least the $\sqsubseteq$-maximal element of $\mathfrak{S}$ belongs to it. Even if $\mathcal{W}$ is infinite,
finite complexity of the hierarchically hyperbolic space implies that there exists a natural number $n$, not greater than the complexity of the hierarchically
hyperbolic space, such that 
$U\vee V=W_1\wedge\dots\wedge W_n$, where $W_i\in\mathcal{W}$ for all $i$. Indeed, if this were not the case, one could find elements $W_i\in\mathcal{W}$ for $i=1,\dots,r$, where $r$ is strictly bigger than the finite-complexity constant, such that
\begin{equation*}
W_1\sqsupset W_1\wedge W_2\sqsupset \dots\sqsupset W_1\wedge\dots\wedge W_r\neq \emptyset,
\end{equation*}
contradicting the fifth axiom of the definition of hierarchically hyperbolic space.
By definition, $U\vee V$ is the $\sqsubseteq$-minimal element of $\mathfrak{S}$ in which both $U$ and $V$ are nested. 

In raags, the join of two parabolic subgroups is the subgroup they generate, and in mapping class groups the join of two subsurfaces is their union (which might be disconnected).

\smallskip
In the following lemma we prove that direct product of hierarchically hyperbolic spaces/groups with the intersection property continues to satisfy the intersection property. As a consequence of Theorem \ref{mainGraphProducts}, the intersection property is preserved also by graph products, and in particular by free products, when in presence of \emph{clean containers}.

The intersection property for free products of hierarchically hyperbolic groups is preserved also \emph{without} assuming clean containers, by deducing it from \cite[Theorem 8.6]{BHS2}, but we elected not to write down the details, as clean containers is such a natural hypothesis to make.

\begin{lemma}\label{directproduct_intersection}
The intersection property is preserved by direct products. If a group is hyperbolic relative to a finite collection of hierarchically hyperbolic spaces (respectively: groups) with the intersection property, then it is a hierarchically hyperbolic space (respectively: group) with the intersection property.
\begin{proof}
Given two hierarchically hyperbolic spaces $(\mathcal{X}_1,\mathfrak{S}_1)$ and $(\mathcal{X}_2,\mathfrak{S}_2)$ with the intersection property, we endow the space $\mathcal{X}_1\times \mathcal{X}_2$ with the hierarchically hyperbolic structure $\mathfrak{S}$ described in Example \ref{example_directproduct} (for hierarchically hyperbolic groups). 

Let $\wedge_1$ and $\wedge_2$ be the wedge maps on $(\mathcal{X}_1,\mathfrak{S}_1)$ and $(\mathcal{X}_2,\mathfrak{S}_2)$, respectively, and let us define $\wedge\colon (\mathfrak{S}\cup\{\emptyset\})\times (\mathfrak{S}\cup\{\emptyset\})\to \mathfrak{S}\cup \{\emptyset\}$.
If $U\in\mathfrak{S}_1, W\in\mathfrak{S}_2$ then $U\perp W$ and therefore $U\wedge W=\emptyset$. On the other hand, $\wedge$ coincides with $\wedge_1$ or $\wedge_2$ if both arguments belong to $\mathfrak{S}_1$ or $\mathfrak{S}_2$  respectively. 
If $W\in \mathfrak{S}_1\cup\mathfrak{S}_2$ and $V_U$, for $U\in \mathfrak{S}_1\cup\mathfrak{S}_2$, is an element with trivial associated hyperbolic space, as described in Example \ref{example_directproduct}, then we have the following exhaustive disjoint cases: either $W\perp U$, or $W$ and $U$ are $\sqsubseteq$-related, or $W\pitchfork U$. In the first case $W\sqsubseteq V_U$, and therefore $W\wedge V_U=W$. In the other two cases, it must be that $U$ and $W$ belong to the same index factor, say $\mathfrak{S}_1$. Therefore, $W\wedge V_{U}=W\wedge_1 \cont_\perp U$, where $\cont_\perp U$ is the orthogonal container of $U$ in $\mathfrak{S}_1$.
Finally, if $S$ is the $\sqsubseteq$-maximal element then $S\wedge U=U$ for every $U\in\mathfrak{S}_1\cup\mathfrak{S}_2$.

\medskip
To conclude, we now prove the statement for groups hyperbolic relative to hierarchically hyperbolic groups. The same argument works if the parabolic subgroups $\{H_1,\dots,H_n\}$ are assumed to be hierarchically hyperbolic \emph{spaces}, with the difference that the resulting group would be a hierarchically hyperbolic \emph{space}.

Let $G$ be a group hyperbolic relative to a finite collection of subgroups $\{H_1,\ldots,H_n\}$ that are hierarchically hyperbolic groups with the intersection property. Let $\mathfrak{S}_{H_i}$ be the hierarchically hyperbolic structure on $H_i$, and let $\wedge_{H_i}$ be the wedge operation of $\mathfrak{S}_{H_i}$. Any coset $gH_i$ admits a hierarchically hyperbolic structure $\mathfrak{S}_{gH_i}$ with wedge operation $\wedge_{gH_i}$ (compare \cite[Theorem 9.1]{BHS2}).

By \cite[Theorem 9.1]{BHS2} the group $G$ is a hierarchically hyperbolic group with index set $\mathfrak{S}=\{\widehat{G}\}\cup\bigsqcup_{gH_i\in GH_i}\mathfrak{S}_{gH_i}$, where $\widehat{G}$ is obtained from $G$ by coning off all left cosets of all the subgroups $H_i$. By \cite[Theorem 9.1]{BHS2} the element $\widehat{G}$ is the $\sqsubseteq$-maximal element, for all $U\in \mathfrak{S}_{gH_i}$ and $V\in \mathfrak{S}_{g'H_j}$ with $gH_i\neq g'H_j$ we have that $U\pitchfork V$, and finally if $U,V\in \mathfrak{S}_{gH_i}\subseteq \mathfrak{S}$ then the elements $U$ and $V$ are transversal (respectively orthogonal, $\sqsubseteq$-related) if and only if they are transversal (respectively orthogonal, $\sqsubseteq$-related) in $\mathfrak{S}_{gH_i}$. 

If $U,V\in\mathfrak{S}_{gH_i}\subseteq\mathfrak{S}$, then define $U\wedge V$ to be $U\wedge_{gH_i}V$. If $U,V$ belong to different cosets and in particular they are orthogonal, define $U\wedge V=\emptyset$. Finally, for every $U\in\mathfrak{S}$ define $U\wedge\widehat{G}=U$.

Thus, $G$ admits a hierarchically hyperbolic group structure with the intersection property.

\end{proof}
\end{lemma}

\begin{lemma}\label{lemma:intersection_wedge_commute}
Let $\phi\colon \bigl(\mathcal{X},\mathfrak{S}\bigr)\to \bigl(\mathcal{X}',\mathfrak{S}'\bigr)$ be a full hieromorphism between hierarchically hyperbolic
spaces with the intersection property, and let $U,V\in \mathfrak{S}$. Then
\[\phi^\lozenge\bigl(U\wedge V\bigr)=\phi^\lozenge(U)\wedge \phi^\lozenge(V),\qquad\qquad\phi^\lozenge\bigl(U\vee V\bigr)=\phi^\lozenge(U)\vee\phi^\lozenge(V).\]
\begin{proof}
We prove the lemma for the wedge $U\wedge V$. The proof for $U\vee V$ follows the same strategy.
Let $U\wedge V=A$, and $\phi^\lozenge(A)=A'\in\mathfrak{S}'$. We need to show that $\phi^\lozenge(U)\wedge \phi^\lozenge(V)=A'$. As $\phi^\lozenge$ preserves
nesting, we have that $A'\sqsubseteq \phi^\lozenge (U)\wedge \phi^\lozenge (V)$. As $\phi$ is full and $\phi^\lozenge(U)\wedge\phi^\lozenge(V)$
is nested into both $\phi^\lozenge (U)$ and $\phi^\lozenge (U)$, 
there exists $B\in \mathfrak{S}$ such that $\phi^\lozenge(B)=\phi^\lozenge(U)\wedge\phi^\lozenge(V)$ and $B$ is nested into both $U$ and $V$.

By maximality of $U\wedge V$, we conclude that $B=U\wedge V$, and it follows that
\[\phi^\lozenge\bigl(U\wedge V\bigr)=\phi^\lozenge(B)=\phi^\lozenge(U)\wedge \phi^\lozenge(V).\]
\end{proof}
\end{lemma}

The next lemma is an example of why clean containers is a very natural property, and should be assumed without any hesitation. In the mapping class group
setting the lemma just proves that if two subsurfaces $U$ and $V$ are disjoint from $W$, then $W$ is also disjoint from the subsurface $U\cup V$.
\begin{lemma}\label{clean_containers_use}
Let $\bigl(\mathcal{X},\mathfrak{S}\bigr)$ be a hierarchically hyperbolic space with the intersection property and clean containers. 
If $U\perp W$ and $V\perp W$, then $(U\vee V)\perp W$.
\begin{proof}
Both the elements $U$ and $V$ are nested into the orthogonal container $\cont_\perp W$, and by definition of join, it follows that
$U\vee V\sqsubseteq \cont_\perp W$ as well. By clean containers we have that $W\perp \cont_\perp W$, and therefore $(U\vee V)\perp W$.

Notice that we need the clean containers hypothesis for the case $U\vee V=\cont_\perp W$.
\end{proof}
\end{lemma}

\begin{lemma}\label{wedge_container_lemma_blue}
Let $(\mathcal{X},\mathfrak{S})$ be a hierarchically hyperbolic space with the intersection property and clean containers. For all $U,V\in\mathfrak{S}$ we have that $\cont_{\perp}^{U}V=U\wedge \cont_{\perp}V$. 
\begin{proof}
If $\cont_\perp V=\emptyset$, then also $\cont_\perp^U V$ is empty, and the equality is trivially satisfied.

If $\cont_\perp V$ is not empty, but $\cont_\perp^U V=\emptyset$, then there does not exist an element nested into both $U$ and into $\cont_\perp V$. Indeed, assume that there esists $W\in\mathfrak{S}$ such that $W\sqsubseteq U$ and $W\sqsubseteq\cont_\perp V$. Then, 
$W\sqsubseteq\cont_\perp^U V$ by definition of orthogonal containers, contradicting the assumption that $\cont_\perp^U V$ is empty. Therefore, also in this case the equality is trivially satisfied.

Suppose now that both $\cont_\perp V$ and $\cont_\perp^U V$ are non-empty.
By definition, we have that $\cont_{\perp}^{U}V\sqsubseteq U$. By clean containers $V\perp \cont_{\perp}^{U}V$, and thus $\cont_{\perp}^UV\sqsubseteq \cont_{\perp}V$. Therefore, $\cont_{\perp}^UV\sqsubseteq U\wedge\cont_{\perp}V$. On the other hand, as $V\perp \cont_{\perp}V$ and $U\wedge \cont_{\perp}V\sqsubseteq U$, we conclude that $U\wedge\cont_{\perp}V\sqsubseteq \cont_{\perp}^{U}V$.

\end{proof}
\end{lemma}

\begin{definition}[{\bf $\varepsilon$-support}]\label{supports}
For $A\subseteq \mathcal{X}$ and a constant $\varepsilon>0$, define the $\varepsilon$-\emph{support} to be
\[\supp_\varepsilon(A)\vcentcolon=\bigl\{ W\in \mathfrak{S}\mid\diam_{\mathcal{C}W} (\pi_W (A)) >\varepsilon\bigr\}.\]
Notice that if $\supp_\varepsilon(A)=\emptyset$, then $A\subseteq\mathcal{X}$ has uniformly bounded diameter: indeed, by the Uniqueness Axiom of 
Definition \ref{HHS_definition} it follows that $\diam_{\mathcal{X}}(A)\leqslant \theta_u(\varepsilon)$.
\end{definition}

In the following lemma, we make use of a relevant feature of a given standard product region $\textbf{P}_U$ associated to a given $U\in\mathfrak{S}$ as defined in Definition \ref{product_region}. For each $e\in\textbf{E}_U$ we denote $\textbf{F}_U\times\{e\}$ a \emph{parallel copy} of $\textbf{F}_U$ in $\mathcal{X}$. By construction of $\textbf{P}_U$ there exists a constant $\alpha$ which depends only on $\mathcal{X}$ and $\mathfrak{S}$, such that for every $x\in\textbf{P}_U$ we have that $d_V(\pi_V(x),\rho_V^U)\leq\alpha$ for all $U\in\mathfrak{S}$ satisfying either $U\pitchfork V$ or $U\sqsubseteq V$. Moreover, we can choose $\alpha$ so that, if $V\perp U$, then $\diam_{\mathcal{C}V}(\pi_V(\textbf{F}_U\times\{e\}))\leq\alpha$ (see \cite[Definition 1.15]{asdim} and \cite[Section 5]{BHS2} for more information). 

We recall that $\xi$ is the constant that uniformly bounds the sets $\rho_V^U$ for $U,V\in\mathfrak{S}$ such that $U\pitchfork V$ or $U\sqsubseteq V$.

\begin{lemma}\label{lemma_inclusion1}
Let $\varepsilon> 3\max\{\xi,\alpha\}$.
If $W\in \supp_\varepsilon({\bf F}_U\times\lbrace e\rbrace)$ then $W\sqsubseteq U$, and therefore $\supp_\varepsilon({\bf F}_U\times\lbrace e\rbrace)\subseteq \mathfrak{S}_U$.
\begin{proof}
If $U$ is either transverse to $V\in \mathfrak{S}$ or properly nested into $V$, then $d_V(\pi_V( x),\rho_V^U)\leq\alpha$ for every $x\in{\bf F}_U\times\lbrace e\rbrace$. As the diameter of the set $\rho_V^U$ is at most $\xi$, we obtain that
\begin{equation*}
\begin{aligned}
d_V(\pi_V(x),\pi_V(y))\leq &d_V(\pi_V(x),p(\pi_V(x)))+d_V(p(\pi_V(x)),p(\pi_V(y)))+\\
&+d_V(p(\pi_V(y)),\pi_V(y)) \leq 2\alpha +\xi<\varepsilon,
\end{aligned}
\end{equation*}
for every $x,y\in {\bf{F}}_U\times\{e\}$, where $p\colon\mathcal{C}V\to\rho_V^U$ denotes the closest point projection.
Therefore, we conclude that $V\notin\supp_\varepsilon ({\bf F}_U\times\lbrace e\rbrace)$.
On the other hand, whenever $U\perp V$ we have that $\pi_V( {\bf F}_U\times\lbrace e\rbrace)$ is a set of diameter bounded by $\alpha$, and again 
$V\notin 
\supp_\varepsilon ({\bf F}_U\times\lbrace e\rbrace)$.

Therefore, by the choice of $\varepsilon$, we have that $\supp_\varepsilon( {\bf F}_U\times\lbrace e\rbrace)\subseteq \mathfrak{S}_U$.
\end{proof}
\end{lemma}

\begin{convention}
From now on, even if not explicitly stated, we assume that $\varepsilon>3\max\{\xi,\alpha\}$.
\end{convention}

\begin{remark}
For an element $U\in\mathfrak{S}$, the set $\supp_\varepsilon({\bf F}_U\times\lbrace e\rbrace)$ defined in Definition \ref{supports} is independent of the parallel copy of ${\bf F}_U\times\lbrace e\rbrace$ that we consider, that is
\[\supp_\varepsilon({\bf F}_U\times \{e\})=\supp_\varepsilon({\bf F}_U\times \{e'\})\]
for any two elements $e,e'\in {\bf E}_U$.
Indeed, $\pi_W\bigl({\bf F}_U\times\{e\}\bigr)$ uniformly coarsely coincides with $\rho_W^U$ when either $W\sqsupseteq U$ or $W\pitchfork U$, or its diameter is bounded by $\alpha$ if $W\perp U$. Therefore, for $\varepsilon> 3\max\{\xi,\alpha\}$, it follows that $W\in \supp_\varepsilon({\bf F}_U\times \{e\})$ if and only if 
$W\in \supp_\varepsilon({\bf F}_U\times \{e'\})$.
\end{remark}

\noindent
\textbf{Notation.} For every $\varepsilon>3\max\{\xi,\alpha\}$ we denote by $\supp_\varepsilon({\bf F}_U)$ the set $\supp_{\varepsilon}({\bf F}_U\times\{e\})$ for any $e\in\textbf{E}_U$.

\begin{lemma}\label{hiero_supports}
Let $\phi:(\mathcal{X},\mathfrak{S})\to(\mathcal{X}',\mathfrak{S}')$ be a full hieromorphism and let $\varepsilon>0$. There exists $\varepsilon_0>0$ such that for every $\varepsilon'\geq \varepsilon_0$ 
\[\phi^{\lozenge}\bigl(\supp_{\varepsilon'}(\mathcal{X})\bigr)\subseteq\supp_\varepsilon\bigl(\phi(\mathcal{X})\bigr).\]
\end{lemma}

\begin{proof}
The hieromorphism $\phi$ is full, and the maps $\phi^*_U\circ\pi_U$ uniformly coarsely coincides with $\pi_{U'}\circ\phi$ for all $U\in \mathfrak{S}$ (here $U'$ denotes $\phi^{\lozenge}(U)$). Therefore, there exists $K>0$ such that
for all $x,y\in\mathcal{X}$, for all $U\in\mathfrak{S}$
\begin{equation}\label{quasiisom}
K^{-1}d_U(\pi_U(x),\pi_U(y))-K\leq d_{U'}\bigl(\pi_{U'}(\phi(x)),\pi_{U'}(\phi(y))\bigr).
\end{equation}
Let $\varepsilon_0:=K\varepsilon +K^2$. For $\varepsilon'\geqslant \varepsilon_0$, consider $W\in\supp_{\varepsilon'}(\mathcal{X})$: we prove that $\phi^\lozenge(W)\in \supp_{\varepsilon}\bigl(\phi(\mathcal{X})\bigr)$.
Indeed, let $x,y\in\mathcal{X}$ be such that $d_W\bigl(\pi_W(x),\pi_W(y)\bigr)>\varepsilon'$. By Equation \eqref{quasiisom} and
the definition of $\varepsilon_0$ we have that
\[d_{W'}\bigl(\pi_{W'}(\phi(x)),\pi_{W'}(\phi(y))\bigr)>\varepsilon,\]
that is $W'=\phi^\lozenge(W)\in\supp_\varepsilon\bigl(\phi(\mathcal{X})\bigr)$.
\end{proof}

\begin{definition}[{\bf Concreteness}] 
\label{concreteness}
Let $(\mathcal{X},\mathfrak{S})$ be a hierarchically hyperbolic space with the intersection property. We say that the hierarchically hyperbolic structure is $\varepsilon$-\emph{concrete} if either the space $\mathcal{X}$ is bounded, or the $\sqsubseteq$-maximal element $S$ of $\mathfrak{S}$ is equal to 
\begin{equation*}
\bigvee\lbrace V\in\mathfrak{S}\mid V\in\supp_{\varepsilon}(\mathcal{X})\rbrace.
\end{equation*}
We say that the hierarchically hyperbolic space is \emph{concrete} if
it is $\varepsilon$-concrete for some $\varepsilon >3\max\{\xi,\alpha\}$.
\end{definition}

\begin{remark}
Given a hierarchically hyperbolic group $(\mathcal{X},\mathfrak{S})$ with $\sqsubseteq$-maximal element $S$, we have that $\supp_\varepsilon({\bf F}_S)\subseteq \supp_\varepsilon(\mathcal{X})$, because
${\bf F}_S \subseteq\mathcal{X}$.

Notice that the other inclusion is not guaranteed, in general. Nevertheless, if the hierarchical structure on $\mathcal{X}$ is \emph{normalized} \cite[Definition 1.15]{DHS}, that is if the projections $\pi_U$ are uniformly coarsely surjective for all $U\in\mathfrak{S}$, then it follows that ${\bf F}_S=\mathcal{X}$, and in particular that 
$\supp_\varepsilon({\bf F}_S)= \supp_\varepsilon(\mathcal{X})$. As specified in Remark \ref{remark_coarsely_surjective}, we are assuming this.

By \cite[Proposition 1.16]{DHS}, any hierarchically hyperbolic space $(\mathcal{X},\mathfrak{S})$ admits a normalized hierarchically hyperbolic structure $(\mathcal{X},\mathfrak{S}')$ and a hieromorphism
$\phi\colon (\mathcal{X},\mathfrak{S})\to (\mathcal{X},\mathfrak{S}')$ where
$\phi\colon \mathcal{X}\to\mathcal{X}$ is the identity and $\phi^\lozenge\colon\mathfrak{S}\to \mathfrak{S}'$ is a bijection.
Therefore, up to considering normalized hierarchically hyperbolic spaces, an unbounded hierarchically hyperbolic space $(\mathcal{X},\mathfrak{S})$ is $\varepsilon$-concrete and its $\sqsubseteq$-maximal element $S$ is equal to
$\bigvee\lbrace V\in\mathfrak{S}\mid V\in\supp_{\varepsilon}({\bf F}_S)\rbrace$.
\end{remark}

In Definition \ref{concreteness} we are not asking that the maximal element $S$ already belongs to $\supp_\varepsilon(\mathcal{X})$: for instance, this is not the case for direct products of hierarchically hyperbolic spaces and groups, where the hyperbolic space associated to this $\sqsubseteq$-maximal element is bounded.

We are interested in concrete hierarchically hyperbolic spaces for the following proposition:

\begin{proposition}\label{concreteness_making}
Let $(\mathcal{X},\mathfrak{S})$ be an unbounded hierarchically hyperbolic space with the intersection property and let $\varepsilon>3\max\lbrace\xi,\alpha\rbrace$. There exists $\mathfrak{S}_{\varepsilon}\subseteq \mathfrak{S}$ such that
$(\mathcal{X},\mathfrak{S}_{\varepsilon})$ is an unbounded, $\varepsilon$-concrete hierarchically hyperbolic space with the intersection property.
\begin{proof}
Let $S$ be the $\sqsubseteq$-maximal element of $\mathfrak{S}$. If 
\begin{equation}\label{equality_bigvee}
S=\bigvee\{V\in\mathfrak{S}\mid V\in\supp_\varepsilon(\mathcal{X})\},
\end{equation}
then $\mathfrak{S}_{\varepsilon}=\mathfrak{S}$ and there is nothing to prove.

\smallskip
If the equality of Equation \eqref{equality_bigvee} is not satisfied, then 
$\bigvee\{V\in\mathfrak{S}\mid V\in\supp_\varepsilon(\mathcal{X})\}$ is properly nested into the $\sqsubseteq$-maximal element $S$.
Let $S_{\varepsilon}:=\bigvee\{V\in\mathfrak{S}\mid V\in\supp_\varepsilon(\mathcal{X})\}$ and
$\mathfrak{S}_{\varepsilon}\vcentcolon=\mathfrak{S}_{S_{\varepsilon}}$. 

We now claim that there exists $C=C(\varepsilon)$ such that $\mathcal{X}=\mathcal{N}_C(\textbf{F}_{S_{\varepsilon}})$. Let $x\in\mathcal{X}$ and consider the tuple $\vec{c}$ defined as follows:
\begin{equation*}
c_V=\begin{cases}
\pi_V(x),&\quad \forall\  V\in\mathfrak{S}_{S_{\varepsilon}};\\
\pi_V(e),& \quad\forall\ V\in\mathfrak{S}^{\perp}_{S_{\varepsilon}};\\
\rho_V^{S_{\varepsilon}}&\quad\forall\ V\pitchfork S_{\varepsilon} \text{ or } V\sqsupseteq S_{\varepsilon};
\end{cases}
\end{equation*}
where $e\in\textbf{E}_{S_{\varepsilon}}$ is a fixed, arbitrarily chosen element.

The tuple $\vec{c}$ is a $\kappa$-consistent tuple, where $\kappa$ depends only on $\varepsilon$ and the constants of the hierarchically hyperbolic space $(\mathcal{X},\mathfrak{S})$. By \cite[Theorem 3.1]{BHS2}, there exists $z\in\mathcal{X}$ such that $\pi_U(z)\asymp\pi_U(\vec{c\ })$ for every $U\in\mathfrak{S}$, and by Definition \ref{product_region} the element $z$ belongs to $\textbf{F}_{S_{\varepsilon}}\times\{e\}$. Let $s_0$ be the constant associated to the Distance Formula Theorem for the space $(\mathcal{X},\mathfrak{S})$, and consider $s>\max\{\varepsilon,s_0\}$. There exist $K,C>0$ such that
\begin{equation}\label{concreteness_make_of}
\begin{aligned}
d(x,z)&\leq K\sum_{U\in\mathfrak{S}}\lbr d_U(\pi_U(x),\pi_U(z))\rbr_s+C\\
&=K\left(\sum_{U\in\mathfrak{S}_{S_{\varepsilon}}}\lbr d_U(\pi_U(x),\pi_U(z))\rbr_s+\sum_{U\in\mathfrak{S}\setminus\mathfrak{S}_{S_{\varepsilon}}}\lbr d_U(\pi_U(x),\pi_U(z))\rbr_s\right)+C\\
&= K\sum_{U\in\mathfrak{S}\setminus\mathfrak{S}_{S_{\varepsilon}}}\lbr d_U(\pi_U(x),\pi_U(z))\rbr_s +C.
\end{aligned}
\end{equation}
Note that $d_U(\pi_U(x),\pi_U(z))\leq\varepsilon$ for every $U\in\mathfrak{S}\setminus\mathfrak{S}_{\varepsilon}$. Since $s>\varepsilon$, from Equation \eqref{concreteness_make_of} we conclude that $d(x,z)\leq C$.

To complete the proof, notice that $\textbf{F}_{S_{\varepsilon}}\times\{e\}$ can be endowed with the hierarchical hyperbolic structure $\mathfrak{S}_{S_{\varepsilon}}$. Since $\mathcal{X}=\mathcal{N}_C(\textbf{F}_{S_{\varepsilon}}\times\{e\})$, the space $(\mathcal{X},\mathfrak{S}_{S_{\varepsilon}})$ is  hierarchically hyperbolic, being quasi isometric to $\bigl(\textbf{F}_{S_{\varepsilon}}\times\{e\},\mathfrak{S}_\varepsilon\bigr)$, and it is concrete by construction.

The intersection property in $(\mathcal{X},\mathfrak{S}_{S_{\varepsilon}})$ follows from the intersection property in $(\mathcal{X},\mathfrak{S})$.
\end{proof}
\end{proposition}

Concreteness will play an important role in Lemma \ref{lemma_concreteness} and Theorem \ref{winning_theorem}, after the proof of Theorem \ref{thmB}.

\section{General structure theorems for hierarchically hyperbolic spaces and groups}\label{Section4}
In this section we prove some general results for hierarchically hyperbolic spaces and for hieromorphisms, in particular we prove Theorem 
\ref{thmB}. All this machinery will be used in Section \ref{Section4} to prove Theorem \ref{mainT} and its corollaries.

The following lemma spells out a fact implicitly used in the proof of \cite[Theorem 8.6]{BHS2}. 

\begin{lemma}\label{support_inclusion}
Let $\mathcal{T}$ be a tree of hierarchically hyperbolic spaces with full edge hieromorphisms. If $[U]\sqsubseteq [V]$ then 
$T_{[V]}\subseteq T_{[U]}$.
\end{lemma}
\begin{proof}
As $[U]\sqsubseteq [V]$, there exist a vertex $u\in T$ and representatives $U_u,V_u\in\mathfrak{S}_u$ of $[U]$ and $[V]$ respectively
such that $U_u\sqsubseteq V_u$.
Let $v\in T_{[V]}$: we will prove that $v\in T_{[U]}$. 

Let $\sigma$ be the geodesic connecting $u$ to $v$ in the tree $T$, with consecutive edges $e_1,\dots,e_k$, so that $e_1^-=u$ and $e_k^+=v$.
Since $u,v\in T_{[V]}$ and supports are connected, we conclude that $e_i^{\pm}\in T_{[V]}$ for all $i=1,\dots,k$. Therefore, there exist representatives 
$V_{e_i^{-}}$ and $V_{e_i^+}=V_{e_{i+1}^-}$ of $[V]$ in each 
index set $\mathfrak{S}_{e_i^{\pm}}$, and there exist representatives $V_{e_i}\in\mathfrak{S}_{e_i}$ in each edge space on $\sigma$ such that 
$\phi_{e_i^{\pm}}^{\lozenge}(V_{e_i})=V_{e_i^{\pm}}$.
	
Since $U_u\sqsubseteq V_u=V_{e_1^-}=\phi_{e_1^-}^{\lozenge}(V_{e_1})$, by fullness of $\phi_{e_1^-}$ (compare Definition \ref{fullness_definition}) we know 
that there exists some $U_{e_1}\in\mathfrak{S}_{e_1}$ 
such that $\phi_{e_1^{-}}^{\lozenge}(U_{e_1})=U_u$ and $U_{e_1}\sqsubseteq V_{e_1}$. Thus there exists a representative 
$U_{e_1^+}=\phi_{e_1^{+}}^{\lozenge}(U_{e_1})$ of $[U]$ in $\mathfrak{S}_{e_1^+}$.
	
As hieromorphisms respect nesting, we know that $U_{e_1^+}\sqsubseteq V_{e_1^+}$. Applying the same argument to the other edges $e_i$ of $\sigma$, we conclude 
that there exists a representative $U_v$ of $[U]$ in $\mathfrak{S}_v$ such that $U_v\sqsubseteq V_v$.

Therefore $T_{[V]}\subseteq T_{[U]}$.
\end{proof}

In general, the converse implication of Lemma \ref{support_inclusion} fails to be true. Nevertheless, in Subsection \ref{trees_with_decorations} we show that the tree $\mathcal{T}$ of hierarchically hyperbolic spaces can always be enlarged in a way so that the converse implication holds in the bigger tree $\widetilde{\mathcal{T}}$. 

We now state a lemma that will be useful later.

\begin{lemma}\label{DFlemma}
Given a full hieromorphism $\phi\colon\left(\mathcal{X},\mathfrak{S}\right)\to\left(\mathcal{X}',\mathfrak{S}'\right)$, there exist constants $K,C\geqslant 0$ and 
$s,s'>0$ such that
\begin{equation*}
\sum_{U\in\mathfrak{S}}\lbr d_U(\pi_U(x),\pi_U(y))\rbr_s\leqslant K\sum_{U'\in\phi^{\lozenge}(\mathfrak{S})}\lbr d_{U'}(\pi_{U'}(\phi(x),\pi_{U'}(\phi(y)))\rbr_{s'}+C\qquad\forall x,y\in\mathcal{X}.
\end{equation*}
\begin{proof}
For $U\in\mathfrak{S}$, we denote $\phi^\lozenge(U)$ by $U'$.
As the hieromorphism is full, there exists a uniform constant $\xi$ such that
\begin{equation}\label{equation_s_prime}
d_U\bigl(\pi_U(x),\pi_U(y)\bigr)\leqslant \xi d_{U'}\bigl(\pi_{U'}(\phi(x)),\pi_{U'}(\phi(y))\bigr)+\xi,\qquad\forall\ U\in\mathfrak{S},\ \forall x,y\in\mathcal{X}.
\end{equation}
Choose $s$ and $s'$ such that 
\begin{equation*}
s'\vcentcolon=\frac{s-\xi}{\xi}>1.
\end{equation*}
Suppose that $s\leqslant d_U\bigl(\pi_U(x),\pi_U(y)\bigr)$ for a given $U\in\mathfrak{S}$. 
Then, using Equation \eqref{equation_s_prime}, we obtain that 
\begin{equation}\label{equation_1s_prime_d}
1<s'\leqslant d_{U'}\bigl(\pi_{U'}(\phi(x)),\pi_{U'}(\phi(y))\bigr)=\lbr d_{U'}\bigl(\pi_{U'}(\phi(x)),\pi_{U'}(\phi(y))\bigr)\rbr_{s'}.
\end{equation}
As $s\leqslant d_U\bigl(\pi_U(x),\pi_U(y)\bigr)$
we have that $\lbr d_U\bigl(\pi_U(x),\pi_U(y)\bigr)\rbr_s=d_U\bigl(\pi_U(x),\pi_U(y)\bigr)$. It then follows that 
\begin{equation}\label{equation_clip_inequality}
\begin{split}
\lbr d_U\bigl(\pi_U(x),\pi_U(y)\bigr)\rbr_s&=d_U\bigl(\pi_U(x),\pi_U(y)\bigr)\leqslant  \xi d_{U'}\bigl(\pi_{U'}(\phi(x)),\pi_{U'}(\phi(y))\bigr)+\xi\\
&\leqslant \xi\lbr d_{U'}\bigl(\pi_{U'}(\phi(x)),\pi_{U'}(\phi(y))\bigr)\rbr_{s'}+\xi.
\end{split}
\end{equation}
Therefore, using Equation \eqref{equation_1s_prime_d} and Equation \eqref{equation_clip_inequality}, we obtain 
\begin{equation}\label{eq1_DFL}
\begin{split}
\lbr d_U\bigl(\pi_U(x),\pi_U(y)\bigr)\rbr_s&\leqslant \xi\lbr d_{U'}\bigl(\pi_{U'}(\phi(x)),\pi_{U'}(\phi(y))\bigr)\rbr_{s'}+\xi\\
&\leqslant \xi\lbr d_{U'}\bigl(\pi_{U'}(\phi(x)),\pi_{U'}(\phi(y))\bigr)\rbr_{s'}+\xi\lbr d_{U'}\bigl(\pi_{U'}(\phi(x)),\pi_{U'}(\phi(y))\bigr)\rbr_{s'}\\
&= 2\xi\lbr d_{U'}\bigl(\pi_{U'}(\phi(x)),\pi_{U'}(\phi(y))\bigr)\rbr_{s'}.
\end{split}
\end{equation}
On the other hand, if $s> d_U\bigl(\pi_U(x),\pi_U(y)\bigr)$ then 
\begin{equation}\label{eq2_DFL}
\lbr d_U\bigl(\pi_U(x),\pi_U(y)\bigr)\rbr_s=0\leqslant 2\xi \lbr d_{U'}\bigl(\pi_{U'}(\phi(x)),\pi_{U'}(\phi(y))\bigr)\rbr_{s'},
\end{equation}
so the inequality of Equation \eqref{eq1_DFL} is satisfied also in this case.

Concluding, we use Equation \eqref{eq1_DFL} and Equation \eqref{eq2_DFL} to obtain that
\begin{equation*}
\begin{split}
\sum_{U\in\mathfrak{S}}\lbr d_U(\pi_U(x),\pi_U(y))\rbr_s&\leqslant \sum_{U\in\mathfrak{S}}2\xi \lbr d_{U'}\bigl(\pi_{U'}(\phi(x)),\pi_{U'}(\phi(y))\bigr)\rbr_{s'}\\
&=2\xi \sum_{U\in\mathfrak{S}}\lbr d_{U'}\bigl(\pi_{U'}(\phi(x)),\pi_{U'}(\phi(y))\bigr)\rbr_{s'}\\
&=2\xi \sum_{U'\in\phi^\lozenge(\mathfrak{S})}\lbr d_{U'}\bigl(\pi_{U'}(\phi(x)),\pi_{U'}(\phi(y))\bigr)\rbr_{s'},
\end{split}
\end{equation*}
and therefore the lemma is satisfied with $K=2\xi$ and $C=0$.
\end{proof}
\end{lemma}

\begin{remark}\label{DFremark}
The argument of Lemma \ref{DFlemma} can be used to show that there exist constants $\bar K,\bar C\geqslant 0$ and $\bar s,\bar s'>0$ such that
\begin{equation*}
\sum_{U'\in\phi^{\lozenge}(\mathfrak{S})}\lbr d_{U'}\bigl(\pi_{U'}(\phi(x)),\pi_{U'}(\phi(y))\bigr)\rbr_{\bar{s}}\leqslant 
\bar K\sum_{U\in\mathfrak{S}}\lbr d_U(\pi_U(x),\pi_U(y))\rbr_{\bar s'}+\bar C\qquad \forall\ x,y\in\mathcal{X}.
\end{equation*}
\end{remark}

\begin{lemma}\label{useful_lemma}
Let $\phi\colon (\mathcal{X},\mathfrak{S})\to(\mathcal{X}',\mathfrak{S}')$ be a full hieromorphism and $S$ be the $\sqsubseteq$-maximal element in $\mathfrak{S}$. If $S'=\phi^{\lozenge}(S)$ and $\textbf{F}_{S'}\times\{e\}$ is a parallel copy of $\textbf{F}_{S'}$, then $\pi_{V'}(\textbf{F}_{S'}\times\{e\})$ is coarsely equal to $\pi_{V'}(\phi(\mathcal{X}))$ for all $V'\in\mathfrak{S}'_{S'}$. 
\end{lemma}
\begin{proof}
Let $z\in\textbf{F}_{S'}$ and consider the tuple $\vec{b}=\bigl(\pi_{V'}(z)\bigr)_{V'\in\mathfrak{S}'_{S'}}$. 
As $z\in  {\bf F}_{S'}$, the tuple $\vec{b}$ is $\kappa$-consistent.
The hieromorphism $\phi$ is full, therefore $\mathfrak{S}'_{S'}=\phi^\lozenge(\mathfrak{S})$ and
\[\bigl(\pi_{V'}(z)\bigr)_{V'\in\mathfrak{S}'_{S'}}=\bigl(\pi_{V'}(z)\bigr)_{V'\in\phi^\lozenge(\mathfrak{S})}.\]
As the full hieromorphism $\phi$ 
induces uniform quasi isometries $\bar{\phi}_V^*\colon \mathcal{C}V'\to\mathcal{C}V$ at the level of hyperbolic spaces, we obtain a tuple 
$\vec{a}= (a_V)_{V\in\mathfrak{S}}$, where $a_V\vcentcolon=\bar{\phi}_V^*\bigl(\pi_{V'}(z)\bigr)\subseteq \mathcal{C}V$. 

The tuple $\vec{a}$ is $\kappa'$-consistent, and therefore there exists $x\in\mathcal{X}$ that realizes it, by \cite[Theorem 3.1]{BHS2}. Exploiting the fact that the maps $\phi_V^*\circ\pi_V$ uniformly
coarsely coincide with the $\pi_{V'}\circ\phi$ (compare Definition \ref{HHS_hieromorphism} and in particular Equation \eqref{coarsely.commuting.diagrams}), 
we conclude that the element $\phi(x)$ realizes the tuple $\vec{b}$:
\begin{equation}\label{realization_points}
\bigl(\pi_{V'}(z)\bigr)_{V'\in\phi^\lozenge(\mathfrak{S})}\asymp \bigl(\pi_{V'}(\phi(x))\bigr)_{V'\in\phi^\lozenge(\mathfrak{S})}.
\end{equation}
That is, there exists a constant $T_1$ depending only on the realization Theorem \cite[Theorem 3.1]{BHS2} and the hieromorphism $\phi$ such that $d_{V'}(\pi_{V'}(z),\pi_{V'}(\phi(x)))\leq T_1$ for every $V'\in\mathfrak{S}'_{S'}$.

Conversely, let $\phi(x)\in\phi(\mathcal{X})$ and consider the tuple $\vec{c}$:
\begin{equation*}
c_{V'}=\begin{cases}
\pi_{V'}(\phi(x)),&\quad \forall\  V'\in\mathfrak{S}'_{S'};\\
\pi_{V'}(e),& \quad\forall\ V'\in\mathfrak{S}'^{\perp}_{S'};\\
\rho_{V'}^{S'}&\quad\forall\ V'\pitchfork S' \text{ or } V'\sqsupseteq S'.
\end{cases}
\end{equation*}
Since $\vec{c}\ $ is a $\kappa$-consistent tuple, there exists $z\in\mathcal{X}$ such that $\pi_V(z)\asymp\pi_V(\vec{c}\ )$, and $z$ belongs to $\textbf{F}_{S'}\times\{e\}$ by Definition \ref{product_region}. Therefore there exists $T_2$ such that $d_{V'}(\pi_{V'}(z),\pi_{V'}(\phi(x)))\leq T_2$ for every $V'\in\mathfrak{S}'_{S'}$.
\end{proof}

\begin{proposition}\label{XqiFS}
If $\phi\colon (\mathcal{X},\mathfrak{S})\to(\mathcal{X}',\mathfrak{S}')$ is a full hieromorphism between hierarchically hyperbolic spaces,
then the spaces $\mathcal{X}$ and ${\bf F}_{S'}$ are quasi isometric, where $S'$ is the image in $\mathfrak{S}'$ of the $\sqsubseteq$-maximal element 
of $\mathfrak{S}$. 
\begin{proof}

We define a map $\psi\colon {\bf F}_{S'}\to\mathcal{X}$ and we prove that it is a quasi isometry. 
Let $z\in  {\bf F}_{S'}$, and consider the tuple $\vec{b}=\bigl(\pi_{V'}(z)\bigr)_{V'\in\mathfrak{S}'_{S'}}$. 
As $z\in  {\bf F}_{S'}$, the tuple $\vec{b}$ is $\kappa$--consistent.
The hieromorphism $\phi$ is full, so that $\mathfrak{S}'_{S'}=\phi^\lozenge(\mathfrak{S})$ and
\[\bigl(\pi_{V'}(z)\bigr)_{V'\in\mathfrak{S}'_{S'}}=\bigl(\pi_{V'}(z)\bigr)_{V'\in\phi^\lozenge(\mathfrak{S})}.\]
As the full hieromorphism $\phi$ 
induces uniform quasi isometries $\bar{\phi}_V^*\colon \mathcal{C}V'\to\mathcal{C}V$ at the level of hyperbolic spaces, we obtain a tuple 
$\vec{a}= (a_V)_{V\in\mathfrak{S}}$, where $a_V\vcentcolon=\bar{\phi}_V^*\bigl(\pi_{V'}(z)\bigr)\subseteq \mathcal{C}V$.

The tuple $\vec{a}$ is $\kappa'$-consistent, and therefore there exists $x\in\mathcal{X}$ that realizes it by \cite[Theorem 3.1]{BHS2}. Exploiting the fact that the maps $\phi_V^*$ uniformly
coarsely commute with the projections $\pi_{V}$ (compare Definition \ref{HHS_hieromorphism} and in particular Equation \eqref{coarsely.commuting.diagrams}), 
we conclude that the element $\phi(x)$ realizes the tuple $\vec{b}$:
\begin{equation}\label{realization_points_1}
\bigl(\pi_{V'}(z)\bigr)_{V'\in\phi^\lozenge(\mathfrak{S})}\asymp \bigl(\pi_{V'}(\phi(x))\bigr)_{V'\in\phi^\lozenge(\mathfrak{S})}.
\end{equation}
Define $\psi(z)\vcentcolon=x$.
The element $x$ is not uniquely determined by the tuple $\vec{b}$, but it is up to uniformly bounded error.

Let us prove that $\psi$ is a quasi isometry. Indeed, let $z_1,z_2\in{\bf F}_{S'}$. Using, in this order, the Distance Formula in $\mathcal{X}'$, Remark \ref{DFremark}, and
the fact that $\phi$ is a full hieromorphism combined with the Distance Formula in ${\bf F}_{S'}$, we have that
\begin{equation}\label{quasi_isometry_1}
\begin{aligned}
d_{\mathcal{X}}(\psi(z_1),\psi(z_2))&\leqslant K\sum_{U\in\mathfrak{S}}\lbr d_U(\pi_U(\psi(z_1)),\pi_U(\psi(z_2))) \rbr_s+C\\
&\leqslant K\Bigl(K_1\sum_{U'\in\phi^\lozenge(\mathfrak{S})}\lbr d_{U'}(\pi_{U'}(z_1),\pi_{U'}(z_2) )\rbr_{\bar s}+C_1\Bigr)+C\\
&\leqslant K\bigl(K_1(K_2 d_{\mathcal{X}'}(z_1,z_2)+C_2)+C_1\bigr)+C.
\end{aligned}
\end{equation}
On the other hand, we have that
\begin{equation}\label{quasi_isometry_2}
\begin{aligned}
d_{\mathcal{X}'}(z_1,z_2)&\leqslant K_3\sum_{U'\in\phi^\lozenge(\mathfrak{S})}\lbr d_{U'}\bigl(\pi_{U'}(z_1),\pi_{U'}(z_2)\bigr)\rbr_{s'} + C_3\\
&\leqslant K_3\Bigl(K_4\sum_{U\in\mathfrak{S}}\lbr d_U\bigl(\pi_{U}(\psi(z_1)),\pi_{U}(\psi(z_2))\bigr)\rbr_{\bar{s}'}\Bigr)+ C_3\\
&\leqslant K_3\bigl(K_4(K_5d_{\mathcal{X}}(\psi(z_1),\psi(z_2))+C_5)+C_4\bigr)+C_3.
\end{aligned}
\end{equation}
Equation \eqref{quasi_isometry_1} and Equation \eqref{quasi_isometry_2} prove that $\psi$ is a quasi-isometric embedding.

Moreover, the map $\psi$ is coarsely surjective. Indeed, given an element $x\in \mathcal{X}$, the tuple $\left(\pi_{V'}(\phi(x)\right)_{V'\in
\phi^{\lozenge}(\mathfrak{S})}$ is consistent, and therefore there exists a point $z\in {\bf F}_{S'}$ coarsely realizing it, that is uniformly close to $x$. 

\end{proof}
\end{proposition}

\begin{example}[Hagen] 
\label{example_Hagen}
It very well may happen that a full hieromorphism 
between hierarchically hyperbolic spaces fails to be coarsely lipschitz.

We describe such a hieromorphism $\phi\colon \bigl(\R,\{\R\}\bigr)\to \bigl(X,\mathfrak{S}\bigr)$ here, where $X$ is the Cayley graph of the free group $F_2=F(a,b)$ with respect to the free generating set $\{a,b\}$.
The structure $\mathfrak{S}$ on $X$ is given by the family
$\mathfrak{S}$ of all axes of conjugates of $a$ and of $b$, and
a $\sqsubseteq$-maximal element $M$:
\[\mathfrak{S}:=\bigl\{\bigcup_{g\in F_2}\Ax(a^g)\bigr\}\cup\bigl\{\bigcup_{g\in F_2}\Ax(b^g)\bigr\}\cup\{M\},\]
where the axis $\Ax(x)$ of an element $x$ is defined to be the set of vertices of $X$ with minimal displacement with respect to $x$, that is
$\Ax(x):=\{y\in F_2\mid d_X(y,xy)\text{ is minimal}\}$.

In $\mathfrak{S}$ any two different axes are transverse, and everything is nested into $M$. The hyperbolic spaces associated to the axes are their corresponding lines in $X$, and $\mathcal{C}M$ is obtained from $X$ by coning off all these axes. 

The projections $\pi_{\Ax(x^g)}\colon F_2\to 2^{\Ax(x^g)}$ are given by closest-point projections, for all $x=a,b$ and $g\in F_2$, as well as the $\rho$ maps between two axes. The sets $\rho_M^{\Ax(x^g)}$ are the inclusion of the axis into the coned-off Cayley graph.

\smallskip
The map $\phi$ is defined as follows. At the level of metric spaces, $\phi$ maps $\R$ homeomorphically into $X$ in the following way. For $n\in\Z$, the segment $[n,n+1]\subseteq \R$ is mapped to the geodesic path that connects $a^nb^n$ to $a^{n+1}b^{n+1}$ in $X$. For this reason the map $\phi$ is not coarsely lipschitz, because the segment $[n,n+1]\subseteq \R$, which has length one, is mapped to a geodesic path of length $2n+2$ in $X$.

The map $\phi^\lozenge\colon\{\R\}\to\mathfrak{S}$ is defined as $\phi^\lozenge(\R)=\Ax(a)$, whilst
the map $\phi^*_\R\colon \R\to\Ax(a)$ is the isometry such that $\phi^*_\R(0)=e$ and $\phi^*_\R(1)=a$.

It can be checked that $\phi$ is a hieromorphism, and that it is full. Moreover,
$\phi(\R)$ is hierarchically quasiconvex inside $\bigl(X,\mathfrak{S}\bigr)$.
\end{example}

We now prove Theorem \ref{thmB} from the Introduction.
\begin{thmB}
Let $\phi\colon(\mathcal{X},\mathfrak{S})\to(\mathcal{X}',\mathfrak{S}')$ be a full hieromorphism with hierarchically quasiconvex image, and let 
$S$ be the $\sqsubseteq$-maximal element of $\mathfrak{S}$. 
The following are equivalent:
\begin{enumerate}
 \item $\phi$ is coarsely lipschitz;
 \item $\phi$ is a quasi-isometric embedding;
  \item the maps $\mathfrak{g}_{\phi(\mathcal{X})}\colon{\textbf F}_{\phi^\lozenge(S)}\to\phi(\mathcal{X})$ and $\mathfrak{g}_{{\textbf F}_{\phi^\lozenge(S)}}\colon\phi(\mathcal{X})\to\textbf{F}_{\phi^\lozenge(S)}$ are quasi-inverses of each other, and in particular quasi isometries;
 \item the subspace $\phi(\mathcal{X})\subseteq \mathcal{X}'$, endowed with the subspace metric, admits a hierarchically hyperbolic structure obtained by from 
 one of $\mathcal{X}$ by composition with the map $\phi$ (and its induced maps at the level of hyperbolic spaces);
 \item $\pi_{W}(\phi(\mathcal{X}))$ is uniformly bounded for every $W\in\mathfrak{S}'\setminus\phi^\lozenge(\mathfrak{S})$.
\end{enumerate}
\end{thmB}
\begin{proof}
The implications $3\Leftrightarrow 5\Rightarrow 1 \Leftrightarrow 2\Rightarrow 4\Rightarrow 1$ and $2\Rightarrow 3$ are enough to prove the theorem.
\item[\fbox{$5\Rightarrow 1$}] By the Distance Formula applied in $(\mathcal{X}',\mathfrak{S}')$, there exists $s_0$ such that for every $s>s_0$ there exists $K',C'\geqslant 0$ for which 
\begin{equation}\label{DFinequality}
d_{\mathcal{X}'}(\phi(x),\phi(y))\leq K'\sum_{V\in\mathfrak{S}'}\lbr d_V\left(\pi_{V}(\phi(x)),\pi_{V}(\phi(y))\right)\rbr_s+C'\qquad \forall x,y\in\mathcal{X}.
\end{equation}
Also, the Distance Formula applied in $(\mathcal{X},\mathfrak{S})$ implies that there exists $s_1$ such that for every $s>s_1$ there exist $K,C\geqslant 0$ 
for which 
\begin{equation}\label{DFinequality2}
d_{\mathcal{X}}(x,y)\geq K^{-1}\sum_{U\in\mathfrak{S}}\lbr d_U(\pi_U(x),\pi_U(y))\rbr_s-C\qquad\forall x,y\in\mathcal{X}.
\end{equation}
Now let $x,y\in\mathcal{X}$. By hypothesis $\pi_W\left(\phi(\mathcal{X})\right)$ is uniformly bounded for every 
$W\in\mathfrak{S}'\setminus\phi^\lozenge(\mathfrak{S})$. Let $M$ be this uniform bound, and choose $s$ such that $s>\max\{M,s_0\}$. Therefore
\begin{equation*}
\sum_{V\in\mathfrak{S}'}\lbr d_V\left(\pi_V(\phi(x)),\pi_V(\phi(y))\right)\rbr_s=\sum_{U'\in\phi^{\lozenge}(\mathfrak{S})}\lbr d_{U'}
\left(\pi_{U'}(\phi(x)),\pi_{U'}(\phi(y))\right)\rbr_s
\end{equation*}
and Equation \eqref{DFinequality} implies that
\begin{equation*}
d_{\mathcal{X}'}(\phi(x),\phi(y))\leq K'\sum_{U'\in\phi^{\lozenge}(\mathfrak{S})}\lbr d_{U'}\left(\pi_{U'}(\phi(x)),\pi_{U'}(\phi(y))\right)\rbr_s+C'.
\end{equation*}
Using Remark \ref{DFremark}, we can choose $\bar{s}, \bar{s}'>s_1$ and $\bar{K},\bar{C}\geqslant 0$ for which 
\begin{equation*}
\sum_{U'\in\phi^{\lozenge}(\mathfrak{S})}\lbr d_{U'}\left(\pi_{U'}(\phi(x)),\pi_{U'}(\phi(y))\right)\rbr_{\bar{s}}\leq \bar{K}\sum_{U\in\mathfrak{S}}\lbr d_U\left( \pi_U(x),\pi_U(y)\right)\rbr_{\bar s'}+\bar{C}.
\end{equation*}
By taking $\tilde{s}=\max\lbrace s_0,\bar{s}\rbrace$ we get 
\begin{equation*}\begin{aligned}\sum_{U'\in\phi^{\lozenge}(\mathfrak{S})}\lbr d_{U'}\left(\pi_{U'}(\phi(x)),\pi_{U'}(\phi(y))\right)\rbr_{\tilde{s}}&\leq \sum_{U'\in\phi^{\lozenge}(\mathfrak{S})}\lbr d_{U'}\left(\pi_{U'}(\phi(x)),\pi_{U'}(\phi(y))\right)\rbr_{\bar{s}}\\
&\leq \bar{K}\sum_{U\in\mathfrak{S}}\lbr d_{U}\left(\pi_{U}(x),\pi_{U}(y)\right)\rbr_{\bar s'}+\bar{C}.\end{aligned}\end{equation*}
As $\bar s'>s_1$, by the Distance Formula, Equation \eqref{DFinequality} and Equation \eqref{DFinequality2} we obtain
\begin{equation*}
\begin{aligned}
d_{\mathcal{X}'}(\phi(x),\phi(y))&\leq K'\sum_{U'\in\phi^{\lozenge}(\mathfrak{S})}\lbr d_{U'}\left(\pi_{U'}(\phi(x)),\pi_{U'}(\phi(y))\right)\rbr_{\tilde{s}}+C'\\
&\leq K'\bar{K}\sum_{U\in\mathfrak{S}}\lbr d_U\left(\pi_U(x),\pi_U(y)\right)\rbr_{\bar s'}+K'\bar{C}+C'\\
&\leq K'\bar{K}\left(Kd_{\mathcal{X}}(x,y)+KC\right)+K'\bar{C}+C'=Rd_{\mathcal{X}}(x,y)+R'
\end{aligned}
\end{equation*}
for appropriate constants $R$ and $R'$. Therefore, $\phi$ is a coarsely lipschitz map.

\item[\fbox{$1\Leftrightarrow 2$}]
If $\phi$ is a quasi-isometric embedding, then it is a coarsely lipschitz map.

Suppose now that $\phi$ is a coarsely lipschitz map. To conclude that it is a quasi-isometric embedding, we need to prove that there exist constants 
$K,C\geq 0$ such that $d_{\mathcal{X}}(x,y)\leq Kd_{\mathcal{X}'}(\phi(x),\phi(y))+C$ for every $x,y\in\mathcal{X}$.

By the Distance Formula applied in $\left(\mathcal{X},\mathfrak{S}\right)$, there exists $s_0$ so that for every $s\geqslant s_0$ there exist $K_1,C_1\geq 0$ so that
\begin{equation*}
d_{\mathcal{X}}(x,y)\leq K_1\sum_{U\in\mathfrak{S}}\lbr d_U(\pi_U(x),\pi_U(y))\rbr_s+C_1,\qquad \forall x,y\in\mathcal{X}.
\end{equation*}
Also by the Distance Formula applied to $(\mathcal{X}',\mathfrak{S}')$, there exists $s_1$ so that for every $s\geq s_1$ there exist $K_2,C_2\geq 0$ so that 
\begin{equation*}
d_{\mathcal{X}'}(\phi(x),\phi(y))\geq K_2^{-1}\sum_{W\in\mathfrak{S}'}\lbr d_W(\pi_W(\phi(x)),\pi_W(\phi(y)))\rbr_s-C_2,\qquad \forall x,y\in\mathcal{X}.
\end{equation*}
By Lemma \ref{DFlemma}, we can choose $\bar{s},\bar{s}'>s_1$ and $\bar{K},\bar{C}\geqslant 0$ such that 
\begin{equation*}
\begin{aligned}
\sum_{U\in\mathfrak{S}}\lbr d_U\left(\pi_U(x),\pi_U(y)\right)\rbr_{\bar{s}}&\leq\bar{K}\sum_{U'\in\phi^{\lozenge}(\mathfrak{S})}\lbr d_{U'}(\pi_{U'}(\phi(x)),\pi_{U'}(\phi(y)))\rbr_{\bar{s}'}+\bar{C}\\
&\leq \bar{K}\sum_{W\in\mathfrak{S}'}\lbr d_W(\pi_W(\phi(x)),\pi_W(\phi(y)))\rbr_{\bar{s}'}+\bar{C},\qquad\forall x,y\in\mathcal{X}.
\end{aligned}
\end{equation*}
Let $s=\max\lbrace s_0,\bar{s}\rbrace$. Since $s\geq s_0$ and $s\geq \bar{s}$, for any $x,y\in\mathcal{X}$ we obtain that
\begin{equation*}
\begin{aligned}
d_{\mathcal{X}}(x,y)&\leq K_1\sum_{U\in\mathfrak{S}}\lbr d_U(\pi_U(x),\pi_U(y))\rbr_s+C_1\leq K_1\sum_{U\in\mathfrak{S}}\lbr d_U(\pi_U(x),\pi_U(y))\rbr_{\bar{s}}+C_1\\
&\leq K_1\left(\bar{K}\sum_{W\in\mathfrak{S}'}\lbr d_W(\pi_W(\phi(x)),\pi_W(\phi(y)))\rbr_{\bar{s}'}+\bar{C}\right)+C_1\\
&\leq K_1\bar{K}\left(K_2d_{\mathcal{X}'}(\phi(x),\phi(y))+\bar{K} C_2\right)+K_1\bar{C}+C_1\\
&=Sd_{\mathcal{X}'}(\phi(x),\phi(y))+S'
\end{aligned}
\end{equation*}
for appropriate constants $S$ and $S'$.
Therefore, $\phi$ is a quasi-isometric embedding.
\item[\fbox{$2\Rightarrow 4$}]
If the map $\phi$ is a quasi-isometric embedding then $(4)$ is automatically satisfied, because hierarchical hyperbolicity is preserved under quasi 
isometries (compare with the remark before \cite[Theorem G]{BHS1}).

\item[\fbox{$4\Rightarrow 1$}]
As the hieromorphism is full, every induced map $\phi^*_U\colon\mathcal{C}U\to\mathcal{C}(\phi^{\lozenge}(U))$ is a $(\xi,\xi)$-quasi isometry, where $\xi$ is independent of $U\in\mathfrak{S}$, that is
\begin{equation*}\label{equation_boh}
\xi^{-1}d_U(\pi_U(x),\pi_U(y)) -\xi\leqslant d_{\phi^{\lozenge}(U)}(\phi^*_U(\pi_U(x)),\phi^*_U(\pi_U(y)))\leqslant \xi d_U(\pi_U(x),\pi_U(y))+\xi
\end{equation*}
for all $U\in\mathfrak{S}$ and for all $x,y\in\mathcal{X}$.

By the Distance Formula applied in $\left(\mathcal{X},\mathfrak{S}\right)$, there exists $s_0$ such that for every $s\geqslant s_0$ there exist $K_1,C_1\geq 0$ satisfying
\begin{equation}\label{4->2_first}
d_{\mathcal{X}}(x,y)\geqslant K_1^{-1}\sum_{U\in\mathfrak{S}}\lbr d_U(\pi_U(x),\pi_U(y))\rbr_s-C_1,\qquad \forall x,y\in\mathcal{X}.
\end{equation}
We apply now the Distance Formula to the hierarchically hyperbolic space $(\phi(\mathcal{X}),\phi^{\lozenge}(\mathfrak{S}))$. 
Therefore, there exists $s_1$ such that for every $s\geqslant s_1$ there exist $K_2,C_2\geq 0$ satisfying 
\begin{equation}\label{4->2_second}
d_{\mathcal{X}'}(\phi(x),\phi(y))\leqslant K_2\sum_{U'\in\phi^{\lozenge}(\mathfrak{S})}\lbr d_{U'}(\pi_{U'}(\phi(x)),\pi_{U'}(\phi(y)))\rbr_s+C_2,\qquad 
\forall x,y\in\mathcal{X}.
\end{equation}
By Remark \ref{DFremark}, we can choose $\bar{s}, \bar{s}'>s_0$ and $\bar{K},\bar{C}\geqslant 0$ for which 
\begin{equation}\label{4->2_third}
\sum_{U'\in\phi^{\lozenge}(\mathfrak{S})}\lbr d_{U'}\left(\pi_{U'}(\phi(x)),\pi_{U'}(\phi(y))\right)\rbr_{\bar{s}}\leq \bar{K}\sum_{U\in\mathfrak{S}}
\lbr d_U\left( \pi_U(x),\pi_U(y)\right)\rbr_{\bar s'}+\bar{C},\qquad \forall x,y\in\mathcal{X}.
\end{equation}
For $s=\max\lbrace s_1,\bar{s}\rbrace$, combining Equation \eqref{4->2_first}, Equation \eqref{4->2_second}, and Equation \eqref{4->2_third}, we obtain that
\begin{equation*}
\begin{aligned}
d_{\mathcal{X}'}(\phi(x),\phi(y))&\leq K_2\sum_{U'\in\phi^{\lozenge}(\mathfrak{S})}\lbr d_{U'}(\pi_{U'}(\phi(x)),\pi_{U'}(\phi(y)))\rbr_s+C_2\\
&\leq K_2\left(\bar{K}\sum_{U\in\mathfrak{S}}\lbr d_U(\pi_U(x),\pi_U(y))\rbr_{\bar{s}'}+\bar{C}\right) +C_2\\
&\leq K_2\bar{K}\left(K_1 d_{\mathcal{X}}(x,y)+K_1C_1\right)+K_2\bar{C}+C_2 = Td_{\mathcal{X}}(x,y)+T'
\end{aligned}
\end{equation*}
for appropriate constants $T$ and $T'$. Therefore, $\phi$ is a coarsely lipschitz map.

\item[\fbox{$3\Rightarrow 5$}]
By hypothesis, $\mathfrak{g}_{\textbf{F}_{S'}}\colon \phi(\mathcal{X})\to\textbf{F}_{S'}$ and $\mathfrak{g}_{\phi(\mathcal{X})}:\textbf{F}_{S'}\to\phi(\mathcal{X})$ are quasi inverses of each other, and by construction of gate maps they are also coarsely lipschitz. 
Therefore $\textbf{F}_{S'}$ and $\phi(\mathcal{X})$ are quasi-isometric, where the quasi-isometry is given by $\mathfrak{g}_{\textbf{F}_{S'}}$, and in particular there exists $C>0$ such that 

\[\phi(\mathcal{X})\subseteq\mathcal{N}_C(\mathfrak{g}_{\phi(\mathcal{X})}(\textbf{F}_{S'})).\]
Let $W\in\mathfrak{S}'\setminus\phi^{\lozenge}(\mathfrak{S})$. By the previous inclusion, there exists $C'>0$, depending on $C$ and on $\pi_{W}$, such that 
\begin{equation}\label{coarse_surjection}
\pi_W(\phi(\mathcal{X}))\subseteq\mathcal{N}_{C'}\bigl(\pi_W(\mathfrak{g}_{\phi(\mathcal{X})}(\textbf{F}_{S'}))\bigr).
\end{equation}
Since the hieromorphism $\phi$ is full, $\phi^{\lozenge}(\mathfrak{S})=\mathfrak{S}'_{S'}$. Moreover, by construction of gate maps, the set $\pi_W(\mathfrak{g}_{\phi(\mathcal{X})}(\textbf{F}_{S'}))$ is uniformly coarsely equal to $p_{\pi_W(\phi(\mathcal{X}))}(\pi_W(\textbf{F}_{S'}))$, where $p_{\pi_W(\phi(\mathcal{X}))}$ is the closest-point projection in $\mathcal{C}W$ to the quasiconvex subspace $\pi_W(\phi(\mathcal{X}))$. Since $W\in\mathfrak{S}'\setminus\mathfrak{S}'_{S'}$, we have that $\diam(\pi_W(\textbf{F}_{S'}))\leq\alpha$ by \cite[Construction 5.10]{BHS2} and, as a consequence, that there exists $\alpha'$ such that $\diam(\pi_W(\mathfrak{g}_{\phi(\mathcal{X})}(\textbf{F}_{S'})))\leq\alpha'$. The first condition of the theorem follows from this, and Equation \eqref{coarse_surjection}.

\item[\fbox{$5\Rightarrow 3$}] We claim that there exists $M>0$ such that 
\[d_{\mathcal{X}'}(\mathfrak{g}_{\textbf{F}_{S'}}\circ\mathfrak{g}_{\phi(\mathcal{X})}(z),z)\leq M,\qquad d_{\mathcal{X}'}(\mathfrak{g}_{\phi(\mathcal{X})}\circ\mathfrak{g}_{\textbf{F}_{S'}}(y),y)\leq M, \qquad \forall z\in\textbf{F}_{S'},\ \forall y\in\phi(\mathcal{X}).\]
By applying the Distance Formula to the space $(\mathcal{X}',\mathfrak{S}')$, there exists $s_0$ such that for every $s\geq s_0$ there exist $\overline{K}_1,\overline{C}_1>0$ such that
\[
d_{\mathcal{X}'}(\mathfrak{g}_{\textbf{F}_{S'}}\circ\mathfrak{g}_{\phi(\mathcal{X})}(z),z)\leq\overline{K}_1\sum_{U'\in\mathfrak{S}'}\lbr d_{U'}(\pi_{U'}(\mathfrak{g}_{\textbf{F}_{S'}}\circ\mathfrak{g}_{\phi(\mathcal{X})}(z)),\pi_{U'}(z))\rbr_s+\overline{C}_1,
\qquad \forall z\in\textbf{F}_{S'}.\]
By Lemma \ref{lemma_inclusion1}, $\diam(\pi_W(\textbf{F}_{S'}))\leq \varepsilon$ for every $W\in\mathfrak{S'}\setminus\mathfrak{S'}_{S'}$ for an appropriate $\varepsilon>0$. For $s\geqslant \max\{s_0,\varepsilon\}$ and the previous equation, it follows that
\begin{equation}\label{implication_5_1}
d_{\mathcal{X}'}(\mathfrak{g}_{\textbf{F}_{S'}}\circ\mathfrak{g}_{\phi(\mathcal{X})}(z),z)\leq\overline{K}_1\sum_{U'\in\mathfrak{S}_{S'}'}\lbr d_{U'}(\pi_{U'}(\mathfrak{g}_{\textbf{F}_{S'}}\circ\mathfrak{g}_{\phi(\mathcal{X})}(z)),\pi_{U'}(z))\rbr_s+\overline{C}_1,\qquad \forall z\in\textbf{F}_{S'}.
\end{equation}
For $z\in\textbf{F}_{S'}$, using the fact that $\mathfrak{g}_{\textbf{F}_{S'}}(z)=z$, we obtain
\begin{equation}
\begin{aligned}d_{U'}\bigl(\pi_{U'}(\mathfrak{g}_{\textbf{F}_{S'}}\circ\mathfrak{g}_{\phi(\mathcal{X})}(z)),\pi_{U'}(z)\bigr)&=d_{U'}(\pi_{U'}(\mathfrak{g}_{\textbf{F}_{S'}}\circ\mathfrak{g}_{\phi(\mathcal{X})}(z)),\pi_{U'}(\mathfrak{g}_{\textbf{F}_{S'}}(z)))\leq\\
&\leq d_{U'}(p(\pi_{U'}\circ\mathfrak{g}_{\phi(\mathcal{X})}(z)),p(\pi_{U'}(z)))+2k\\
&\leq k'd_{U'}(\pi_{U'}(\mathfrak{g}_{\phi(\mathcal{X})}(z),\pi_{U'}(z))+c'+2k,
\end{aligned}
\end{equation}
where $p\colon \mathcal{C}U'\to\pi_{U'}(\textbf{F}_{S'})$ is the closest-point projection to the quasiconvex subspace $\pi_{U'}(\textbf{F}_{S'})\subseteq \mathcal{C}U'$, and $k',c'$ denote the multiplicative and additive constants associated to the coarsely lipschitz map $p$, and $k$ denotes the Hausdorff distance between the (uniformly) coarsely equal sets $\pi_{W}(\mathfrak{g}_{\textbf{F}_{S'}}(x))$ and $p(\pi_W(x))$, for every $x\in~\mathcal{X'}$.

By Lemma \ref{useful_lemma} there exists a constant $T>0$ such that for every $z\in\textbf{F}_{S'}$ there exists $\phi(x)\in\phi(\mathcal{X})$ for which $d_{U'}(\pi_{U'}(\phi(x)),\pi_{U'}(z))\leq T$ for every $U'\in\mathfrak{S}'_{S'}$. Since $\pi_{U'}(\mathfrak{g}_{\phi(\mathcal{X})}(z))$ coarsely equals $p_{\pi_{U'}(\phi(\mathcal{X}))}(\pi_{U'}(z))$, we obtain that
\begin{equation*}
\begin{aligned}
d_{U'}(\pi_{U'}(\mathfrak{g}_{\phi(\mathcal{X})}(z)),\pi_{U'}(z))\leq T'\qquad \forall\ U'\in\mathfrak{S}'_{S'}.
\end{aligned}
\end{equation*}
By choosing an adequate $s$ in Equation \eqref{implication_5_1}, we conclude that 
\[d_{\mathcal{X}'}(\mathfrak{g}_{\textbf{F}_{S'}}\circ\mathfrak{g}_{\phi(\mathcal{X})}(z),z)\leq\overline{C}_1.\]
In order to show that $d_{\mathcal{X}'}(\mathfrak{g}_{\phi(\mathcal{X})}\circ\mathfrak{g}_{\textbf{F}_{S'}}(y),y)$ is uniformly bounded for every $y\in\phi(\mathcal{X})$ let $\mu>0$ denote the constant such that $\diam(\pi_W(\phi(\mathcal{X})))<\mu$ for every $W\in\mathfrak{S}'\setminus\phi^{\lozenge}(\mathfrak{S})=\mathfrak{S}'\setminus\mathfrak{S}^{'}_{S'}$. By the Distance Formula there exists $s_0>0$ such that for all $s\geq s_0$ there exists $\overline{K}_2,\overline{C}_2$ such that
\begin{equation}\label{equation_1_5}
d_{\mathcal{X}'}(\mathfrak{g}_{\phi(\mathcal{X})}\circ\mathfrak{g}_{\textbf{F}_{S'}}(y),y)\leq\overline{K}_2\sum_{U'\in\mathfrak{S}'}\lbr d_{U'}(\pi_{U'}(\mathfrak{g}_{\phi(\mathcal{X})}\circ\mathfrak{g}_{\textbf{F}_{S'}}(y)),\pi_{U'}(y))\rbr_s+\overline{C}_2,\qquad \forall\ y\in\phi(\mathcal{X}).
\end{equation}
Since $\pi_{U'}\circ\mathfrak{g}_{\phi(\mathcal{X})}\asymp p_{\pi_{U'}(\phi(\mathcal{X}))}\circ\pi_{U'}$, it follows that $\pi_{U'}(\mathfrak{g}_{\phi(\mathcal{X})}\circ\mathfrak{g}_{\textbf{F}_{S'}})\asymp p_{\pi_{U'}(\phi(\mathcal{X}))}(\pi_{U'}\circ\mathfrak{g}_{\textbf{F}_{S'}})$. Moreover, if $U'\sqsubseteq S'$, it follows that $\pi_{U'}\circ\mathfrak{g}_{\textbf{F}_{S'}}\asymp \pi_{U'}$, because $\pi_{U'}(\textbf{F}_{S'})\asymp \pi_{U'}(\mathcal{X}')$ for every $U'\sqsubseteq S'$. Therefore, we conclude that $\pi_{U'}(\mathfrak{g}_{\phi(\mathcal{X})}\circ\mathfrak{g}_{\textbf{F}_{S'}})\asymp p_{\pi_{U'}(\phi(\mathcal{X}))}\circ\pi_{U'}$. For any $y\in\phi(\mathcal{X})$ we have that $p_{\pi_{U'}(\phi(\mathcal{X}))}\circ\pi_{U'}(y)=\pi_{U'}(y)$ and, therefore, $\pi_{U'}(\mathfrak{g}_{\phi(\mathcal{X})}\circ\mathfrak{g}_{\textbf{F}_{S'}}(y))\asymp\pi_{U'}(y)$ for every $U'\in\mathfrak{S}'_{S'}$, that is for all $U'\in\mathfrak{S}'$ and for all $y\in\phi(\mathcal{X})$, we have that $d_{U'}(\pi_{U'}(\mathfrak{g}_{\phi(\mathcal{X})}\circ\mathfrak{g}_{\textbf{F}_{S'}}(y)),\pi_{U'}(y))\leq \bar\mu$ for some constant $\bar \mu$.

For $s\geq\max\{s_0,\mu,\bar\mu\}$, Equation \eqref{equation_1_5} yields that $d(\mathfrak{g}_{\phi(\mathcal{X})}\circ\mathfrak{g}_{\textbf{F}_{S'}}(y),y)\leq \overline{C}_2$, that is the distance is uniformly bounded.

\item[\fbox{$2\Rightarrow 3$}] 
We claim that $(\phi(\mathcal{X}),\mathfrak{S'}_{S'})$ is a hierarchically hyperbolic space. Since $(\mathcal{X},\mathfrak{S})$ is a hierarchically hyperbolic space and $\phi(\mathcal{X})$ is quasi isometric to $\mathcal{X}$, we can endow $\phi(\mathcal{X})$ with the hierarchically hyperbolic structure given by the index set $\mathfrak{S}$. 
For $V\in\mathfrak{S}$, the projections $\overline{\pi}_V\colon\phi(\mathcal{X})\to\mathcal{C}V$ in this latter hierarchically hyperbolic space are defined to be $\pi_V\circ\phi^{-1}$, where $\phi^{-1}$ is a fixed quasi inverse of $\phi\colon \mathcal{X}\to\phi(\mathcal{X})$, and $\pi_V$ are the projections in the space $(\mathcal{X},\mathfrak{S})$.

Moreover, we can define the hierarchically hyperbolic space $(\phi(\mathcal{X}),\phi^{\lozenge}(\mathfrak{S}))$. For $V'\in \phi^\lozenge(\mathfrak{S})$, that is for $V'=\phi^\lozenge(V)$ with $V\in\mathfrak{S}$, the projections $\overline{\overline{\pi}}_{V'}\colon \phi(\mathcal{X})\to\mathcal{C}V'$ are defined to be $\phi^{\ast}_V\circ\pi_V\circ\phi^{-1}$, where $\phi^{-1}$ and $\pi_V$ are as before, and $\phi_V^*\colon \mathcal{C}V\to\mathcal{C}V'$ are the (uniform) quasi isometries provided by the hierarchically hyperbolic space $(\mathcal{X},\mathfrak{S})$. 

By Definition \ref{HHS_hieromorphism} we have that $\phi^{\ast}_V\circ\pi_V\asymp\pi_{V'}\circ\phi$, where $\pi_{V'}$ is the projection in the space $(\mathcal{X}',\mathfrak{S}')$. Therefore $\overline{\overline{\pi}}_{V'}\asymp\pi_{V'}\circ\phi\circ\phi^{-1}$, which uniformly coarsely coincides with $\pi_{V'}$, being $\phi$ and $\phi^{-1}$ quasi inverses of each other. Thus $(\phi(\mathcal{X}),\phi^{\lozenge}(\mathfrak{S}))$ is a hierarchically hyperbolic space, where we can take the projections to be $\pi_{V'}$ for all $V'\in\phi^\lozenge(\mathfrak{S})$, instead of $\overline{\overline{\pi}}_{V'}$.

From this point, the argument to prove that there exists $M>0$ such that 
\[d_{\mathcal{X}'}(\mathfrak{g}_{\textbf{F}_{S'}}\circ\mathfrak{g}_{\phi(\mathcal{X})}(z),z)\leq M,\qquad d_{\mathcal{X}'}(\mathfrak{g}_{\phi(\mathcal{X})}\circ\mathfrak{g}_{\textbf{F}_{S'}}(y),y)\leq M\qquad \forall\ z\in\textbf{F}_{S'},y\in\phi(\mathcal{X})\]
is exactly the same as the one used in the previous implication $5\Rightarrow 3$, and it is omitted.
\end{proof}

Theorem \ref{thmB} has several consequences. We start with the following, in the form of a remark:
\begin{remark} \label{remark_comparealso}
The combination theorem of Behrstock, Hagen, and Sisto \cite[Theorem 8.6]{BHS2} holds without the first part of their fourth hypothesis, that is
\begin{quote}
if $e$ is an edge of $T$ and $S_e$ is the $\sqsubseteq$-maximal element of $\mathfrak{S}_e$, then for all $V\in \mathfrak{S}_{e^\pm}$, the elements $V$ and $\phi_{e^\pm}^\lozenge(S_e)$ are not orthogonal in $\mathfrak{S}_{e^\pm}$.
\end{quote}
Indeed, this hypothesis is used (compare \cite[Definition 8.23]{BHS2}) to define
the uniformly bounded sets $\rho^{[W]}_{[V]}$ when $[W]$ and $[V]$ are transverse equivalence classes whose supports do not intersect. By Theorem~\ref{thmB}, instead of
defining 
\[\rho^{[W]}_{[V]}=\mathfrak{c}_V\circ\rho_{V_v}^{\phi^{\lozenge}_{e^+}(S)}\]
as done in \cite[Definition 8.23]{BHS2}, we can impose that
\[\rho^{[W]}_{[V]}=\mathfrak{c}_V\bigl(\pi_{V_{e^+}}(\phi_{e^+}(\mathcal{X}_e))\bigr),\]
where $e$ is the last edge in the geodesic connecting $T_{[W]}$ to $T_{[V]}$, with $e^+\in T_{[V]}$, and $\mathfrak{c}_V$ is the comparison map from $\mathcal{C}V_{e^+}$ to the favorite representative of $[V]$. We will exploit this fact in the proof of Theorem \ref{mainT} (compare Subsection \ref{proj_hyp} and Equation \eqref{new_rhos}).
The proof of \cite[Theorem 8.6]{BHS2}, after this modification, is not altered.
\end{remark}

\begin{lemma}\label{lemma_concreteness}
Let $\phi\colon (\mathcal{X},\mathfrak{S})\to(\mathcal{X}',\mathfrak{S}')$ be a full, coarsely lipschitz hieromorphism between hierarchically hyperbolic spaces such that $\phi(\mathcal{X})$ is hierarchically quasiconvex in $\mathcal{X}'$, and let $S$ be the $\sqsubseteq$-maximal element of~$\mathfrak{S}$. 

There exist $\varepsilon$ and $\varepsilon_0$ such that for all $\varepsilon'\geqslant\varepsilon_0$, if $(\mathcal{X},\mathfrak{S})$ is $\varepsilon'$-concrete, with the intersection property and clean containers, then for any element $W\in\mathfrak{S}'$ we have that $W\perp \supp_{\varepsilon}\bigl(\phi(\mathcal{X})\bigr)$ if and only if $W\perp \phi^\lozenge(S)$.

\begin{proof}
Let $\varepsilon>\max\{3\alpha,3\xi,\mu\}$, where $\mu$ is the uniform bound given by Theorem \ref{thmB} on the diameters of $\pi_U\bigl(\phi(\mathcal{X})\bigr)$ for all $U\in\mathfrak{S}'\setminus\phi^\lozenge(\mathfrak{S})$, and $\varepsilon_0$ and $\varepsilon'$ be as in Lemma \ref{hiero_supports}.
Suppose that $W\perp \phi^\lozenge (S)$, so that $W\perp \mathfrak{S}'_{\phi^\lozenge(S)}$.
By the choice of $\varepsilon$ and by Theorem \ref{thmB}, we have that $\supp_\varepsilon(\phi(\mathcal{X}))\subseteq \mathfrak{S}'_{\phi^\lozenge(S)}$, because the hieromorphism if full, coarsely lipschitz, and with hierarchically quasiconvex image. Thus $W\perp \supp_{\varepsilon} \bigl(\phi(\mathcal{X})\bigr)$.

\smallskip
Assume now that $W\perp \supp_\varepsilon\bigl(\phi(\mathcal{X})\bigr)$.
As the hierarchically hyperbolic space $(\mathcal{X},\mathfrak{S})$ is $\varepsilon'$-concrete, we have that $S=\bigvee\supp_{\varepsilon'}(\mathcal{X})$, and therefore
\begin{equation}\label{eq1:lemma_concreteness}
\phi^\lozenge(S)=\phi^\lozenge\Bigl(\bigvee\supp_{\varepsilon'}(\mathcal{X})\Bigr).
\end{equation}
The hieromorphism $\phi$ is full and $\bigl(\mathcal{X},\mathfrak{S}\bigr)$ satisfies the intersection property, therefore by Lemma \ref{lemma:intersection_wedge_commute} and Equation~\eqref{eq1:lemma_concreteness} we obtain that
\begin{equation}\label{eq2:lemma_concreteness}
\phi^\lozenge(S)=\bigvee\phi^\lozenge\bigl(\supp_{\varepsilon'}(\mathcal{X})\bigr),
\end{equation}
and by Lemma \ref{hiero_supports} we have that
\begin{equation}\label{eq3:lemma_concreteness}
\phi^{\lozenge}(\supp_{\varepsilon'}(\mathcal{X}))\subseteq\supp_{\varepsilon}(\phi(\mathcal{X})).
\end{equation}
Combining Equation \eqref{eq2:lemma_concreteness} and Equation \eqref{eq3:lemma_concreteness}, we conclude that $\phi^\lozenge(S)\sqsubseteq 
\bigvee\supp_{\varepsilon}(\phi(\mathcal{X}))$.
As $W\perp \supp_{\varepsilon'}(\phi(\mathcal{X}))$, by clean containers and Lemma \ref{clean_containers_use} it follows that
$W\perp \bigvee \supp_{\varepsilon}(\phi(\mathcal{X}))$. Therefore
$W\perp \phi^\lozenge(S)$.
\end{proof}
\end{lemma}

\begin{theorem}\label{winning_theorem}
Let $\phi\colon(\mathcal{X},\mathfrak{S})\to(\mathcal{X}',\mathfrak{S}')$ be a full, coarsely lipschitz hieromorphism with hierarchically quasiconvex image, and assume that $\mathcal{X}$ is unbounded and concrete. 
There exists a constant $\eta\geqslant 0$, depending only on the hierarchical structures and the hieromorphism $\phi$, such that $d_{\mathcal{X}'}({\bf F}_{S'},\phi(\mathcal{X}))\leqslant\eta$, where $S'=\phi^\lozenge(S)$ and $S$ is the $\sqsubseteq$-maximal element of $\mathfrak{S}$.
\end{theorem}
\begin{proof}

Let $\kappa_0$ and $E$ be the constants coming from the hierarchically hyperbolic space $\mathcal{X}'$,
and let $\mu$ be the uniform constant on the diameters of the sets $\pi_W(\phi(\mathcal{X}))$, for all $W\in\mathfrak{S}'\setminus\phi^{\lozenge}
(\mathfrak{S})$, provided by Theorem \ref{thmB}. 

Let $V'\in\supp(\phi(\mathcal{X}))$, take $\kappa$ such that 
\begin{equation*}
\kappa >\max\{2\kappa_0,2E,E+\mu\}
\end{equation*}
and consider $x,y\in\mathcal{X}$ for which
\begin{equation}\label{thm3.6_eq2}
d_{V'}(\pi_{V'}(\phi(x)),\pi_{V'}(\phi(y)))>2\kappa.
\end{equation}
Let $W\in\mathfrak{S}' \setminus\phi^{\lozenge}(\mathfrak{S})$ be such that either $V'\pitchfork W$ or $V'\sqsubseteq W$. We claim that either
\begin{equation}\label{thm3.6_eq1}
d_W\left(\pi_W(\phi(x)),\rho_W^{V'}\right)\leq 2\kappa\quad\text{ or }\quad d_W\left(\pi_W(\phi(y)),\rho_W^{V'}\right)\leq 2\kappa.
\end{equation}
Indeed, assume that Equation \eqref{thm3.6_eq1} is not satisfied and that $W\pitchfork V'$. By consistency, as $2\kappa>\kappa_0$, we have that
\[d_{V'}\left(\pi_{V'}(\phi(x)),\rho_{V'}^{W}\right)\leq\kappa_0\quad\text{ and }\quad d_{V'}\left(\pi_{V'}(\phi(y)),\rho_{V'}^{W}\right)\leq\kappa_0.\]
This leads to a contradiction with Equation \eqref{thm3.6_eq2}.
	
Assume now that $V'\sqsubseteq W$. Again by consistency, we have that
\[\diam_{\mathcal{C}V'}\left(\pi_{V'}(\phi(x))\cup\rho_{V'}^{W}(\pi_W(\phi(x)))\right)\leq\kappa_0\quad\text{and}\quad \diam_{\mathcal{C}V'}\left(\pi_{V'}
(\phi(y))\cup\rho_{V'}^{W}(\pi_W(\phi(y)))\right)\leq\kappa_0.\] 
Let $\sigma$ be the geodesic in $\mathcal{C}W$ with endpoints $\pi_W(\phi(x))$ and $\pi_W(\phi(y))$. By the Bounded Geodesic Axiom there are two possibilities: 
\begin{enumerate}
	\item $\diam_{\mathcal{C}V'}\left(\rho_{V'}^W(\sigma)\right)\leq E$, \quad or
	\item $\sigma\cap\mathcal{N}_E(\rho_W^{V'})\neq\emptyset.$
\end{enumerate}
In the first case, applying the triangle inequality we conclude that 
\[d_{V'}(\pi_{V'}(\phi(x)),\pi_{V'}(\phi(y))\leqslant \kappa_0+E+\kappa_0=2\kappa_0+E\leqslant 2\kappa,\]
which contradicts Equation \eqref{thm3.6_eq2}.
	
For the second case, since $W\in\mathfrak{S}'\setminus\phi^{\lozenge}(\mathfrak{S})$ we know that $\pi_W(\phi(\mathcal{X}))$
is bounded by the uniform constant $\mu$. This means that $d_W\bigl(\pi_W(\phi(x)),\pi_W(\phi(y))\bigr)\leq\mu$. Furthermore, since there exists $z\in\sigma$ such that 
$d_W\left(z,\rho_W^{V'}\right)\leqslant E$, using the triangle inequality we have that
\[d_W\left(\pi_W(\phi(x)),\rho_W^{V'}\right)\leqslant E+\mu\quad\text{ and }\quad d_W\left(\pi_W(\phi(y)),\rho_W^{V'}\right)\leqslant E+\mu.\]
Using the triangle inequality we obtain that 
\[d_W\bigl(\pi_W(\phi(x)),\pi_W(\phi(y))\bigr)\leqslant d_W\left(\pi_W(\phi(x)),\rho_W^{V'}\right)+d_W\left(\pi_W(\phi(y)),\rho_W^{V'}\right)\leqslant 2(E+\mu)<2\kappa,\]
contradicting the assumption that the conditions in Equation~\eqref{thm3.6_eq1} are not satisfied. Therefore, Equation~\eqref{thm3.6_eq1} follows.

We have shown that for every $V'\in\supp_{2\kappa}(\phi(\mathcal{X}))$ and every $W\in\mathfrak{S}'\setminus\phi^{\lozenge}(\mathfrak{S})$ such that $W\sqsupseteq V'$ or $W\pitchfork V'$ we have that 
$d_W(\pi_W(\phi(\mathcal{X}),\rho_W^{V'})\leq 2\kappa$. For $S'=\phi^{\lozenge}(S)$, let $U\in\mathfrak{S}'$ be such that $U\sqsupseteq S'$ or $U\pitchfork S'$ (in particular, $ U\in\mathfrak{S}'\setminus\phi^{\lozenge}(\mathfrak{S}))$. 
By Lemma \ref{lemma_concreteness}, there exists $V'\in\supp_{2\kappa}(\phi(\mathcal{X}))$ for which $U\not\perp V'$.
Since $U\not\sqsubseteq S'$ and $V'\sqsubseteq S'$, it follows that $U\not\sqsubseteq V'$. 
Therefore, either $U\sqsupseteq V'$ or $U\pitchfork V'$, and by the above argument $d_U(\pi_U(\phi(\mathcal{X})),\rho_U^{V'})\leq 2\kappa$. Since $d_U(\rho_U^{S'},\rho_U^{V'})\leq\kappa_0$, it follows that $d_U(\rho_U^{S'},\pi_U(\phi(\mathcal{X})))\leq 3\kappa$.

We now claim that there exists some constant $\nu'$ such that $d_{\mathcal{X}'}\left({\bf F}_{S'},\phi(\mathcal{X})\right)\leq\nu'$. Fix $x_0\in\mathcal{X}$, and let 
$z\in{\bf F}_{S'}$ be the realization point of the consistent tuple 
\begin{equation*}
\begin{cases}
\pi_U(\phi(x_0)),&\quad \forall\  U\in\mathfrak{S}_{S'};\\
\pi_U(\phi(x_0)),& \quad\forall\ U\in\mathfrak{S}^{\perp}_{S'};\\
\rho_U^{S'}&\quad\forall\ U\pitchfork S' \text{ or } U\sqsupseteq S'.
\end{cases}
\end{equation*}
By the above argument and the choice of the realization point $z$, if follows that the distance $d_U(\pi_U(z),\pi_U(\phi(\mathcal{X}))$ is uniformly bounded, for all $U\in\mathfrak{S}'$. 
Since $\phi(\mathcal{X})$ is a hierarchical quasiconvex subspace of $\mathcal{X}'$, there exists a constant $\eta$ depending only  on the hierarchically hyperbolic structure of $(\mathcal{X}',\mathfrak{S}')$ for which $d_{\mathcal{X}'}(z,\phi(\mathcal{X}))\leq\eta$. Therefore $d_{\mathcal{X}'}({\bf F}_{S'},\phi(\mathcal{X}))\leq\eta$ and the proof is complete.

\end{proof}

From the previous theorem, we obtain the following lemmas:
\begin{lemma}\label{consequence_winning_theorem}
Let $\phi\colon (\mathcal{X},\mathfrak{S})\to(\mathcal{X}',\mathfrak{S}')$ be a full, coarsely lipschitz hieromoprhism, assume that $\mathcal{X}$ 
is unbounded and concrete, and let $S'=\phi^\lozenge (S)$, where $S\in\mathfrak{S}$ is the $\sqsubseteq$-maximal element. For all $U\in \mathfrak{S}'$
such that either $S'\sqsubsetneq U$ or $S'\pitchfork U$, the sets $\rho_{U}^{S'}$ and $\pi_{U}(\phi(\mathcal{X}))$ coarsely coincide.
\end{lemma}
\begin{proof}
For any $U\in\mathfrak{S}'$ such that either $S'\sqsubsetneq U$ or $S'\pitchfork U$, we have that $\pi_U({\bf F}_{S'})=\rho_U^{S'}$ by \cite[Construction 5.10]{BHS2}.  
Moreover, the distance $d_U(\pi_U({\bf F}_{S'}),\pi_U(\phi(\mathcal{X})))$ is at most $K\eta + K$
by Theorem \ref{winning_theorem}.

Since $\diam_{\mathcal{C}U}(\pi_U(\phi(\mathcal{X})))\leq\mu$ and $\diam_{\mathcal{C}U}(\rho_U^{S'})\leq\xi$, any pair of elements in the sets 
$\rho_U^{S'}=\pi_U(\phi(\mathcal{X}))$ and $\pi_U({\bf F}_{S'})$ is at uniform bounded distance from each other.

\end{proof}

\begin{lemma}\label{conclusion_fourth_hypothesis}
Let $\phi\colon (\mathcal{X},\mathfrak{S})\to(\mathcal{X}',\mathfrak{S}')$ be a full, coarsely lipschitz hieromoprhism, assume that $\mathcal{X}$ 
is unbounded and concrete, and let $S\in\mathfrak{S}$ be the $\sqsubseteq$-maximal element. There exists a constant $J$, depending only 
on the hierarchically hyperbolic structure $(\mathcal{X}',\mathfrak{S}')$ and hieromorphism constants, such that $d_{Haus}(\phi(\mathcal{X}),\textbf{F}_{\phi^\lozenge(S)})\leqslant J$.
\end{lemma}
\begin{proof}
By the third hypothesis of Theorem \ref{thmB}, there exist $\nu,\nu'$, depending only on the hierarchically hyperbolic structures, such that $d_{\mathcal{X}'}\bigl(\phi(x),\mathfrak{g}_{\phi(\mathcal{X})}\circ\mathfrak{g}_{\textbf{F}_{\phi^\lozenge(S)}}(\phi(x))\bigr)\leqslant\nu$ and $d_{\mathcal{X}'}\bigl(z,\mathfrak{g}_{\textbf{F}_{\phi^\lozenge(S)}}\circ\mathfrak{g}_{\phi(\mathcal{X})}(z)\bigr)\leqslant\nu'$, for all $\phi(x)\in\phi(\mathcal{X})$ and for all $z\in\textbf{F}_{\phi^\lozenge(S)}$. 

By \cite[Lemma 1.26]{quasiflats}, there exists a constant $\bar J$, depending on the hierarchically hyperbolic structure, such that 
\[d_{\mathcal{X}'}(\phi(\mathcal{X}),\textbf{F}_{\phi^\lozenge(S)})\asymp_{(\bar J,\bar J)} d_{\mathcal{X}'}\bigl(\mathfrak{g}_{\textbf{F}_{\phi^\lozenge(S)}}(\phi(x)),\mathfrak{g}_{\phi(\mathcal{X})}
\circ\mathfrak{g}_{\textbf{F}_{\phi^\lozenge(S)}}(\phi(x))\bigr)\] 
for all $\phi(x)\in\phi(\mathcal{X})$. 
Furthermore, if $\eta$ denotes the bound of Theorem \ref{winning_theorem}, using the previous equation we obtain that
\begin{equation*}
\begin{aligned}
d_{\mathcal{X}'}\bigl(\phi(x),\mathfrak{g}_{\textbf{F}_{\phi^\lozenge(S)}}(\phi(x))\bigr)&\leq d_{\mathcal{X}'}\bigl(\phi(x),\mathfrak{g}_{\phi(\mathcal{X})}\circ\mathfrak{g}_{\textbf{F}_{\phi^\lozenge(S)}}(\phi(x))\bigr) + d_{\mathcal{X}'}\bigl(\mathfrak{g}_{\phi(\mathcal{X})}\circ\mathfrak{g}_{\textbf{F}_{\phi^\lozenge(S)}}(\phi(x)),\mathfrak{g}_{\textbf{F}_{\phi^\lozenge(S)}}(\phi(x))\bigr)\\
&\leqslant\nu + \bar J d_{\mathcal{X}'}(\phi(\mathcal{X}),\textbf{F}_{\phi^\lozenge(S)})+\bar J \leqslant \nu + \bar J\eta+\bar J=:J
\end{aligned}
\end{equation*} 
for all $\phi(x)\in\phi(\mathcal{X})$. 

In an analogous manner, we obtain that $d_{\mathcal{X}'}(z,\mathfrak{g}_{\phi(\mathcal{X})}(z))\leq J$ for all $z\in\textbf{F}_{\phi^\lozenge(S)}$. Thus, the bound on the Hausdorff distance is proved.
\end{proof}

\section{Combination Theorem}\label{Section5}
In this section we prove Theorem \ref{mainT} of the Introduction:
\begin{thmA}
Let $\mathcal{T}$ be a tree of hierarchically hyperbolic spaces. Suppose that: 
\begin{enumerate}
\item each edge-hieromorphism is hierarchically quasiconvex, uniformly coarsely lipschitz and full;
\item comparison maps are uniform quasi isometries;
\item the hierarchically hyperbolic spaces of $\mathcal{T}$ have the intersection property and clean containers.
\end{enumerate} 
Then the metric space $\mathcal{X}(\mathcal{T})$ associated to $\mathcal{T}$ is a hierarchically hyperbolic space with clean containers and the intersection property.
\end{thmA}

We recall that $\mathcal{X}(\mathcal{T})$ denotes the metric space defined in Definition \ref{equiv_class_def}.

\begin{remark}
Comparison maps being uniform quasi isometries is, in particular, a consequence of uniformly bounded supports (see \cite[Lemma 8.20]{BHS2}), which is a hypothesis in the original Combination Theorem of Behrstock, Hagen, and Sisto \cite[Theorem 8.6]{BHS2}. 

\smallskip
The second part of the fourth hypothesis of \cite[Theorem 8.6]{BHS2} 
\begin{quote}
there exists $K\geq 0$ such that for all vertices $v$ of $\mathcal{T}$ and edges $e$ incident to $v$, we have $d_{Haus}(\phi_v(\mathcal{X}_e),\textbf{F}_{\phi^{\lozenge}_v(S_e)}\times\{\star\})\leq K$, where $S_e\in\mathfrak{S}_e$ is the unique maximal element and $\star\in~\textbf{E}_{\phi_v^{\lozenge}(S_e)}$
\end{quote}
is a consequence of Theorem \ref{winning_theorem} through Lemma \ref{conclusion_fourth_hypothesis}, and therefore is automatically satisfied in our setting.

\smallskip
Hypothesis \ref{hyp3} of Theorem \ref{mainT} cannot be further relaxed, or dropped. A counterexample of Theorem \ref{mainT} without the second hypothesis is given by Bass-Serre trees of Baumslag-Solitar groups. Indeed, non-abelian 
Baumslag-Solitar groups are HNN extensions $\Z\ast_{\Z}$, that is graph of groups of hierarchically hyperbolic groups, and their Dehn function is not quadratic~\cite[Theorem~B]{G92}. Therefore, they are not hierarchically hyperbolic, because hierarchically hyperbolic groups have quadratic Dehn function \cite[Corollary~7.5]{BHS2}.
\end{remark}

\begin{example}[Baumslag-Solitar groups]
Let us consider more in detail non-amenable Baumslag-Solitar groups $\BS(1,k)=\langle a,t\mid tat^{-1}=a^k\rangle$, where $k\neq \pm 1$. Let $T=(V,E)$ be the Bass-Serre tree associated to the HNN extension $\BS(1,k)$, so that $V=\{g\langle a\rangle\mid g\in \BS(1,k)\}$. Two distinct vertices $g\langle a\rangle$ and $h\langle a\rangle$ are joined by an edge $e\in E$ if and only if there exists $b\in \langle a\rangle$ such that either $h\langle a\rangle=gbt^{\pm 1}\langle a\rangle$, or $h\langle a\rangle=gbt^{-1}\langle a\rangle$. For a vertex $g\langle a\rangle=v \in V$ let $\bigl(\mathcal{X}_v,\mathfrak{S}_v\bigr):=\bigl(g\langle a\rangle,\{\langle a\rangle\}\bigr)$ be the hierarchically hyperbolic space associated to the vertex, and for any edge $e\in E$ let $\bigl(\mathcal{X}_e,\mathfrak{S}_e\bigr):=\bigl(\langle a\rangle,\{\langle a\rangle\} \bigr)$ be the hierarchically hyperbolic space associated to the edge.
Given $\{g\langle a\rangle, h\langle a\rangle\}=e\in E$, consider the hieromorphisms $\phi_{e_+}\colon \bigl(\langle a\rangle,\{\langle a\rangle\} \bigr)\to \bigl(g\langle a\rangle,\{\langle a\rangle\}\bigr)$ be defined as $\phi_{e_+}(a)=ga$, and $\phi_{e_-}\colon \bigl(\langle a\rangle,\{\langle a\rangle\} \bigr)\to \bigl(h\langle a\rangle,\{\langle a\rangle\}\bigr)$ be defined as $\phi_{e_-}(a)=ha^k$.

We have that
\[\mathcal{T}=\Bigl(T,\bigl\{\bigl(\bigl(\mathcal{X}_{g\langle a\rangle},\mathfrak{S}_{g\langle a\rangle}\bigr)\bigr)\bigr\}_{g\in G},\bigl\{\bigl(\mathcal{X}_e,\mathfrak{S}_e\bigr)\}\bigr)\bigr\}_{e\in E}, \bigl\{\phi_{e_\pm}\bigr\}_{e\in E}\Bigr)\]
is a tree of hierarchically hyperbolic spaces. The vertex-spaces and edge-spaces all have the intersection property and clean containers, because their index set consists of only one element. Moreover, hieromorphisms are hierarchically quasiconvex, uniformly coarsely lipschitz, and full.

Let us prove that comparison maps are not uniform quasi isometries. First notice that, as each hierarchically hyperbolic space has an index set of cardinality one, there is only equivalence class that spans the whole tree~$\mathcal{T}$. Let $v$ and $u$ be two vertices in $T$, at distance $d$. Then, the comparison map $\mathfrak{c}_{v\to u}\colon \langle a \rangle\to \langle a\rangle$ is a $(\lvert k\rvert^d,0)$-lipschitz map. Therefore, as $\lvert k\rvert>1$ and we cannot bound the distance $d$ between two vertices in the unbounded tree $T$, comparison maps cannot be uniform quasi isometries, as claimed.
\end{example}

\begin{remark}[{\bf Comparison between index sets}]
Before proving Theorem \ref{mainT} (and therefore before describing the index set $\mathfrak{S}$ for the tree of hierarchically hyperbolic spaces $\mathcal{T}$, compare Section \ref{nes_ort_tra}), we wish to compare the index set that we are using with the index set of \cite{BHS2}, for the familiar readers.

Although our approach might appear more complicated than the one of Berhstock, Hagen, and Sisto \cite{BHS2}, we want to emphasize that the index set we consider and use is more natural than (and a generalization of) the one constructed in \cite{BHS2}. Indeed, in \cite[Definition 8.11]{BHS2}, the index set of Behrstock, Hagen, and Sisto is defined as the inductive closure of the set $\mathfrak{S}_0=\{T\}\cup \bigl(\bigsqcup_{v\in V}\mathfrak{S}_v\bigr)/\sim$, with respect of adding orthogonal containers (compare the sets $\mathcal{K}_\eta$ in that definition).

On the other hand, our index set is defined (compare Equation \eqref{newfactorsystem}) as 
\[\{T\}\cup \Bigl(\bigl(\bigsqcup_{v\in V}\mathfrak{S}_v\bigr)/\sim\Bigr)\cup 
\Bigl\{T_{[V]}\mid [V]\in \bigl(\bigsqcup_{v\in V}\mathfrak{S}_v\bigr)/\sim\Bigr\},\]
because, allowing infinitely-supported equivalence classes, we are forced to add their supports to our index set. This turns out to be a very reasonable addition because, among other things, it ``fixes'' orthogonal containers, which no longer need to be added manually.

Conversely, if the support trees were uniformly bounded (as assumed in \cite{BHS2}), then the support trees we are considering would correspond to the artificial containers defined in \cite[Definition 8.11]{BHS2} (the ones with uniformly bounded attached hyperbolic spaces, and the others with trivial hyperbolic spaces), and  therefore the two constructions would coincide.
\end{remark}

Before going into the proof of Theorem \ref{mainT}, we state an immediate consequence of it:

\begin{corollary}
Let $M$ be a closed $3$-manifold that does not have a $Sol$ or $Nil$ component in its prime decomposition. Then $\pi_1(M)$ is a hierarchically hyperbolic space with the intersection property and clean containers.
\end{corollary}
\begin{proof}
The fact that $\pi_1(M)$ is a hierarchically hyperbolic space is proved in \cite[Theorem 10.1]{BHS2}, and it has clean containers by \cite[Proposition 3.5]{ABD}. Hierarchical hyperbolicity is proved, for closed non-geometric \emph{irreducible} $3$-manifolds, by constructing a tree of hierarchically hyperbolic spaces where supports are uniformly bounded (therefore, by \cite[Lemma 8.20]{BHS2}, comparison maps are uniform quasi isometries), and where vertex spaces are direct products $\R_v\times \Sigma_v$. Here, $\R_v$ is a copy of the real line, and $\Sigma_v$ is the universal cover of a hyperbolic surface with totally geodesic boundary, whose hierarchically hyperbolic structure originates from the fact that $\Sigma_v$ is hyperbolic relative to its boundary components. 
Edge spaces are $\R_v\times \partial_0\Sigma_v$, where $\partial_0\Sigma_v$ is a particular boundary component of $\Sigma_v$.

By Lemma \ref{directproduct_intersection}, the spaces $\Sigma_v$, $\R_v\times \partial_0\Sigma_v$ and $\R_v\times \Sigma_v$ are hierarchically hyperbolic spaces with the intersection property, and they all have clean containers by \cite[Section 3]{ABD}. Moreover, as seen in \cite[Theorem 10.1]{BHS2}, the hieromorphisms are coarsely lipschitz, full, and with hierarchically quasiconvex images. Therefore all the hypotheses of Theorem \ref{mainT} are satisfied, thus proving that $\pi_1(M)$ is a hierarchically hyperbolic space with the intersection property and clean containers.

Geometric irreducible $3$-manifolds have fundamental groups that are quasi isometric to direct products of hyperbolic groups, and therefore they have the intersection property by Lemma \ref{directproduct_intersection}. Finally, the fundamental group of any reducible $3$-manifold is the free product of the fundamental groups of irreducible $3$-manifolds, and therefore it has the intersection property by what just proved, and Lemma \ref{directproduct_intersection}.
\end{proof}

\subsection{Trees with decorations}\label{trees_with_decorations}
Recall that a tree of hierarchically hyperbolic spaces (as defined in Definition~\ref{equiv_class_def}) is a tuple
\begin{equation}\label{how_to_decorate}
\mathcal{T}=\Bigl(T,\{(\mathcal{X}_v,\mathfrak{S}_v)\}_{v\in V},\{(\mathcal{X}_e,\mathfrak{S}_e)\}_{e\in E}\bigr),\{\phi_{e^\pm}\colon (\mathcal{X}_e,\mathfrak{S}_e) \to(\mathcal{X}_{e^\pm},\mathfrak{S}_{e^\pm})\}\Bigr),
\end{equation}
where $T=(V,E)$ is a tree, $\{(\mathcal{X}_{v},\mathfrak{S}_{v})\}_{v\in V}\}$ and $\{(\mathcal{X}_{e},\mathfrak{S}_{e})\}_{e\in E}\}$ are families of uniformly hierarchically hyperbolic spaces, and $\phi_{e^+}\colon (\mathcal{X}_{e},\mathfrak{S}_{e})\to (\mathcal{X}_{e^+},\mathfrak{S}_{e^+})$ and
$\phi_{e^-}\colon (\mathcal{X}_{e},\mathfrak{S}_{e})\to (\mathcal{X}_{e^-},\mathfrak{S}_{e^-})$ are hieromorphisms with constants all bounded uniformly.

On $\bigsqcup_{v\in V}\mathfrak{S}_v$ one defines the following equivalence class: given an edge $e=\{v,w\}\in E$ and $U\in \mathfrak{S}_e$, impose $\phi_{v}^\lozenge(U)$ to be equivalent to $\phi_w^\lozenge(U)$, and take the transitive closure of this to obtain the desired equivalence relation. Given $U\in \bigsqcup_{v\in V}\mathfrak{S}_v$, its equivalence class is denoted by $[U]$.

\smallskip
In general, in a tree of hierarchically hyperbolic spaces $\mathcal{T}$ it might happen that two distinct equivalence classes $[U]\neq [V]$ are supported on exactly the same vertices of the tree $T$, that is $T_{[U]}=T_{[V]}$.
This is not desirable, and in this subsection we describe a slight modification of the tree $\mathcal{T}$ (and therefore of the metric space $\mathcal{X}(\mathcal{T})$ associated to it) that ensures that $[U]=[V]$ if and only if $T_{[U]}=T_{[V]}$. 
We achieve this by attaching to each vertex $v$ of $T$ a tree of uniformly bounded diameter, 
and refer to these attached trees as \emph{decorations}. We denote the tree that is obtained with this process by $\widetilde{T}$. As a consequence, the new support trees $\widetilde{T}_{[U]}$ will become larger than the original ones (i.e. $T_{[U]}\subseteq \widetilde{T}_{[U]}$ for each equivalence class $[U]$).

All the hypotheses of Theorem \ref{mainT} are preserved by adding these decorated trees (furthermore, the metric spaces associated to the two trees of hierarchically hyperbolic spaces are quasi-isometric), and therefore for the proof of the theorem we will assume without loss of generality that equivalence classes are discriminated by their supports.

\smallskip
We now describe how to decorate the tree $\mathcal{T}$ of hierarchically hyperbolic spaces of Equation \eqref{how_to_decorate}, to ensure that $[U]=[V]$ if and only if $T_{[U]}=T_{[V]}$.

For any vertex $v\in T$, let $S_v$ be the $\sqsubseteq$-maximal element in $\mathfrak{S}_v$, let $U$ be any $\sqsubseteq$-maximal element of $\mathfrak{S}_v\setminus\{S_v\}$ and let ${\textbf{F}_U\times\lbrace f\rbrace}$ be a parallel copy of the $\textbf{F}_U$ inside of $\mathcal{X}_v$. For any such choice, we add a new vertex $\tilde{v}$
and a new edge $\tilde e$ connecting $v$ and $\tilde v$.
The metric spaces $\mathcal{X}_{\tilde{v}}$ and $\mathcal{X}_{\tilde e}$ are defined to be ${\textbf{F}_U}\times\lbrace f\rbrace$, with the induced metric. 

It follows from \cite[Proposition 5.11]{BHS2} that $\bigl(\mathcal{X}_{\widetilde{v}},\mathfrak{S}_U\bigr)$ and $\bigl(\mathcal{X}_{\widetilde{e}},\mathfrak{S}_U\bigr)$ are hierarchically hyperbolic spaces, of complexity strictly lower than $\bigl(\mathcal{X}_v,\mathfrak{S}_v\bigr)$. We refer to these
index sets as $\mathfrak{S}_{\widetilde v}^{U,f}$ and $\mathfrak{S}_{\widetilde e}^{U,f}$ respectively, where the exponent is added to keep track of the choices of the $\sqsubseteq$-maximal element $U\in\mathfrak{S}_v\setminus\{S_v\}$, and of the parallel copy ${\bf F}_U\times \{f\}$.

The hieromorphisms $\phi_{\widetilde e^+}$ and $\phi_{\widetilde e^-}$ are defined as follows. At the level of metric spaces, $\phi_{\widetilde{e}^+}\colon \mathcal{X}_{\widetilde{e}}\to\mathcal{X}_{\widetilde{v}}$ is the identity map and $\phi_{\widetilde{e}^-}\colon \mathcal{X}_{\widetilde{e}}\to\mathcal{X}_{v}$ is the subspace inclusion. The map $\phi^{\lozenge}_{\widetilde{e}^+}:\mathfrak{S}_{\widetilde{e}}^{U,f}\to\mathfrak{S}_{\widetilde{v}}^{U,f}$ is the identity of the set $\mathfrak{S}_U$, and $\phi_{\widetilde{e}^-}^{\lozenge}\colon \mathfrak{S}_{\widetilde{e}}^{U,f}\to\mathfrak{S}_v$ is the inclusion. At the level of hyperbolic spaces, the maps $\phi_{\widetilde{e}^-,W}^*,\phi_{\widetilde{e}^+,W}^*\colon\mathcal{C}W\to\mathcal{C}W$ are the identity for each $W\in\mathfrak{S}_{\widetilde{e}}^{U,f}$. It is straightforward to check that the commutative diagrams of Definition~\ref{HHS_hieromorphism} are satisfied. Furthermore, since $\phi_{\widetilde{e}^+}^{\lozenge},\phi_{\widetilde{e}^-}^{\lozenge}$ and $\phi_{\widetilde{e}^+,W}^*,\phi_{\widetilde{e}^-,W}^*$ are identity maps or inclusions, it follows that $\phi_{\widetilde{e}^+}$ and $\phi_{\widetilde{e}^-}$ are full hieromorphism. Moreover, they are quasiconvex.

We repeat this process for any newly produced vertex, until the complexity of the resulting hierarchically hyperbolic spaces is one. 
In particular, given a new vertex $\widetilde v$ with associated hierarchically hyperbolic space $\bigl({\bf F}_U\times\{f\},\mathfrak{S}_U\bigr)$ not of complexity one, consider a $\sqsubseteq$-maximal element $V\in\mathfrak{S}_U\setminus\{U\}$. Consider moreover a parallel copy ${\bf F}_V\times \{f'\}$ of ${\bf F}_V$ in ${\bf F}_U\times \{f\}$, and repeat the process to construct a new vertex with associated hierarchically hyperbolic space $\bigl({\bf F}_V\times \{f'\},\mathfrak{S}_V\bigr)$. We stress that ${\bf F}_V$ is defined \emph{in} the hierarchically hyperbolic space $\bigl({\bf F}_U\times\{f\},\mathfrak{S}_U\bigr)$, and not in the space $\bigl(\mathcal{X}_v,\mathfrak{S}_v\bigr)$ for which $U\in\mathfrak{S}_v$.

We denote by $\widetilde{\mathcal{T}}$ the tree of hierarchically hyperbolic spaces obtained from $\mathcal{T}$ following this process. Notice that $\mathcal{X}(\mathcal{T})$ can be 
naturally seen as a subspace of $\mathcal{X}(\widetilde{\mathcal{T}})$, that is 
$\mathcal{X}(\mathcal{T})\subseteq \mathcal{X}(\widetilde{\mathcal{T}})$.
Moreover, as the complexity of the hierarchically hyperbolic spaces of $\mathcal{T}$ is uniformly bounded and each step of the described process reduces the complexity by one, there exists a uniform constant $C$ such that $\mathcal{N}_C\bigl(\mathcal{X}(\mathcal{T})\bigr)=\mathcal{X}(\widetilde{\mathcal{T}})$. In particular, the inclusion map $\iota\colon\mathcal{X}(\mathcal{T})\hookrightarrow \mathcal{X}(\widetilde{\mathcal{T}})$ is a quasi isometry, and therefore the two spaces $\mathcal{X}(\mathcal{T})$ and $\mathcal{X}(\widetilde{\mathcal{T}})$ are quasi isometric.

In $\mathcal{X}(\widetilde{\mathcal{T}})$, we denote by $\sim_\star$ the equivalence relation described in Subsection \ref{section_trees}, by $[U]_{\star}$ the equivalence class of $U\in\bigsqcup_{\widetilde v\in\widetilde V}\mathfrak{S}_{\widetilde v}$ with respect to $\sim_\star$, and by $\widetilde T_{[U]_\star}$ the support of $[U]_\star$. Notice that $\widetilde T_{[U]_\star}\cap T=T_{[U]}$ for all $U\in \bigsqcup_{v\in V}\mathfrak{S}_{v}$, and that for all $\widetilde V\in \bigsqcup_{\widetilde v\in\widetilde V\setminus V}\mathfrak{S}_{\widetilde{v}}$
there exists $V\in \bigsqcup_{ v\in V}\mathfrak{S}_v$ such that $\widetilde V\sim_\star V$.

\begin{remark}
In the context of hierarchically hyperbolic groups, decorating a tree $\mathcal{T}$ amounts to the following.
Let $v$ be a vertex in $\mathcal{T}$ with associated group $G$, and consider the Bass-Serre tree of $G\ast_H H$, where $H$ is a hierarchically quasiconvex subgroup of $G$ of maximal, strictly smaller complexity, and the two edge-embeddings are given by the identity map
$\id_H\colon H\to H$ and by the inclusion $\iota\colon H\to G$.
This Bass-Serre tree has one vertex $v_0$
with associated group $G$, and $[G:H]$ vertices $v_i$ whose associated groups are the $G$-cosets of the subgroup $H$, and edges $e_i$ connecting $v_0$ to $v_i$. 

In the tree $\mathcal{T}$, we replace the vertex $v$ by $v_0$, and we add new vertices $v_i$ and
edges $e_i$ connecting $v_0$ to $v_i$. To these new vertices $v_0$ and $v_i$, we associate the groups given by the Bass-Serre tree of the splitting $G\ast_H H$.

For any new vertex $v_i$ added in such way, we repeat the process unless the vertex group $H$ has complexity~one.
\end{remark}

\begin{lemma}\label{lemma_decorations}
In the tree of hierarchically hyperbolic spaces $\widetilde{\mathcal{T}}$ we have that $[U]_\star=[V]_\star$ if and only if $\widetilde T_{[U]_\star}=\widetilde T_{[V]_\star}$.
\begin{proof}
One implication is trivial. Assume now that $\widetilde T_{[V]_\star}=\widetilde T_{[W]_\star}$. If the complexity of the two equivalence classes $[V]_\star$ and $[U]_\star$ is different, then the decorations added to the tree $T$ are trees of different diameter, and therefore we cannot have that $\widetilde T_{[V]_\star}=\widetilde T_{[W]_\star}$. Thus, the equivalence classes have the same complexity, so neither cannot be properly nested into the other. 

By construction, in the tree $\widetilde T$ there are vertices $\widetilde u$ and $\widetilde v$ such that $U$ and $V$ are $\sqsubseteq$-maximal elements of $\mathfrak{S}_{\widetilde u}$ and $\mathfrak{S}_{\widetilde v}$, respectively. 
As $\widetilde T_{[U]_\star}=\widetilde T_{[V]_\star}$, the equivalence class $[U]_\star$ must have a representative in $\mathfrak{S}_{\widetilde v}$, and $[V]_\star$ must have a representative in $\mathfrak{S}_{\widetilde u}$. As neither equivalence class can be properly nested into the other, it must then be that $[U]_\star=[V]_\star$.
\end{proof}
\end{lemma}

If the tree $\mathcal{T}$ satisfies the hypotheses of Theorem \ref{mainT}, then also $\widetilde{\mathcal{T}}$ does. We prove this in the following lemmas.

\begin{lemma}
In the tree of hierarchically hyperbolic spaces $\widetilde{\mathcal{T}}$ the edge hieromorphisms are full, coarsely lipschitz,  and hierarchically quasiconvex.
\begin{proof}
Let $e$ be an edge in $\widetilde T$. Two cases can occur: either $e$ is an edge already in the tree $T$, or it was added with the decoration of $T$.

If $e$ was already an edge in $T$, then the edge hieromorphisms are full, coarsely lipschitz, and hierarchically quasiconvex by the hypotheses of Theorem \ref{mainT}.
On the other hand, if $e$ is a new edge then the two maps $\phi_{e_-}$ and $\phi_{e_+}$ are full, hierarchically quasiconvex isometric embeddings (one is actually an isometry), by construction.

\end{proof}
\end{lemma}

\begin{lemma}
The hierarchically hyperbolic spaces of $\widetilde{\mathcal{T}}$ have the intersection property and clean containers.
\begin{proof}
Let $\widetilde v$ be a vertex of $\widetilde T$. If $\widetilde v\in T$ then $\mathfrak{S}_{\widetilde v}$ has the intersection property and clean containers, by the hypotheses of Theorem \ref{mainT}. If $\widetilde v\in\widetilde T\setminus T$, then $\mathfrak{S}_{\widetilde v}=\mathfrak{S}_{\widetilde v}^{U,f}$ coincides with $\mathfrak{S}_U$, for some $U\in\bigsqcup_{v\in V}\mathfrak{S}_v$. Therefore, $\mathfrak{S}_{\widetilde v}$ has in intersection property. Let $v\in T$ be the vertex such that $U\in\mathfrak{S}_v$.

Suppose that $\mathfrak{S}_{\widetilde v}=\mathfrak{S}_{\widetilde v}^{U,f}=\mathfrak{S}_U$ does not have clean containers. Therefore, there exists $W\in\mathfrak{S}_U\setminus\{U\}$ such that the set $\{Z\in\mathfrak{S}_U\mid Z\perp W\}$ is not empty, and $W\not\perp\cont_\perp^UW$. By Lemma \ref{wedge_container_lemma_blue} we know that $\cont_\perp^UW=U\wedge \cont_\perp W$, where $\cont_\perp W$ is the orthogonal container of $W$ in $\mathfrak{S}_v$. Moreover $W\perp \cont_\perp W$ by clean containers in $\mathfrak{S}_v$, and therefore we reach a contradiction, as $\cont_\perp^UW\sqsubseteq\cont_\perp U$. Thus, $\mathfrak{S}_{\widetilde v}^{U,f}$ has clean containers.

The argument for edge spaces is similar.
\end{proof}
\end{lemma}

\begin{lemma}
Comparison maps in $\widetilde{\mathcal{T}}$ are uniformly quasi-isometries.
\begin{proof}
Let $v,w$ be two vertices in $\widetilde T$ and let $[V]_\star$ be an equivalence class supported on both vertices, with representatives $V_v$ and $V_w$ respectively.
Consider the comparison map $\mathfrak{c}\colon \mathcal{C}V_v\to \mathcal{C}V_w$, as defined in Equation \eqref{definition_comparison}.
If both vertices already belong to $T\subseteq\widetilde T$, then the map $\mathfrak{c}$ is a uniform quasi-isometry by the hypotheses of Theorem \ref{mainT}.

If one vertex, say $w$, belongs in $\widetilde T\setminus T$, and $v\in T$, consider the geodesic $\sigma$ in $\widetilde T$ connecting $v$ to $w$. Let $v=v_0,\dots v_n=w$ be the vertices of $\sigma$, such that $v_i$ is joined by an edge to $v_{i+i}$ for all $i=0,\dots,n-1$.
Then, there exists a maximal index $i_\star$ such that $v_{i_\star}\in T$ and $v_{i_\star+1}\in\widetilde T\setminus T$; let $V_\star$ be the representative of $[V]$ in $\mathfrak{S}_{v_{i_\star}}$.
From Equation \eqref{definition_comparison} we see that $\mathfrak{c}$ is the composition of $\mathfrak{c}_1\colon \mathcal{C}V_v\to\mathcal{C}V_{v_{i_\star}}$
with $\mathfrak{c}_2\colon \mathcal{C} V_{v_{i_\star}}\to
\mathcal{C}V_w$. As noticed in the previous case, the map $\mathfrak{c}_1$ is a uniform quasi-isometry. Moreover, by construction, the map $\mathfrak{c}_2$ is an isometry, and therefore $\mathfrak{c}$ is a uniform quasi-isometry, being the composition of these two maps.

The last case to consider is when both vertices belong to $\widetilde T\setminus T$. Depending on whether the geodesic $\sigma$ does not intersect $T$, or does intersect it, the map $\mathfrak{c}$ will be an isometry, or a composition of three maps, two of which isometies and the remaining a uniform quasi isometry. 

Therefore, all comparison maps are uniform quasi isometries.
\end{proof}
\end{lemma}

In view of this, for the whole proof of Theorem \ref{mainT} we assume without loss of generality that equivalence classes are differentiated by their supports already in the tree of hierarchically hyperbolic space $\mathcal{T}$, that is 
$[U]=[V]$ if and only if $T_{[U]}=T_{[V]}$.

On the other hand, for the proof of Corollary \ref{mainC}, that is the application of Theorem \ref{mainT} to hierarchically hyperbolic groups, we will \emph{not} decorate the tree $\mathcal{T}$. This is because, even if a hierarchically hyperbolic group $\bigl(G,\mathfrak{S}\bigr)$ acts on the index set $\mathfrak{S}$, the set of product regions $\bigl\{{\bf F}_U\times\{f\}\mid U\in\mathfrak{S},\, f\in{\bf E}_U\bigr\}$ might not be $G$-invariant. Therefore, it might happen that the hierarchically hyperbolic space $(\mathcal{X}(\widetilde{\mathcal{T}}),\widetilde{\mathfrak{S}})$, where $\widetilde{\mathfrak{S}}$ denotes the index set associated to the decorated tree $\widetilde{\mathcal{T}}$, does not admit a non-trivial action of $G$ onto $\widetilde{\mathfrak{S}}$. We refer to Section 
\ref{subsection_graph_hhgs} for the complete treatment of this delicate point.

\smallskip
We now define the hierarchically hyperbolic structure on this tree of hierarchically hyperbolic spaces.

\subsection{Index set, nesting, orthogonality, and transversality}\label{nes_ort_tra}

\begin{remark}[{\bf Concreteness of the edge spaces}]\label{remark_crucial_concrete}
In the proof of Theorem \ref{mainT} we will need to exploit concreteness of the edge spaces, which is not an hypothesis of the theorem. We now explain why we can suppose, 
without loss of generality, that all the hierarchically hyperbolic edge-spaces of $\mathcal{T}$ are $\varepsilon$-concrete. 

Let $\varepsilon\geq 3\max\{\alpha,\xi\}$ as in Lemma \ref{lemma_inclusion1}. If the edge spaces are not all $\varepsilon$-concrete,
then we apply Proposition \ref{concreteness_making} to each edge space $\mathfrak{S}_e$ of $\mathcal{T}$ to obtain a sub-index set $\mathfrak{S}_{e,\varepsilon}\subseteq\mathfrak{S}_e$ such that $(\mathcal{X}_e,\mathfrak{S}_{e,\varepsilon})$ is $\varepsilon$-concrete. Notice that if $\mathfrak{S}_e$ is already $\varepsilon$-concrete, then $\mathfrak{S}_{e,\varepsilon}=\mathfrak{S}_e$. 

Similarly to what defined in Subsection \ref{section_trees}, define $\sim_\varepsilon$ to be the transitive closure of $\sim_{d,\varepsilon}$: for any edge $e$ and any $U\in\mathfrak{S}_{e,\varepsilon}$,
we have that $\phi_{e_+}(U)\sim_{d,\varepsilon}\phi_{e_-}(U)$.

Doing so (and not defining equivalence classes with respect to the equivalence class $\sim$ of Subsection \ref{section_trees}) will be crucial to be able to apply Lemma \ref{consequence_winning_theorem} during the proof of Theorem \ref{mainT}. Moreover, this does not affect the hypotheses of the theorem, that continue to be satisfied. Indeed, edge spaces continue to be uniformly hierarchically quasiconvex in vertex spaces, with edge hieromorphisms being full and uniformly coarsely lipschitz. Comparison maps are not affected by this change (but there might be fewer of them, as we are considering possibly smaller edge-space index sets). Finally, the intersection property is preserved by Proposition \ref{concreteness_making}, and clean containers are preserved by Lemma \ref{wedge_container_lemma_blue}.
\end{remark}

In view of Remark \ref{remark_crucial_concrete}, from now on we assume without loss of generality that all edge spaces are $\varepsilon$-concrete for some appropriate $\varepsilon$, that is that the equivalence relations $\sim_\varepsilon$ and $\sim$ are the same.

Let $\widehat{T}$ be the result of coning off the underlying tree associated to the tree of spaces $\mathcal{T}$ with respect to every support tree $T_{[V]}$. We define the index set $\mathfrak{S}$ associated to the tree of hierarchically hyperbolic spaces $\mathcal{T}$ as
\begin{equation}\label{newfactorsystem}
\mathfrak{S}=\mathfrak{S}_1\sqcup\mathfrak{S}_2\sqcup\{\widehat{T}\}.
\end{equation}
The set $\mathfrak{S}_1$ is 
\begin{equation}\label{S_1}
\mathfrak{S}_1\vcentcolon=\Bigl(\bigsqcup_{ v \in V}\mathfrak{S}_{ v}\Bigr)/\sim,
\end{equation} 
as defined in Subsection \ref{section_trees}.

Elements of $\mathfrak{S}_2$ correspond to supports of elements 
in $\mathfrak{S}_1$:
\begin{equation}\label{S_2}
\mathfrak{S}_2\vcentcolon=\lbrace {T}_{[V]}\mid [V]\in\mathfrak{S}_1\rbrace.
\end{equation}
We stress that all these elements are subtrees of  the tree $T$, the tree attached to the tree of hierarchically hyperbolic spaces $\mathcal{T}$.
By the following lemma, the set $\mathfrak{S}_2$ is closed under intersections.

\begin{lemma}\label{lemma_equality_supports}
Suppose that ${T}_{[U]}\cap {T}_{[V]}$ is not empty. Then there exists $[A]\in\mathfrak{S}_1$ for which ${T}_{[A]}={T}_{[U]}\cap {T}_{[V]}$ and $[U], [V]\sqsubseteq [A]$.
\end{lemma}
\begin{proof}
Let $V_v$ and $U_v$ be the representatives of $[V]$ and $[U]$ in the index set $\mathfrak{S}_v$, for all $v\in {T}_{[U]}\cap {T}_{[V]}$.
	
For all $v\in {T}_{[U]}\cap {T}_{[V]}$, consider the set 
\[ \Lambda_v=\{W\in\mathfrak{S}_v\mid V_v,U_v\sqsubseteq W\},\]
which is non-empty since it contains the maximal element of $\mathfrak{S}_v$. 

Since $V_v\vee W_v$ is, by definition, the $\sqsubseteq$-minimal element of $\mathfrak{S}_v$ containing both $V_v$ and $W_v$, it is the unique $\sqsubseteq$-minimal element of $\Lambda_v$, 
which we denote also by $A_v$.
If ${T}_{[U]}\cap {T}_{[V]}$ consists of just one vertex $v$, then $[A]=[V_v\vee U_v]$ is the desired equivalence class: as $[V_v]$ and $[U_v]$ are nested into $[A]$, it follows that ${T}_{[A]}\subseteq {T}_{[V]}\cap T_{[U]}$. 
Therefore ${T}_{[A]}= {T}_{[V]}\cap {T}_{[U]}$. 

If ${T}_{[V]}\cap {T}_{[U]}$ has more than one vertex, analogously to what constructed in the index sets of the vertices, 
there is a unique $\sqsubseteq$-minimal element in the edge-index set $\mathfrak{S}_e$ that we denote by $A_e$, where $e$ is any edge that contains representatives of both $[U]$ and $[V]$.
	 
Assume now that $v,w\in {T}_{[U]}\cap {T}_{[V]}$ and that there is an edge $e$ that connects these two vertices. Then $\phi_v^{\lozenge}(A_e)=A_v$ and $\phi_w^{\lozenge}(A_e)=A_w$. 
Therefore
\[\phi^\lozenge_{v}(A_e)=\phi^\lozenge_{v}(V_e\vee U_e)=\phi^\lozenge_{v}(V_e)\vee \phi^\lozenge_{v}(U_e)=V_v\vee U_v=A_v\] 
by Lemma \ref{lemma:intersection_wedge_commute}.

Thus $A_v\sim A_w$ for all $v,w\in {T}_{[U]}\cap {T}_{[V]}$, and we denote by $[A]$ the equivalence class of (any of the) $[A_v]$. By construction, $[A]$ has a representative where both $[V]$ and $[U]$ have, and hence
${T}_{[U]}\cap {T}_{[V]}\subseteq {T}_{[A]}$.
	 
On the other hand we have that $[V]$ and $[U]$ are nested in $[U_v\vee V_v]=[A]$, and therefore ${T}_{[A]}\subseteq {T}_{[U]}\cap {T}_{[V]}$ by Lemma~\ref{support_inclusion}. Thus, the lemma is proved.
\end{proof}

\begin{corollary}\label{reverse_nesting_blue}
Let $[V],[W]$ be equivalence classes. Then, $[V]\sqsubseteq [W]$ if and only if $T_{[W]}\subseteq T_{[V]}$.
\end{corollary}
\begin{proof}
If $[V]\sqsubseteq [W]$ then $T_{[W]}\subseteq T_{[V]}$, by Lemma \ref{support_inclusion}.
On the other hand, if $T_{[W]}\subseteq T_{[V]}$ we can see that $T_{[W]}=T_{[W]}\cap T_{[V]}$. By Lemma \ref{lemma_equality_supports} there exists $[A]\in\mathfrak{S}_1$ for which $T_{[A]}=T_{[W]}\cap T_{[V]}$ and $[V],[W]\sqsubseteq[A]$. It follows that $T_{[W]}=T_{[A]}$, and therefore that $[W]=[A]$, because we are assuming that the tree $\mathcal{T}$ is decorated (compare Lemma \ref{lemma_decorations}). Thus $[V]\sqsubseteq [W]$.
\end{proof}

\smallskip
To define nesting, orthogonality, and transversality, we proceed as follow.
The element
$\widehat{T}$ is the $\sqsubseteq$-maximal element. 

Relations in $\mathfrak{S}_1$ are as in \cite{BHS2}: two $\sim$-equivalence classes $[V]$ and $[W]$ are nested (respectively orthogonal), 
$[V]\sqsubseteq [W]$ (respectively $[V]\perp [W]$), if there exist a vertex $v\in T$ and representatives $V_v,W_v\in\mathfrak{S}_v$ such that 
$[V]=[V_v]$, $[W]=[W_v]$ and $V_v\sqsubseteq W_v$ (respectively $V_v\perp W_v$) in $\mathfrak{S}_v$. 
If $[V]$ and $[W]$ are not orthogonal and neither is nested into the other, then they are transverse: $[V]\pitchfork [W]$.

Relations in $\mathfrak{S}_2$ are as follows. For two elements $T_{[V]}, T_{[U]}\in\mathfrak{S}_2$, if $T_{[V]}$ is contained as a set in $T_{[U]}$ then $T_{[V]}\sqsubseteq T_{[U]}$, and vice versa. 
Otherwise they are transverse, $T_{[V]}\pitchfork T_{[U]}$.

Relations between an equivalence class $[W]$ and an element $T_{[V]}\in\mathfrak{S}_2$ are as follows:
\begin{itemize}
\item[$(c_1)$] if $[W]\sqsubseteq [V]$ we declare $[W]\perp T_{[V]}$;
\item[$(c_2)$] if $[W]\perp [V]$ we declare $[W]\sqsubseteq T_{[V]}$;
\item[$(c_3)$] otherwise, we declare $[W]\pitchfork T_{[V]}$;
\end{itemize}
Notice that $[W]\perp T_{[V]}$ if and only if $T_{[V]}\subseteq T_{[W]}$, by Corollary \ref{reverse_nesting_blue}.

\subsection{Hyperbolic spaces associated to elements of the index set, and projections onto them}\label{unif_hyp}
Let $\mathcal{C}\hat T=\hat T$, which is produced from the tree $T$ by coning-off each subtree $T_{[W]}\in\mathfrak{S}_2$.

\begin{remark}
As soon as there exists a vertex space $(\mathcal{X}_v,\mathfrak{S}_v)$ and two orthogonal elements $U\perp V$ in $\mathfrak{S}_v$, then the decoration trick of Section \ref{trees_with_decorations} implies that all supports trees $T_{[W]}\in\mathfrak{S}_2$ are properly contained into the tree $T$. Indeed, if $T_{[W]}=T$ for some equivalence
class, it must then be that $T_{[U]}$ and $T_{[V]}$ are properly nested into $T_{[W]}$, and thus $[W]\sqsubseteq [U]$ and $[W]\sqsubseteq [V]$ by 
Lemma \ref{lemma_decorations}. This contradicts the fact that $[U]\perp [V]$, and in particular that there is no equivalence class nested into both.
\end{remark}

To each equivalence class $[V]$ we associate a \emph{favorite vertex} $v\in T_{[V]}$ and the \emph{favorite representative} $V_v\in\mathfrak{S}_v$, so that 
$[V]=[V_v]$. 
Then, define $\mathcal{C}[V]$ to be $\mathcal{C}V_v$. By assumption, there exists a uniform constant $\xi\geqslant 1$ such 
that for all vertices $w$ such that there exists $W\in\mathfrak{S}_w$ with $W\sim V_v$, the comparison map $\mathfrak{c}\colon V_v\to W$ is a 
$(\xi,\xi)$-quasi-isometry.

For $T_{[W]}\in\mathfrak{S}_2$, let $\mathcal{C}T_{[W]}:=\widehat T_{[W]}$ be the hyperbolic space obtained from the tree $T_{[W]}$ by coning-off each subtree 
$T_{[V]}\in\mathfrak{S}_2$ properly contained in $T_{[W]}$, that is $T_{[V]}\subsetneq T_{[W]}$.

\smallskip
Define $\pi_{\widehat T}\colon \mathcal{X}(\mathcal{T})\to \widehat T$ as follows: for $x\in \mathcal{X}_v$, define $\pi_{\widehat T}(x)=v$. 
Notice that $\pi_{\widehat T}$ is the composition of the projection $\mathcal{X}\to T$ of $\mathcal{X}$ on its Bass-Serre tree with the inclusion of the tree $T$ into $\widehat T$.
For all $T_{[W]}\in\mathfrak{S}_2$ the projection $\pi_{T_{[W]}}$ is defined analogously: for $x\in\mathcal{X}_v$, consider the closest-point projection of the vertex $v$ onto the subtree $T_{[W]}$ in the tree $T$. The image of this point under the inclusion map $T\hookrightarrow \widehat T$ is $\pi_{T_{[W]}}(x)\in\mathcal{C}T_{[W]}=\widehat{T}_{[W]}$.
These projection maps $\pi_{T_{[W]}}$ and the projection map $\pi_{\widehat T}$ are uniformly coarsely surjective, being surjective on the set of non-cone points.

Given $[V]\in\mathfrak{S}$ with favorite representative 
$V_{\tilde v}\in\mathfrak{S}_{\tilde v}$, we define $\pi_{[V]}\colon \mathcal{X}\to\mathcal{C}[V]$ as follows. 
If $\pi_{\widehat T}(x)=v$ is a vertex in the support of $[V]$, then there exists a representative $V_{v}\in\mathfrak{S}_{v}$ of the class $[V]$, and 
$\pi_{[V]}(x)$ is defined to be
\begin{equation}\label{proj1}
\pi_{[V]}(x)\vcentcolon=\mathfrak{c}\circ \pi_{V_{v}}(x)\subseteq \mathcal{C}V_{\tilde v}=\mathcal{C}[V],
\end{equation}
where $\mathfrak{c}\colon \mathcal{C}V_v\to\mathcal{C}V_{\tilde v}$ is the comparison map.

If $\pi_{\widehat T}(x)=v$ is not in the support of $[V]$, let $e$ be the last edge in the geodesic connecting $v$ to $T_{[V]}$, so that $e^+\in T_{[V]}$. Define
\begin{equation}\label{proj2}
\pi_{[V]}(x)\vcentcolon=\mathfrak{c}\circ \pi_{V_{e^+}}\bigl(\phi_{e^+}(\mathcal{X}_e)\bigr)\subseteq \mathcal{C}V_{\tilde v}=\mathcal{C}[V],
\end{equation}
where $\mathfrak{c}\colon \mathcal{C}V_{e^+}\to\mathcal{C}V_{\tilde v}$ is the comparison map.

\begin{lemma}\label{proj_ucl}
The projections defined in Equation \eqref{proj1} and Equation \eqref{proj2} are uniform coarsely lipschitz maps. Moreover, they are uniformly coarsely surjective. 
\begin{proof}
In Equation \eqref{proj1} the projections are defined as a composition of a uniform quasi isometry with a uniform coarsely lipschitz map. Therefore, it 
suffices to show that the projections in Equation \eqref{proj2} are uniformly coarsely lipschitz too.

To prove so, notice that the edge $e$ connects the vertex $e^-$, which lies outside of $T_{[V]}$, with the vertex $e^+\in T_{[V]}$, and notice that there exists a representative $V_{e^+}\in\mathfrak{S}_{e^+}$ of $[V]$.
This means that $V_{e^+} \not\sim U$ for any $U\in\mathfrak{S}_{e^-}$, that is $V_{e^+}\in\mathfrak{S}_{e^+}\setminus\phi_{e^+}^\lozenge(\mathcal{X}_e)$.

As all hieromorphisms are full and coarsely lipschitz, invoking Theorem \ref{thmB} we know that the set $\pi_{V_{e^+}}(\phi_{e^+}(\mathcal{X}_e))$ are uniformly 
bounded. Therefore the projections as defined in Equation \eqref{proj2} are uniformly coarsely lipschitz, because the comparison maps $\mathfrak{c}$ are uniform 
quasi-isometries and the sets on which they are applied to are uniformly bounded.

These projections are uniformly coarsely surjective, because the projections of the vertex spaces are, following the assumption of Remark \ref{remark_coarsely_surjective}.
\end{proof}
\end{lemma}
\subsection{Projections between hyperbolic spaces}\label{proj_hyp}
Given an equivalence class $[V]$, define $\rho^{[V]}_{\widehat T}$ to be the support $T_{[V]}$ of the equivalence class $[V]$, which is uniformly bounded in $\widehat T$ because it is coned-off. Define 
$\rho_{[V]}^{\widehat T}\colon \widehat T\to \mathcal{C}[V]$
as follows. 
For $w\in T\setminus T_{[V]}$, consider the geodesic connecting $w$ to $T_{[V]}$, and let $e$ be its last edge, 
so that $e^+\in T_{[V]}$. Define
\begin{equation}\label{proj_maximal}
\rho_{[V]}^{\widehat T}(w)\vcentcolon=\mathfrak{c}\circ \pi_{V_{e^+}}\bigl(\phi_{e^+}(\mathcal{X}_e)\bigr)\subseteq \mathcal{C}V_{\tilde v}=\mathcal{C}[V],
\end{equation}
where $\mathfrak{c}\colon \mathcal{C}V_{e^+}\to\mathcal{C}V_{\tilde v}$ is the comparison map. If $w\in T_{[V]}$, then $\rho_{[V]}^{\widehat T}(w)$ can be 
chosen arbitrarily. 
On the other hand, if $w\in\widehat T\setminus T$, that is $w$ is a cone point, then define $\rho^{\widehat T}_{T_{[V]}}(w)=\rho^{\widehat T}_{T_{[V]}}(w')$, where $w'$ is an arbitrarily chosen vertex in the support tree associated to the cone-point $w$.

For an element $T_{[W]}\in\mathfrak{S}_2$, define $\rho_{\widehat T}^{T_{[W]}}$ to be $ T_{[W]}$, and $\rho^{\widehat T}_{T_{[W]}}\colon \widehat T\to\widehat T_{[W]}$ as follows. For $v\in T$, let $\rho^{\widehat T}_{T_{[W]}}(v)$ be the closest-point projection (in the tree $T$) of $v$ onto $T_{[W]}$. On the other hand, if $v\in\widehat T\setminus T$, that is $v$ is a cone point, then define $\rho^{\widehat T}_{T_{[W]}}(v)=\rho^{\widehat T}_{T_{[W]}}(v')$, where $v'$ is any of the points in the support tree associated to the cone-point $v$.

\smallskip
To define the projections $\rho_{[W]}^{[V]}$ between ($\sim$-classes of) hyperbolic spaces, we proceed as follows. If  $[V]\sqsubseteq [W]$ or 
$[V]\pitchfork[W]$,
then we define the projections as in \cite[Theorem 8.6]{BHS2}. In particular, if $[V]\sqsubseteq [W]$ there exist vertices $v,w, v'$ such that $V_v,W_w$ are the favorite 
representatives of $[V]$ and $[W]$ respectively, $V_{v'}$ and $W_{v'}$ are representatives of $[V]$ and $[W]$ (possibly different from the favorite ones), and 
$V_{v'}\sqsubseteq W_{v'}$. Moreover, let 
$\mathfrak{c}_V\colon \mathcal{C}V_{v'}\to\mathcal{C}V_v$ and $\mathfrak{c}_W\colon \mathcal{C}W_{v'}\to\mathcal{C}W_w$ be comparison maps.
Define
\begin{equation}
\rho_{[W]}^{[V]}=\mathfrak{c}_W\Bigl(\rho_{W_{v'}}^{V_{ v'}}\Bigr)\subseteq \mathcal{C}W_w=\mathcal{C}[W],
\end{equation}
which is a uniformly bounded set in $\mathcal{C}[W]$, and define $\rho_{[V]}^{[W]}\colon \mathcal{C}[W]\to \mathcal{C}[V]$ as
\begin{equation}\label{projrho_eq1}
\rho_{[V]}^{[W]}=\mathfrak{c}_V\circ\rho_{V_{v'}}^{W_{v'}}\circ \bar{\mathfrak{c}}_W,
\end{equation}
where $\bar{\mathfrak{c}}_W$ is a quasi inverse of $\mathfrak{c}_W$ and $\rho_{V_{v'}}^{W_{v'}}\colon \mathcal{C}W_{v'}\to \mathcal{C}V_{v'}$ is the 
projection provided by the hierarchical hyperbolicity of the vertex space $(\mathcal{X}_{v'},\mathfrak{S}_{v'})$.

Analogously, if $[V]\pitchfork[W]$ and there exists a vertex $w'\in T$ such that $\mathfrak{S}_{w'}$ contains representatives $V_{w'}\pitchfork W_{w'}$ of $[V]$ 
and $[W]$, then define
\begin{equation}
\rho_{[W]}^{[V]}=\mathfrak{c}_W\Bigl(\rho_{W_{w'}}^{V_{w'}}\Bigr)\subseteq \mathcal{C}W_w=\mathcal{C}[W]
\end{equation}
and 
\begin{equation}\label{projrho_eq2}
\rho_{[V]}^{[W]}=\mathfrak{c}_V\Bigl(\rho_{V_{w'}}^{W_{w'}}\Bigr).
\end{equation}
If there is no common vertex for the supports of $[V]$ and $[W]$, let $v,w$ be the closest pair of vertices such that $\mathfrak{S}_v,\mathfrak{S}_w$ contain 
representatives $V_v$ of $[V]$ and $W_w$ of $[W]$ respectively, and let $e$ be the last edge of the geodesic starting at $w$ and ending at $v=e^+$.
Define
\begin{equation}\label{new_rhos}
\rho_{[V]}^{[W]}=\mathfrak{c} \circ\pi_{V_{e^+}}( \phi_{e^+}(\mathcal{X}_e)),
\end{equation}
where $\mathfrak{c}\colon \mathcal{C}V_v\to\mathcal{C}V_{\tilde v}$ is the comparison map to the favorite representative.
In a completely symmetrical way we also define $\rho^{[V]}_{[W]}$.

\smallskip
For two elements $T_{[V]}$ and $T_{[V']}$ of $\mathfrak{S}_2$, if $T_{[V]}\sqsubsetneq T_{[V']}$ then define $\rho_{T_{[V']}}^{T_{[V]}}$
to be $\widehat T_{[V]}$, which is uniformly bounded in $\widehat T_{[V']}$ since it is coned-off. Define $\rho_{T_{[V]}}^{T_{[V']}}\colon \widehat T_{[V']} \to\widehat T_{[V]}$ as the 
closest-point projection.

If $T_{[V]}\pitchfork T_{[V']}$, then $\rho_{T_{[V]}}^{T_{[V']}}$ and $\rho_{T_{[V']}}^{T_{[V]}}$ are either the closest-point projections (if $T_{[V]}$ and $T_{[V']}$ do not intersect), or
are defined to be $\widehat T_{[V]}\cap\widehat T_{[V']}$, which by (the proof of) Lemma \ref{lemma_equality_supports} is equal to $\widehat T_{[V_v\vee V'_v]}$, where $V_v$ and $V_v'$ are representatives of $[V]$ and $[V']$ in a vertex $v\in T_{[V]}\cap T_{[V']}$. Notice that if $T_{[V]}\cap T_{[V']}$ is not empty, then it is properly contained in both $T_{[V]}$ and $T_{[V']}$, and therefore will be coned-off in both $\widehat T_{[V]}$ and $\widehat T_{[V']}$. 

\smallskip
Finally, we define projections between an equivalence class $[W]$ and an element $T_{[V]}\in\mathfrak{S}_2$ as follows. The relations between $[W]$ and $T_{[V]}$ were described at the end of Subsection \ref{nes_ort_tra}, as follows:
\begin{itemize}
\item[$(c_1)$] if $[W]\sqsubseteq [V]$ then $[W]\perp T_{[V]}$;
\item[$(c_2)$] if $[W]\perp [V]$ then $[W]\sqsubseteq T_{[V]}$;
\item[$(c_3)$] in any other case, $[W]\pitchfork T_{[V]}$.
\end{itemize}
The projections are defined according to each case:
\begin{itemize}
\item[$(c_1)$] \label{itm:first} in this case $[W]$ and $T_{[V]}$ are orthogonal, and no projection needs to be defined;
\item[$(c_2)$] \label{itm:second} define the set $\rho_{T_{[V]}}^{[W]}$ to be $T_{[V]}\cap T_{[W]}$, which is uniformly bounded in $\widehat T_{[V]}$ because it is coned off, being properly contained in $T_{[V]}$. Define the map $\rho_{[W]}^{T_{[V]}}\colon \widehat T_{[V]}\to 2^{\mathcal{C}[W]}$ as follows. For $x\in\widehat T_{[V]}\setminus\widehat T_{[W]}$, define $\rho_{[W]}^{T_{[V]}}(x)=\mathfrak{c}\circ \pi_{W_{e^+}}\bigl(\phi_{e^+}(\mathcal{X}_e)\bigr)$, 
where the edge $e$ is the last edge on the geodesic connecting $x$ to the support $T_{[W]}$, the vertex $e^+$ is in $T_{[W]}$, the element $W_{e^+}\in\mathfrak{S}_{e^+}$ is the 
representative of $[W]$, and $\mathfrak{c}\colon \mathcal{C}W_{e^+}\to\mathcal{C}W_{v}$ is the comparison map to the favorite representative of $[W]$. For $x\in\widehat T_{[W]}$, define $\rho_{[W]}^{T_{[V]}}(x)$ arbitrarily;
\item[$(c_3)$] \label{itm:third} assume first that $T_{[V]}\cap T_{[W]}\neq\emptyset$. Define $\rho_{T_{[V]}}^{[W]}$ to be $\widehat T_{[V]}\cap\widehat T_{[W]}
$ (the intersection $T_{[V]}\cap T_{[W]}$ must be properly contained in $T_{[V]}$, if not we would fall in case $(c_1)$), and define $\rho_{[W]}^{T_{[V]}}=\rho_{[W]}^{[V\vee W]}$. 

On the other hand, suppose $T_{[V]}\cap T_{[W]}=\emptyset$. Define the set $\rho_{T_{[V]}}^{[W]}$ to be the closest-point projection from $T_{[W]}$ to $T_{[V]}$, and the set $\rho_{[W]}^{T_{[V]}}\subseteq\mathcal{C}[W]$ as follows: let $e$ be the 
last edge on the geodesic (in the tree $T$) connecting $T_{[V]}$ to $T_{[W]}$, and define $\rho_{[W]}^{T_{[V]}}=\mathfrak{c}\circ\pi_{W_{e^+}}\bigl(\phi_{e^+}
(\mathcal{X}_e)\bigr)$, where $\mathfrak{c}\colon \mathcal{C}W_{e^+}\to\mathcal{C}W_{v}$ is the comparison map to the favorite representative of $[W]$.
\end{itemize}

\begin{lemma}\label{rhos_uniformly}
All the maps and sets $\rho_{\bullet}^{\star}$ between hyperbolic spaces defined in this subsection are uniformly bounded sets and well-defined maps, for all
$\bullet,\star \in\mathfrak{S}$.
\begin{proof}
The case when $T_{[V]}\sqsubsetneq T_{[W]}$ is
immediate.

For any equivalence class $[W]$, the set $\rho_{\widehat T}^{[W]}=T_{[W]}$ is uniformly bounded because it is coned off in $\widehat T$, and the map $\rho^{\widehat T}_{[W]}$ is well defined: if $w\in T\setminus T_{[W]}$, then $\rho^{\widehat T}_{[W]}(w)$ is defined in terms of the closest-point projection in the tree $T$ of $w$ onto $T_{[W]}$. Suppose now that $w$ is a cone point of a support which is not $T_{[W]}$, nor contained in $T_{[W]}$. By definition $\rho^{\widehat T}_{[W]}(w)=
\rho^{\widehat T}_{[W]}(w')$, where $w'$ is a chosen vertex in the support whose cone point is $w$. If $w$ is a vertex in $T_{[W]}$, or a cone point of a support contained in $T_{[W]}$, then $\rho^{\widehat T}_{[W]}(w)$ is defined arbitrarily.
Analogously, for a support $T_{[V]}$, the set $\rho_{\widehat T}^{T_{[V]}}$ is uniformly bounded and the map $\rho^{\widehat T}_{T_{[V]}}$ is well defined.

The sets and maps $\rho_{[V]}^{[W]}$ between two equivalence classes are uniformly bounded sets and well-defined maps because comparison maps are quasi isometries, and by Theorem \ref{thmB} (compare also Remark \ref{remark_comparealso}).
For instance, the set $\rho_{[V]}^{[W]}$ of Equation \eqref{new_rhos} is uniformly bounded, because comparison maps are uniform quasi isometries by hypotheses, and because the set $\pi_{V_{e^+}}\bigl(\phi_{e^+}(\mathcal{X}_e)\bigr)$ happearing in the equation is uniformly bounded by Theorem~\ref{thmB}.

The set $\rho_{T_{[W]}}^{[V]}= T_{[V]}\cap T_{[W]}$ defined in item $(c_2)$ is uniformly bounded, because $T_{[V]}\cap  T_{[W]}$ is properly contained in $T_{[W]}$, and therefore it is coned off, and an analogous argument proves that the sets defined in item $(c_3)$ are uniformly bounded. The map $\rho_{[V]}^{T_{[W]}}$ of item $(c_2)$ is also well defined because $T$ is a tree, and therefore for $x\in \widehat T_{[V]}\setminus \widehat T_{[W]}$ the image  $\rho_{[V]}^{T_{[W]}}(x)$ is well-defined.

\end{proof}
\end{lemma}


We are now ready to prove Theorem \ref{mainT}:
\subsection{Proof of Theorem A}\label{proof_mainT}
We verify that the axioms for hierarchically hyperbolic spaces hold for $\bigl(\mathcal{X},\mathfrak{S}\bigr)$.

The set of uniform hyperbolic spaces is described in Subsection \ref{unif_hyp}, along with the projections from $\mathcal{X}$ onto these hyperbolic 
spaces. These are uniformly coarsely lipschitz maps, as proved in Lemma \ref{proj_ucl}.
The projections $\rho_\bullet^\star$ between hyperbolic spaces are uniformly bounded sets, and well defined maps, by Lemma \ref{rhos_uniformly}.

Nesting, orthogonality, and transversality are defined in Subsection \ref{nes_ort_tra}.

\medskip
\noindent

{\bf (Nesting)}
The only non-immediate condition to check is the transitivity of the nesting we defined, and in particular that if $[U]\sqsubseteq [V]$ and $[V]\sqsubseteq T_{[W]}$, then $[U]\sqsubseteq T_{[W]}$.
If $[V]\sqsubseteq T_{[W]}$, by definition $[V]\perp [W]$. Furthermore, since $[U]\sqsubseteq [V]$ then $[W]\perp[U]$, which implies that $[U]\sqsubseteq T_{[W]}$.

Assume now that $[U]\sqsubseteq T_{[V]}$ and $T_{[V]}\sqsubseteq T_{[W]}$. By Corollary \ref{reverse_nesting_blue} it follows that $[W]\sqsubseteq[V]$. By definition we get $[U]\perp [V]$. Therefore $[W]\perp[U]$, which implies that $[U]\sqsubseteq T_{[W]}$.

\medskip
\noindent
{\bf (Intersection property)}
We construct the wedges between elements of $\mathfrak{S}$, for all possible cases.

\smallskip
\noindent
\fbox{$[V]\wedge [W]$}
Let $[V]$ and $[W]$ be two equivalence classes. If $T_{[V]}\cap T_{[W]}$ is non-empty, then there exists a vertex $v$ and representatives 
$V_v$ and $W_v$ of the two classes in $\mathfrak{S}_v$. We have that
\[[V]\wedge [W]= [V_v\wedge W_v],\] 
where we define $[V_v\wedge W_v]=\emptyset$ if $V_v\wedge W_v=\emptyset$. 

If the supports $T_{[V]}$ and $T_{[W]}$ do not intersect, then $[V]$ and $[W]$ are transverse. If  $\mathfrak{S}_{[V]}\cap\mathfrak{S}_{[W]}=\emptyset$ then we define $[V]\wedge [W]=\emptyset$. On the other hand, suppose that $\mathfrak{S}_{[V]}\cap 
\mathfrak{S}_{[W]}$ is non-empty, and suppose that it has more than one $\sqsubseteq$-maximal.
Call these maximals $[U_i]$, for $i\in I$. 
As $[U_i]\sqsubseteq [V]$ and $[U_i]\sqsubseteq [W]$, the supports $T_{[V]}$ and $T_{[W]}$ are both contained into $T_{[U_i]}$, for all $i$. 
As supports are connected, each $T_{[U_i]}$ contains the geodesic $\sigma$ that connects $T_{[V]}$ to $T_{[W]}$. Therefore, each $[U_i]$ has 
representatives
in all edge-spaces in the geodesic $\sigma$, which by abuse of notation we also denote by $U_i$. 

Let $U_{\vee}:=\bigvee_{i\in I}U_i$. Notice that 
$U_\vee$ is nested into each $\sqsubseteq$-maximal element of each edge-space on $\sigma$. Moreover,
$[U_i]\sqsubseteq[U_\vee]$
for all $i\in I$, which leads to a contradiction if $\lvert I\rvert >1$. Therefore, there is only one $\sqsubseteq$-maximal element $[U_1]$
in $\mathfrak{S}_{[V]}\cap \mathfrak{S}_{[U]}$, and $[V]\wedge [W]=[U_1]$.

\smallskip 
\noindent
\fbox{$[V]\wedge T_{[W]}$}
Let $[V]$ be an equivalence class and $T_{[W]}$ be a support. We have that
\begin{equation}\label{delicate}
\begin{aligned}
[V]\wedge T_{[W]}&=\bigvee\bigl\{[U]\mid [U]\sqsubseteq [V]\text{ and }[U]\sqsubseteq T_{[W]}\bigr\}\\
&=[V]\wedge[\cont_{\perp}W_v],
\end{aligned}
\end{equation}
where $v\in T_{[W]}$ is the favorite vertex of $[W]$.

The only non-immediate point of Equation \eqref{delicate} is to check that if two
equivalence classes $[U]$ and $[U']$ are nested into $T_{[W]}$, then so is their join $[U]\vee [U']$. This is indeed the case, by clean containers, as proved in Lemma \ref{clean_containers_use}. 

Therefore, $[V]\wedge T_{[W]}$ is nested into both $[V]$ and $T_{[W]}$, and by construction is the $\sqsubseteq$-maximal of such elements.

\smallskip
\noindent
\fbox{$T_{[V]}\wedge T_{[W]}$}
Let $T_{[V]}$ and $T_{[W]}$ be two distinct supports. If $T_{[V]}\cap T_{[W]}\neq \emptyset$, then the support $T_{[V]}\cap T_{[W]}$ is nested in both $T_{[V]}$ and $T_{[W]}$. We prove that
\begin{equation}\label{intersection_intersection}
T_{[V]}\wedge T_{[W]}=T_{[V]}\cap T_{[W]}.
\end{equation}
To prove that Equation \eqref{intersection_intersection} defines the wedge between $T_{[V]}$ and $T_{[W]}$, it needs to be shown that if $[U]$ is
nested into both $T_{[V]}$ and $T_{[W]}$, then it is also nested into $T_{[V]}\cap T_{[W]}$.

By definition of nesting, we have that $[U]\perp [V]$ 
and $[U]\perp [W]$, and therefore, by Lemma \ref{clean_containers_use}, we have that $[U]\perp \bigl([V]\vee[W]\bigr)=[V\vee W]$, that is $[U]\sqsubseteq T_{[V\vee W]}=T_{[V]}\cap T_{[W]}$.

If $T_{[V]}\cap T_{[W]}=\emptyset$, then there is no element $S\in\mathfrak{S}_2$ (compare Equation \eqref{S_2}) that is nested in both $T_{[V]}$ and $T_{[W]}$. 
The wedge between these two elements of the index set is
\begin{equation}\label{equation_intersection_final}
\begin{aligned}
T_{[V]}\wedge T_{[W]}&=\bigvee \bigl\{[U]\mid [U]\sqsubseteq T_{[V]}\text{ and } [U]\sqsubseteq T_{[W]}\bigr\}\\
&=[\cont_{\perp}{V_v}]\wedge[\cont_{\perp}{W_w}]
\end{aligned}
\end{equation}
Notice that any $[U]$ as in Equation \eqref{equation_intersection_final} will be supported on the geodesic $\sigma$ connecting $T_{[V]}$ to $T_{[W]}$.

\medskip
\noindent
{\bf (Orthogonality)} 
We first prove that if $T_{[V]}\sqsubseteq T_{[W]}$ and $T_{[W]}\perp [U]$, then $ T_{[V]}\perp [U]$. As $[U]\perp T_{[W]}$, we have that $T_{[W]}\subseteq T_{[U]}$. Therefore $T_{[V]}\subseteq T_{[U]}$, that is $[U]\perp T_{[V]}$.
The analogous case of three equivalence classes satisfying the relations $[V]\sqsubseteq [W]$ and $[W]\perp [U]$ is proved in \cite[Lemma 8.9]{BHS2}.

We now construct the (upper) orthogonal containers for elements of $\mathfrak{S}$. 
Consider $T_{[V]}\in\mathfrak{S}_2$. By definition, there is no orthogonality between elements of $\mathfrak{S}_2$. We have that $\cont_\perp T_{[V]}=[V]$. This follows from the definition of orthogonality between equivalence classes and supports.

We claim that $\cont_\perp [V]=T_{[V]}$. To prove this claim, first notice that a support $T_{[W]}$ is orthogonal to $[V]$ if and only if
$T_{[W]}\subseteq T_{[V]}$. Consider now an equivalence class $[W]$ orthogonal to $[V]$. By definition, $[W]\sqsubseteq T_{[V]}$, thus all elements orthogonal to $[V]$ are nested into $T_{[V]}$, proving the claim $\cont_\perp [V]=T_{[V]}$.

To conclude, exploiting the fact that $\mathfrak{S}$ has a wedge operation and just constructed upper orthogonal containers, we notice that the argument of Lemma \ref{wedge_container_lemma_blue} proves that the lower orthogonal containers are $\cont_\perp^UV=U\wedge \cont_\perp V$, for all $U,V\in\mathfrak{S}$.

\medskip
\noindent
{\bf (Consistency)} 
We verify the various cases for this Axiom.

\fbox{$\lbrack W\rbrack \sqsubseteq \hat T$}
Choose a vertex $z\notin T_{[W]}$ and let $x\in\mathcal{X}_z$. 
Let $e$ be the last edge in the geodesic connecting the vertex $z$ to $T_{[W]}$, so that $e^+=w\in T_{[W]}$. 

As $\pi_T(x)=z$, we have that $\rho_{[W]}^{\widehat T}(\pi_T(x))=\mathfrak{c}_W\circ\pi_{W_w}(\phi_w(\mathcal{X}_e))$, where $\mathfrak{c}_W$ is
the comparison map from $\mathcal{C}W_w$ to the favorite representative of $[W]$.
On the other hand, $\pi_{[W]}(x)=\mathfrak{c}_W\circ\pi_{W_w}(\phi_w(\mathcal{X}_e))$.
This means that 
\begin{equation*}
\rho_{[W]}^{\widehat T}(\pi_{T}(x))=\pi_{[W]}(x)=\mathfrak{c}_W\circ\pi_{W_w}(\phi_w(\mathcal{X}_e))
\end{equation*}
is a uniformly bounded set by Theorem \ref{thmB}, and therefore
\begin{equation*}
\diam_{\mathcal{C}[W]}\left(\pi_{[W]}(x)\cup\rho_{[W]}^{\hat T}(\pi_T(x))\right)=\diam\bigl(\mathfrak{c}_W\circ\pi_{W_w}(\phi_w(\mathcal{X}_e))\bigr)
\end{equation*}
is uniformly bounded.

If $z\in T_{[W]}$, then
\begin{equation*}
d_{\hat T}(\pi_{\widehat T}(x),\rho_{\widehat T}^{[W]})=d_{\widehat T}(z,T_{[W]})=0.
\end{equation*}
Therefore, there exists a uniform bound $N$ such that
\begin{equation*}
\min\Bigl\{ d_{\widehat T}\bigl(\pi_{\widehat T}(x),\rho_{\widehat T}^{[W]}\bigr),\diam_{\mathcal{C}[W]}\bigl(\pi_{[W]}(x)\cup\rho_{[W]}^{\widehat T}(\pi_{\widehat T}(x))\bigr)\Bigr\}\leqslant N
\end{equation*}
for all $x\in\mathcal{X}$ and for all $[W]\in\mathfrak{S}$.

\fbox{$T_{[W]}\sqsubseteq \hat T$} Let $T_{[W]}\in\mathfrak{S}_2$ and $x\in\mathcal{X}$. If $x\in\mathcal{X}_z$ for some $z\in T_{[W]}$, then $\pi_{\hat T}(x)\in\rho_{\hat T}^{T_{[W]}}$, and 
therefore $d_{\hat T}(\pi_{\hat T}(x),\rho_{\hat T}^{T_{[W]}})=0$.
On the other hand, if $d_{\hat T}(\pi_{\hat T}(x),\rho_{\hat T}^{T_{[W]}})>1$, and in particular $x\in\mathcal{X}_v$, where $v\notin T_{[W]}$, then $\pi_{T_{[W]}}(x)=\rho^{\hat T}_{T_{[W]}}\bigl(\pi_{\hat T}(x)\bigr)$, and therefore
\[\diam_{T_{[W]}}\bigl(\pi_{T_{[W]}}(x)\cup\rho^{\hat T}_{T_{[W]}}\bigl(\pi_{\hat T}(x)\bigr)\bigr)=
\diam_{T_{[W]}}\bigl(\pi_{T_{[W]}}(x)\bigr)=0.\]
This concludes consistency for this case.

\fbox{$\lbrack U\rbrack \pitchfork\lbrack V\rbrack$} Let $[U],[V]\in\mathfrak{S}$ and assume that $[U]\pitchfork[V]$. We need to prove that there exists some uniform 
constant $\kappa$ such that either 
\begin{equation}\label{eq_consistency_1}
d_{[U]}\bigl(\pi_{[U]}(x),\rho_{[U]}^{[V]}\bigr)\leqslant\kappa\quad\text{ or }\quad d_{[V]}\bigl(\pi_{[V]}(x),\rho_{[V]}^{[U]}\bigr)\leqslant\kappa
\end{equation}
for each $x\in\mathcal{X}$.
We proceed by induction on $d_T(T_{[U]},T_{[V]})$.

If $d_T(T_{[U]},T_{[V]})=0$, then these two finite sets intersect. Therefore,
there exists a vertex $w$ such that $\mathfrak{S}_w$ contains representatives $U_w\pitchfork V_w$ of $[U]$ and $[V]$ respectively. 
Since consistency holds in each hierarchically hyperbolic vertex space, it follows that there exists $\kappa_0$ that satisfies Equation \eqref{eq_consistency_1}. 

Suppose now that $d_T(T_{[U]},T_{[V]})>0$, and consider the geodesic $\gamma$ in $T$ connecting $T_{[U]}$ to $T_{[V]}$, with initial vertex $u$ and final vertex
$v$, so that $u\in T_{[U]}$ and $v\in T_{[V]}$. Let $x\in\mathcal{X}$ be so that 
$x\in\mathcal{X}_z$ for some vertex $z\in T$. There are three possible configurations: either $d_T(u,z)<d_T(v,z)$, or $d_T(u,z)>d_T(v,z)$, or $d_T(u,z)=d_T(v,z)$.

If one of the geodesics in $T$ connecting $z$ either to $T_{[U]}$ or to $T_{[V]}$ has a vertex that lies in $T_{[V]}$ or $T_{[U]}$, then Equation 
\eqref{eq_consistency_1} is trivially satisfied.
Indeed, suppose that the geodesic connecting the vertex $z$ to $T_{[U]}$ passes through $T_{[V]}$. In this case, it follows from the definitions that 
$\pi_{[V]}(x)\in \rho_{[V]}^{[U]}$, and thus $d_{[V]}(\pi_{[V]}(x),\rho_{[V]}^{[U]})=0$.

Therefore, it remains to check the case in which the geodesics $\sigma$ and $\sigma'$ connecting $z$ to $T_{[U]}$ and to $T_{[V]}$ respectively have that $\gamma\cap\sigma\neq\emptyset$ and $\gamma\cap\sigma'\neq\emptyset$, but $\gamma\not\subseteq\sigma$ and $\gamma\not\subseteq\sigma'$.
Let $e$ and $\tilde e$ be the first and the last edges (possibly equal) of $\gamma$, so that $e^{-}=u\in T_{[U]}$ and $\tilde e^{+}=v\in T_{[V]}$.

The first two cases are symmetric, so suppose that $d_T(u,z)<d_T(v,z)$. In particular, $z\not\in T_{[V]}$, for otherwise we would have $d_T(u,z)\geq d_T(v,z)$.
Let $w\in T_{[V]}$ be the favorite vertex of the class $[V]$, and $V_w\in\mathfrak{S}_w$ be its the favorite representative.
By definition
\begin{equation*}
\pi_{[V]}(x)=\mathfrak{c}_V\circ\pi_{V_{v}}\bigl(\phi_{v}(\mathcal{X}_{\tilde e})\bigr),
\end{equation*}
where $\mathfrak{c}_V\colon \mathcal{C}V_v\to \mathcal{C}V_w$ is the comparison map.
We obtain that 
\begin{equation*}
d_{[V]}\bigl(\pi_{[V]}(x),\rho_{[V]}^{[U]}\bigr)=d_{V_w}\bigl(\mathfrak{c}_V\circ\pi_{V_v}(\phi_v(\mathcal{X}_{\tilde e})),
\mathfrak{c}_V\circ\pi_{V_v}(\phi_v(\mathcal{X}_{\tilde e}))\bigr)=0.
\end{equation*}
If $d_T(u,z)>d_T(v,z)$, the same argument shows that
\begin{equation*}
d_{[U]}(\pi_{[U]}(x),\rho_{[U]}^{[V]})=0.
\end{equation*}
We consider now the case $d_T(u,z)=d_T(v,z)$. As $z\not\in T_{[U]}\cup T_{[V]}$, we have that 
\begin{equation*}
\pi_{[V]}(x)=\mathfrak{c}_V\circ\pi_{V_v}(\phi_v(\mathcal{X}_{\tilde e}))\quad\text{ and }
\quad\pi_{[U]}(x)=\mathfrak{c}_U\circ\pi_{U_u}(\phi_u(\mathcal{X}_{e})).
\end{equation*}
It follows that
\begin{equation*}
d_{[V]}\bigl(\rho_{[V]}^{[U]},\pi_{[V]}(x)\bigr)=0\quad\text{ and }\quad d_{[U]}\bigl(\rho_{[U]}^{[V]},\pi_{[U]}(x)\bigr)=0.
\end{equation*}
Therefore, consistency holds for every $[U]\pitchfork[V]\in\mathfrak{S}$.

\fbox{$\lbrack U\rbrack \sqsubseteq\lbrack V\rbrack$} Consistency for the pair $\lbrack U\rbrack \sqsubseteq\lbrack V\rbrack$
is immediate: by definition there exist a vertex $v$ and representatives $U_v\sqsubseteq V_v$ of $[U]$ and $[V]$ respectively. As Consistency 
holds in all vertex spaces, the statement follows.

Suppose now that $[W]$ is such that either 
\begin{enumerate}
\item $[V]\sqsubsetneq [W]$, or  
\item $[V]\pitchfork[W]$ and $[U]\not\perp[W]$.
\end{enumerate}
We claim that $d_{[W]}(\rho_{[W]}^{[V]},\rho_{[W]}^{[U]})$ is uniformly bounded. 

As $[U]\sqsubseteq [V]$, let $U_u,V_u\in\mathfrak{S}_u$ be representatives of $[U]$ and $[V]$ such that $U_u\sqsubseteq V_u$. We now check all the possible cases.

Suppose that $T_{[U]}\cap T_{[W]}\neq\emptyset$ and 
$T_{[V]}\cap T_{[W]}\neq\emptyset$: this can happen either if $[U]\sqsubseteq [W]$ or if $[U]\pitchfork[W]$ and there exist transverse representatives of $[U]$ and $[V]$. 
Let $v,w\in T$ be such that there exist representatives $V_w,W_w\in\mathfrak{S}_w$ satisfying $V_w\sqsubseteq W_w$ (respectively $V_w\pitchfork W_w$), and 
representatives $U_v, W_v\in\mathfrak{S}_v$ such that $U_v\sqsubseteq W_v$ (respectively $U_v\pitchfork W_v$).

Let $m\in T$ be the median of $u,v,w$. As $u,w$ belong to the support of $[U]$ and $[W]$, then so does $m$, since supports are connected trees. 
Likewise, $m$ lies in the support of $[V]$. Let $U_m,V_m$ and $W_m$ be representatives of $[U],[V]$ and $[W]$ in $\mathfrak{S}_m$. Since edge-hieromorphisms are full, we have that 
$U_m\sqsubseteq V_m$, and $U_m\not\perp W_m$, and $V_m\sqsubseteq W_m$ (respectively $V_m\pitchfork W_m$). Because consistency holds in each vertex space, and in particular in 
$(\mathcal{X}_m,\mathfrak{S}_m)$, we conclude that $d_{W_m}(\rho_{W_m}^{U_m},\rho_{W_m}^{V_m})$ is uniformly bounded.
Applying the appropriate comparison maps (that are uniform quasi isometries), it follows that $d_{[W]}(\rho_{[W]}^{[U]},\rho_{[W]}^{[V]})$ is uniformly bounded.

If $T_{[U]}\cap T_{[W]}\neq\emptyset$ and $T_{[V]}\cap T_{[W]}=\emptyset$, let $w$ be a vertex such that there are transverse representatives $U_{w}\pitchfork W_{w}$ of $[U]$ and $[W]$. 
Moreover, let $e$ be the edge separating $T_{[V]}$ from $T_{[W]}$, so that $e^+\in T_{[W]}$.
We have that $\rho_{[W]}^{[V]}=\overline{\mathfrak{c}}_{W}\circ\pi_{W_{e^+}}(\phi_{e^+}(\mathcal{X}_e))$ and $\rho_{[W]}^{[U]}=\mathfrak{c}_{W}\left(\rho_{W_{w}}^{U_{w}}\right)$, where $\mathfrak{c}_{W}\colon \mathcal{C}W_w\to\mathcal{C}[W]$ and $\overline{\mathfrak{c}}_{W}\colon \mathcal{C}W_{e^+}\to\mathcal{C}[W]$ are the comparison maps to the favorite representative of the equivalence class $[W]$.

Let $S_e$ denote the $\sqsubseteq$-maximal element of the index set $\mathfrak{S}_e$ and $S_e'=\phi_{e^+}^{\lozenge}(S_e)$. Recall that the constant $\kappa_0$ denotes the constant coming from the consistency axiom of Definition \ref{HHS_definition} and $\xi$ denotes the constant which uniformly bounds the multiplicative and additive constant of comparison maps (see Definition \ref{comparison_maps} and the second hypothesis of Theorem \ref{mainT}). Therefore
\begin{equation}\label{terrible_equation}
\begin{aligned}
d_{[W]}\left(\rho_{[W]}^{[V]},\rho_{[W]}^{[U]}\right)&=d_{[W]}\left(\overline{\mathfrak{c}}_W(\pi_{W_{e^+}}(\phi_{e^+}(\mathcal{X}_e))),\mathfrak{c}_W\left(\rho_{W_w}^{U_w}\right)\right)\\
&\leq d_{[W]}\left(\overline{\mathfrak{c}}_W(\pi_{W_{e^+}}(\phi_{e^+}(\mathcal{X}_e))), 
\overline{\mathfrak{c}}_{W}\left(\rho_{W_{e^+}}^{S_e'}\right)\right)+d_{[W]}\left(\overline{\mathfrak{c}}_W\left(\rho_{W_{e^+}}^{S_e'}\right),\mathfrak{c}_W\left(\rho_{W_w}^{U_w}\right)
\right)\\
&\leq d_{[W]}\left(\overline{\mathfrak{c}}_W(\pi_{W_{e^+}}(\phi_{e^+}(\mathcal{X}_e))), \overline{\mathfrak{c}}_{W}\left(\rho_{W_{e^+}}^{S_e'}\right)\right)+\\
&\qquad\qquad +d_{[W]}\left(\overline{\mathfrak{c}}_W\left(\rho_{W_{e^+}}^{S_e'}\right),\overline{\mathfrak{c}}_W\left(\rho_{W_{e^+}}^{U_{e^+}}\right)\right)
+d_{[W]}\left(\overline{\mathfrak{c}}_W\left(\rho_{W_{e^+}}^{U_{e^+}}\right),\mathfrak{c}_W\left(\rho_{W_w}^{U_w}\right)\right).
\end{aligned}
\end{equation}
By hypothesis $[U]\sqsubseteq [V]$, so $T_{[V]}\subseteq T_{[U]}$. Moreover, since $T_{[W]}\cap T_{[U]}\neq\emptyset$, $T_{[W]}\cap T_{[V]}=\emptyset$ and $e$ is the last in the geodesic connecting $T_{[V]}$ to $T_{[W]}$, we have that $e^+\in T_{[U]}\cap T_{[W]}$. Therefore, by Lemma \ref{lemma.good.definition} we have that $\overline{\mathfrak{c}}_{W}(\rho^{U_{e^+}}_{W_{e^+}})\asymp\mathfrak{c}_W(\rho^{U_{w}}_{W_{w}})$, and so the last term $d_{[W]}(\overline{\mathfrak{c}}_W(\rho_{W_{e^+}}^{U_{e^+}}),\mathfrak{c}_W(\rho_{W_w}^{U_w}))$ of Equation \eqref{terrible_equation} is uniformly bounded by some $J$. Therefore
\begin{equation}\label{terrible_equation_2}
\begin{aligned}
d_{[W]}\left(\rho_{[W]}^{[V]},\rho_{[W]}^{[U]}\right)&\leq \xi d_{W_{e^+}}\left(\pi_{W_{e^+}}(\phi_{e^+}(\mathcal{X}_e)), \rho_{W_{e^+}}^{S_e'}\right)+\xi+\xi d_{W_{e^+}}\left(\rho_{W_{e^+}}^{S_e'},\rho_{W_{e^+}}^{U_{e^+}}\right)+\xi+J\\
&\leq \xi d_{W_{e^+}}\left(\pi_{W_{e^+}}(\phi_{e^+}(\mathcal{X}_e)), \rho_{W_{e^+}}^{S_e'}\right)+\xi+\xi\kappa_0+\xi+J.
\end{aligned}
\end{equation}
Notice that 
\begin{equation*}
\begin{aligned}
d_{W_{e^+}}\left(\pi_{W_{e^+}}(\phi_{e^+}(\mathcal{X}_e)), \rho_{W_{e^+}}^{S_e'}\right)&\asymp d_{W_{e^+}}\left(\pi_{W_{e^+}}(\phi_{e^+}(\mathcal{X}_e)), \pi_{W_{e^+}}({\bf F}_{S_{e}'})\right)\\
&\leq Kd(\phi_{e^+}(\mathcal{X}_e),{\bf F}_{S_{e}'})+K,
\end{aligned}
\end{equation*}
and so, by Theorem \ref{winning_theorem}, we have that
\begin{equation}\label{terrible2}
d_{W_{e^+}}\left(\pi_{W_{e^+}}(\phi_{e^+}(\mathcal{X}_e)), \rho_{W_{e^+}}^{S_e'}\right)\leq K\eta+K.
\end{equation}
Combining Equation \eqref{terrible_equation_2} and Equation \eqref{terrible2} we obtain that $d_{[W]}\bigl(\rho_{[W]}^{[V]},\rho_{[W]}^{[U]}\bigr)$ is uniformly bounded.

Assume now that $T_{[U]}\cap T_{[W]}=\emptyset$: in particular $[U]\pitchfork[W]$.
By Lemma \ref{support_inclusion} we know that $T_{[V]}\subseteq T_{[U]}$. Therefore, there exists an edge $e$ separating $T_{[V]}$ 
(and $T_{[U]}$) from $T_{[W]}$, so that $e^+\in T_{[W]}$.

As defined in Equation \eqref{new_rhos}, we have that
\begin{equation*}
\rho_{[W]}^{[V]}=\mathfrak{c}_W\circ\pi_{W_{e^+}}( \phi_{e^+}(\mathcal{X}_e))=\rho_{[W]}^{[U]}.
\end{equation*} 
Therefore $\rho_{[W]}^{[V]}= \rho_{[W]}^{[U]}$, and $d_{[W]}(\rho_{[W]}^{[U]},\rho_{[W]}^{[V]})=0$ is uniformly bounded.

\fbox{$T_{[W_1]}\pitchfork T_{[W_2]}$}
Let $T_{[W_1]},T_{[W_2]}\in\mathfrak{S}_2$ satisfying $T_{[W_1]}\pitchfork T_{[W_2]}$, and let $x\in\mathcal{X}$. 
In this case, we always have that
\[\min \bigl\{d_{T_{[W_1]}}(\pi_{T_{[W_1]}}(x),\rho_{T_{[W_1]}}^{T_{[W_2]}}),d_{T_{[W_2]}}(\pi_{T_{[W_2]}}(x),\rho_{T_{[W_2]}}^{T_{[W_1]}})\bigr\}=0,\]
because $\rho_{T_{[W_2]}}^{T_{[W_1]}}$ and $\rho_{T_{[W_1]}}^{T_{[W_2]}}$ are defined as closest-point projections if $T_{[W_1]}\cap T_{[W_2]}=\emptyset$, or as the (coned-off) intersection, if it is non-empty.

\fbox{$T_{[W_1]}\sqsubseteq T_{[W_2]}$}
Let $T_{[W_1]},T_{[W_2]}\in\mathfrak{S}_2$ satisfying $T_{[W_1]}\sqsubseteq T_{[W_2]}$. Consistency follows, because for all $x\in\mathcal{X}$ we have that
\begin{equation*}
\pi_{T_{[W_1]}}(x)=\rho_{T_{[W_1]}}^{T_{[W_2]}}\bigl(\pi_{T_{[W_2]}}(x)\bigr).
\end{equation*}
Therefore $\diam_{\mathcal{C}T_{[W_1]}}\bigl(\pi_{T_{[W_1]}}(x)\cup\rho_{T_{[W_1]}}^{T_{[W_2]}}(\pi_{T_{[W_1]}}(x))\bigr)=0$, where $\mathcal{C}T_{[V]}=\widehat{T}_{[V]}$, and the consistency inequality is satisfied.

Let $T_{[W_3]}\in\mathfrak{S}_2$ be such that either 
\begin{enumerate}
\item $T_{[W_1]}\sqsubseteq T_{[W_2]}\sqsubsetneq T_{[W_3]}$, or 
\item $T_{[W_2]}\pitchfork T_{[W_3]}$.
\end{enumerate}
In either case we have that $\rho_{T_{[W_3]}}^{T_{[W_1]}}\subseteq \rho_{T_{[W_3]}}^{T_{[W_2]}}$, and therefore $d_{T_{[W_3]}}\bigl(\rho_{T_{[W_3]}}^{T_{[W_1]}},\rho_{T_{[W_1]}}^{T_{[W_2]}}\bigr)=0$.

Let now $[V]\in\mathfrak{S}_1$ be such that $[V]\pitchfork T_{[W_2]}$ and $[V]\not\perp T_{[W_1]}$. We want to prove that $d_{[V]}(\rho^{T_{[W_1]}}_{[V]},
\rho^{T_{[W_2]}}_{[V]})$ is uniformly bounded. We now check every possible case. If the support of $[V]$ does not intersect $T_{[W_2]}$ (and therefore, does not
intersect $T_{[W_1]}\subseteq T_{[W_2]}$), then $\rho^{T_{[W_1]}}_{[V]}=\rho^{T_{[W_2]}}_{[V]}$ and the claim is satisfied. If the support $T_{[V]}$ intersects both $T_{[W_1]}$ and $T_{[W_2]}$,
then also in this case we have that $\rho^{T_{[W_1]}}_{[V]}=\rho^{T_{[W_2]}}_{[V]}$.
Finally, if $T_{[V]}$ intersects $T_{[W_2]}$ but not $T_{[W_1]}$, then $\rho^{T_{[W_1]}}_{[V]}=\mathfrak{c}\circ\pi_{V_{e^+}}\bigl(\phi_{e^+}(\mathcal{X}_e)\bigr)$, 
where $e$ is the last edge in the geodesic connecting $T_{[W_1]}$ to $T_{[V]}$, the vertex $e^+$ lies in $T_{[V]}$, and $V_{e^+}$ is the representative 
of $[V]$ in $\mathfrak{S}_{e^+}$. On the other hand, $\rho^{T_{[W_2]}}_{[V]}=\rho^{[W_2]}_{[V]}$, and $[W_2]\pitchfork [V]$.
As both classes $[V]$ and $[W_2]$ are supported on the vertex $e^+$, we have that $\rho^{[W_2]}_{[V]}=\mathfrak{c}\circ\rho^{{W_2}_{e^+}}_{V_{e^+}}$, where 
${W_2}_{e^+}$ is the representative of $[W_2]$ in that vertex.

By Lemma \ref{consequence_winning_theorem} we have that $\pi_{V_{e^+}}\bigl(\phi_{e^+}(\mathcal{X}_e)\bigr)$ is coarsely equal to 
$\rho^{\widetilde S_{e^+}}_{V_{e^+}}$, where $\widetilde S_{e^+}=\phi_{e^+}^\lozenge(S_e)$ and $S_e$ is the $\sqsubseteq$-maximal element of $\mathfrak{S}_{e}$. Therefore $d_{[V]}(\rho_{[V]}^{T_{[W_1]}},\rho_{[V]}^{T_{[W_2]}})$ is uniformly bounded.

\fbox{$\lbrack V\rbrack\pitchfork T_{[W]}$}
Let $T_{[W]}\in\mathfrak{S}_2$. If $T_{[V]}\cap T_{[W]}=\emptyset$, then 
\[\min \bigl\{d_{[V]}(\pi_{[V]}(x),\rho_{[V]}^{T_{[W]}}),d_{T_{[W]}}(\pi_{T_{[W]}}(x),\rho_{T_{[W]}}^{[V]})\bigr\}=0,\qquad \forall\ x\in\mathcal{X}.\]
Thus, suppose that the intersection is non-empty. Since $[V]\pitchfork T_{[W]}$ it follows that $[V]\pitchfork [W]$.
Suppose that $d_{T_{[W]}}\bigl(\pi_{T_{[W]}}(x),\rho_{T_{[W]}}^{[V]}\bigr)$ is big, so that in particular $x\notin T_{[V]}\cap T_{[W]}=\rho_{T_{[W]}}^{[V]}$ and the geodesic 
connecting $x$ to $T_{[V]}$ passes through the set $T_{[W]}$.

By definition, $\pi_{[V]}(x)=\mathfrak{c}\circ\pi_{V_{e^+}}\bigl(\phi_{e^+}(\mathcal{X}_e)\bigr)$, and $\rho_{[V]}^{T_{[W]}}=\rho_{[V]}^{[W]}=\mathfrak{c}(
\rho_{V_{e^+}}^{W_{e^+}})$, where $e^+$ is the vertex of the edge $e$ that belongs to $T_{[V]}\cap T_{[W]}$, while $e^-\in T_{[W]}\setminus T_{[V]}$, and 
$V_{e^+}$ and $W_{e^+}$ are the representatives of $[V]$ and $[W]$ respectively at the vertex $e^+$.

Let $S_e$ be the $\sqsubseteq$-maximal element of $\mathfrak{S}_e$.
As the equivalence class $[V]$ is not supported in the vertex $e^-$, it follows that $V_{e^+}$ is not nested into $\phi_{e^+}^{\lozenge}(S_e)=\widetilde S_{e}$.
On the other hand $W_{e^+}\sqsubseteq \widetilde S_e$.
Therefore, $\rho^{W_{e^+}}_{V_{e^+}}$ and $\rho^{\widetilde S_e }_{V_{e^+}}$ coarsely coincide by Definition 
\ref{HHS_definition}\eqref{HHS_definition_4}, and by Lemma \ref{consequence_winning_theorem} we obtain that
\begin{equation*}\label{equation_consistency_case}
\pi_{V_{e^+}}\bigl(\phi_{e^+}(\mathcal{X}_e)\bigr)\asymp \rho_{V_{e^+}}^{\tilde{S_{e}}}\asymp \rho_{V_{e^+}}^{W_{e^+}},
\end{equation*}
that is, $\pi_{[V]}(x)$ and $\rho_{[V]}^{T_{[W]}}$ coarsely coincide.
Thus, $d_{[V]}\bigl(\pi_{[V]}(x),\rho^{T_{[W]}}_{[V]}\bigr)$ is uniformly bounded.

\fbox{$\lbrack V\rbrack \sqsubseteq T_{[W]}$}
If the distance 
$d_{T_{[W]}}(\pi_{T_{[W]}}(x),\rho_{T_{[W]}}^{[V]})> \kappa_0$, it follows in particular that $\pi_T(x)\notin \rho_{T_{[W]}}^{[V]}= T_{[V]}\cap T_{[W]}$, and that the geodesic in $\widehat T$
connecting $x$ to $\rho_{T_{[W]}}^{[V]}$ passes through the set $T_{[W]}\setminus T_{[V]}$. In this case, we have that $\pi_{[V]}(x)=\pi_{[V]}\bigl(\pi_{T_{[W]}}(x)\bigr)$ is equal
to $\rho_{[V]}^{T_{[W]}}\bigl(\pi_{T_{[W]}}(x)\bigr)$. Therefore the consistency inequality is satisfied also in this case.

\medskip
\noindent
{\bf (Finite complexity)}
It is enough to show finite complexity in $\mathfrak{S}_1$ and $\mathfrak{S}_2$ independently.

Finite complexity in $\mathfrak{S}_1$ follows from \cite[Lemma 8.11]{BHS2}.
For $\mathfrak{S}_2$, notice that any chain of proper
nestings
\begin{equation*}
T_{[U_1]}\sqsupsetneq T_{[U_2]}\sqsupsetneq \dots \sqsupsetneq T_{[U_n]}
\end{equation*}
induces the corresponding chain of proper nestings
$[U_1]\sqsubsetneq [U_2]\sqsubsetneq\ldots\sqsubsetneq[U_n]$ in $\mathfrak{S}_1$, by Corollary \ref{reverse_nesting_blue}.

As only equivalence classes are allowed to be nested into an intersection of supports, and not vice versa, finite complexity is proved.

In particular, it follows that the complexity of $\bigl(\mathcal{X}(\mathcal{T}),\mathfrak{S}\bigr)$ is twice the complexity of $\mathfrak{S}_1$ plus one, and the complexity of $\mathfrak{S}_1$ is $\max_v\chi_v +1$, where $\chi_v$ is the complexity of the vertex space $(\mathcal{X}_v,\mathfrak{S}_v)$.

\medskip
\noindent
{\bf (Large links)}
Let $[W]\in\mathfrak{S}_1$ and $x,x'\in\mathcal{X}$.
Suppose that $x\in\mathcal{X}_v$ and $x'\in\mathcal{X}_{v'}$ for some $v,v'\in T$, and let $w$ be the favorite vertex for $[W]$. Let $E$ denote the maximal of the constants $E_v$ of the Bounded Geodesic Axiom of the hierarchically hyperbolic space $(\mathcal{X}_v,\mathfrak{S}_v)$.

Suppose that, for some $[V]\sqsubseteq[W]$, we have $d_{[V]}(\pi_{[V]}(x),\pi_{[V]}(x'))\geqslant E'$, where $E'$ depends on $E$ and on the quasi-isometry 
constants of the edge hieromorphisms. 
Then $d_{V_w}(\mathfrak{c}\circ\pi_{V_v}(x),\mathfrak{c}\circ\pi_{V_{v'}}(x'))\geqslant E$, for a representative $V_w\in\mathfrak{S}_w$ of $[V]$. 
As the large links axiom holds in $\mathfrak{S}_w$, we have that $V_w\sqsubseteq T_i$, where $\{ T_i\in\mathfrak{S}_w\}_{i=1}^N$ is a set of $N$ elements in $\mathfrak{S}_w$, where $N=\lfloor d_{[W]}(\pi_{[W]}(x),\pi_{[W]}(x'))\rfloor$ and each $T_i$ satisfies $T_i\sqsubsetneq W_w$. 
Moreover, the Large Links Axiom in $\mathfrak{S}_w$ implies that $d_{[W]}(\pi_{[W]}(x),\rho_{[W]}^{[T_i]})=d_{W_w}(
\mathfrak{c}_W\circ\pi_{W_v}(x)),\rho_{W_w}^{{T_i}})\leqslant N$ for all $i=1,\dots,N$. Thus the large links axiom for elements 
$[V]\in\mathfrak{S}_1$ and $[U]\in\mathfrak{S}_{[V]}$ follows.

\smallskip
We now consider the case of $T_{[W]}\in\mathfrak{S}_2$, and $X\in\mathfrak{S}_{T_{[W]}}$. This can happen both when $X$ is an equivalence class, or when $X\in\mathfrak{S}_2$.
We deal with the case $X\in\mathfrak{S}_2$ in the following lemma, whilst the case $X=[V]\in\mathfrak{S}_1$ is considered after the lemma.

\begin{lemma}\label{lemma_finite_set}
Let $x,x'\in\mathcal{X}$ and $S\in\mathfrak{S}_2\cup\{\widehat{T}\}$. 
The set 
\[Y=\lbrace X\in\mathfrak{S}_2\mid X\sqsubsetneq S,\, d_X(\pi_X(x),\pi_X(x'))>4\rbrace\]
is finite. Moreover, the set of $\sqsubseteq$-maximal elements in $Y$ has cardinality bounded linearly in terms of the distance $d_{S}\bigl(\pi_{S}(x),\pi_{S}(x')\bigr)$.
\begin{proof}
Let $\sigma$ be the geodesic in $T$ connecting $v=\pi_T(x)$ to $v'=\pi_T(x')$.
We begin by noticing that, if $X\cap \sigma=\emptyset$, then $d_X(\pi_X(x),\pi_X(x'))=0$ because these two sets coincide, 
and therefore $X\notin Y$. In particular, as nesting between elements of $\{ \widehat T\}\cup\mathfrak{S}_2$ is inclusion, if $\sigma$ does not intersect $S$ then $Y$ will be empty, and the lemma is trivially satisfied.

Suppose now that $\sigma$ intersects $S$, and consider the map $\varphi\colon Y\to \mathcal{P}(\sigma)$ defined as $\varphi(X)=X\cap\sigma$, where $\mathcal{P}(\sigma)$ is the set of subpaths of $\sigma$. We first prove that $\varphi$ is an injective map. 
Let $X,X'\in Y$ be such that $X\neq X'$ and, looking for a contradiction, suppose that $\varphi(X)=\varphi(X')$, so that $X\cap\sigma=X'\cap\sigma$ and therefore
$X\cap\sigma=X\cap X'\cap\sigma$. 

Since $X$ intersects $\sigma$, we have that $\pi_X(x)$ and $\pi_X(x')$ are vertices of $\sigma$. Therefore
$\pi_X(x)$ and $\pi_X(x')$ lie in $X\cap\sigma\subset X\cap X'$. Since $X\cap X'$ is properly contained in both $X$ and $X'$, it will be coned-off 
in both $\mathcal{C}X$ and $\mathcal{C}X'$ by construction. 
Therefore $d_X(\pi_X(x),\pi_X(x'))\leqslant 2$, which contradicts the definition of the set $Y$. Therefore the map $\varphi$ is injective, 
and the set $Y$ is finite.

\smallskip
We now claim that, for elements $X,X'\in Y$, we have that $\varphi(X)\subsetneq \varphi(X')$ if and only if $X\sqsubsetneq X'$. Indeed, if $X\sqsubsetneq X'$, that is $X\subsetneq X'$, then $\varphi(X)\subsetneq \varphi(X')$.
On the other hand, suppose that $\varphi(X)\subsetneq \varphi(X')$, and let $X=T_{[V]}$ and $X'=T_{[V']}$, for some equivalence classes $[V]$ and $[V']$.
Since $\varphi(X)=X\cap\sigma\subsetneq \varphi(X')=X'\cap\sigma$, we have that \begin{equation}\label{eq1_lemma_largelinks}
X\cap \sigma= X\cap X'\cap \sigma.
\end{equation}
Moreover, as $X\cap X'=T_{[V]}\cap T_{[V']}=T_{[V\vee V']}$, from Equation \eqref{eq1_lemma_largelinks} we obtain that
\begin{equation}\label{eq2_lemma_largelinks}
T_{[V]} \cap \sigma=T_{[V\vee V']}\cap \sigma.
\end{equation}
As $[V]\sqsubseteq [V\vee V']$, Lemma \ref{support_inclusion} implies that $T_{[V\vee V']}\sqsubseteq T_{[V]}$. If $T_{[V\vee V']}$ is properly nested into $T_{[V]}$, then $T_{[V\vee V']}$ is coned off in $\mathcal{C}T_{[V]}=\widehat{T}_{[V]}$. Equation \eqref{eq2_lemma_largelinks} implies that $d_{T_{[V]}}(\pi_{T_{[V]}}(x),\pi_{T_{[V]}}(y))=2$, which is a contradiction since $T_{[V]}\in Y$ by hypothesis. Therefore, $T_{[V\vee V']}=T_{[V]}$, which implies that $T_{[V]}\subseteq T_{[V']}$, as desired.

\smallskip
We now show that $Y_{max}=\{X_1,\ldots,X_n\}\subseteq Y$, the set of $\sqsubseteq$-maximal elements in $Y$, has cardinality at most $d_{S}(\pi_{S}(x),\pi_{S}(x'))$. 
Since every element of $Y_{max}$ is properly nested into $S$, it follows that its support is coned off in $\mathcal{C}S=\widehat{S}$. 
We now prove that $X_j\cap \sigma\not\subseteq (X_{k_1}\cup\dots \cup X_{k_r})\cap \sigma$ for any pairwise distinct elements $X_j,X_{k_1},\dots, X_{k_r}$ all belonging to $Y_{max}$.

The claim was just proved for $r=1$. Indeed, if $X_j\cap \sigma\subseteq X_{k_1}\cap\sigma$ then $X_j\sqsubseteq X_{k_1}$, and this contradicts the fact that $X_j$ and $X_{k_1}$ are distinct $\sqsubseteq$-maximal elements of $Y$.
Suppose that $X_j\cap\sigma\subseteq (X_{k_1}\cup X_{k_2})\cap\sigma$, and let $T_{[U_j]},T_{[U_{k_1}]}$ and $T_{[U_{k_2}]}$ denote $X_j,X_{k_1}$ and $X_{k_2}$ respectively. In this case, there exists a path in $\mathcal{C}X_j$ from $\pi_{X_j}(x)$ to $\pi_{X_j}(x')$ that passes through the cone points of $T_{[U_j\vee U_{k_1}]}$ and $T_{[U_j\vee U_{k_2}]}$, which are properly nested into $X_j$.
Then, $d_{X_j}(\pi_{X_j}(x),\pi_{X_j}(x'))\leq 4$, contradicting the assumption that $X_j\in Y_{max}$.

On the other hand,
assume that $X_j\cap\sigma\subseteq (X_{k_1}\cup X_{k_2}\cup \ldots\cup X_{k_r})\cap\sigma$ where $r>2$, $k_i\neq j$ for all $i$, $k_a\neq k_b$ for all $a\neq b$, and there does not exist $k_i\neq k_j$ such that $X_j\cap\sigma\subseteq (X_{k_i}\cup X_{k_j})\cap\sigma$. 
We claim that there exists $s$ such that $X_{k_s}\cap\sigma\subseteq X_j\cap\sigma$. 

Indeed, assume without loss of generality that the endpoints of $X_j\cap\sigma$ are contained in $X_{k_1}\cap\sigma$ and $X_{k_r}\cap\sigma$ respectively. By hypothesis, $ X_j\cap\sigma$ cannot be entirely contained in $(X_{k_1}\cup X_{k_r})\cap\sigma$. Therefore, there exists $v\in X_j\cap\sigma\setminus (X_{k_1}\cup X_{k_r})\cap\sigma$, that is $v\in X_{k_s}\cap\sigma$ for $1< s< r$. Note that $X_{k_s}\cap\sigma$ cannot contain either of the endpoints of $X_j\cap\sigma$, since that would imply that $X_j\cap\sigma$ is contained in either $(X_{k_1}\cup X_{k_s})\cap\sigma$ or $(X_{k_r}\cup X_{k_s})\cap\sigma$. As a consequence we obtain that $X_{k_s}\cap\sigma\subseteq X_j\cap\sigma$, which is a contradiction, since $X_{k_s}$ is maximal with respect to nesting.

From here we can conclude that $\lvert Y_{max}\rvert\leq d_{S}(\pi_{S}(x),\pi_{S}(x'))$. Indeed, given any $\sqsubseteq$-maximal element $X_i\in Y_{max}$ and its cone point $v_i$, the following dichotomy holds: either $v_i$ is a vertex in the geodesic path $\widehat{\sigma}$, or not, where $\widehat{\sigma}$ is a geodesic path in $\mathcal{C}S$ connecting $\pi_{S}(x)$ to $\pi_{S}(x')$. In the latter case, it must be that $\widehat \sigma$ contains either one or two edges of the support $X_i$. Therefore, the bound is proved.
\end{proof}
\end{lemma}

Therefore, if $d_{T_{[U]}}(\pi_{T_{[U]}}(x),\pi_{T_{[U]}}(x'))>4$ for some $T_{[U]}\in\mathfrak{S}_{S}\setminus\{S\}$, that is $T_{[U]}\in Y$, then $T_{[U]}\sqsubseteq X$ for some $\sqsubseteq$-maximal element $X$ of the set $Y$.

We now address the case when $X$ is an equivalence classes $X=[V]\in\mathfrak{S}_{T_{[W]}}$. By definition, $[V]\sqsubseteq T_{[W]}$ if and only if $[V]$ is orthogonal to $[W]$. In particular, it follows that $T_{[V]}\cap T_{[W]}\neq\emptyset$. 

If $T_{[V]}$ does not intersect the geodesic $\sigma$ then the distance $d_{[V]}\bigl(\pi_{[V]}(x),\pi_{[V]}(x') \bigr)$ is equal to zero by Equation~\eqref{proj2}, because the edge $e$ appearing in the cited equation will be the same for both $x$ and $x'$.

Now assume that $T_{[V]}\cap\sigma\neq\emptyset$.
As a fist sub-case, suppose that $\sigma\cap T_{[W]}$ is empty, let
\begin{equation}\label{largelinks_won_eq1}
\mathcal{I}\vcentcolon=\bigl\{[V]\sqsubseteq T_{[W]}\mid T_{[V]}\cap \sigma\neq\emptyset\bigr\},
\end{equation}
and notice that $\mathcal{I}$ could be infinite. Consider the geodesic $\alpha$ connecting $T_{[W]}$ to $\sigma$ in the tree $T$, and notice that $\alpha$ has at least one edge, being $T_{[W]}$ and $\sigma$ disjoint. For $[V]\in\mathcal{I}$, we have that $T_{[V]}$ intersects both $T_{[W]}$ and $\sigma$, and therefore $\alpha$ is contained in $T_{[V]}$, being $T$ is a tree. Thus the set
$T_{[W]}\cap\bigcap_{[V]\in\mathcal{I}}T_{[V]}$ is not empty, because (at least) the initial vertex on the geodesic $\alpha$ belongs to this intersection.

Let the set $I$ index $\mathcal I$, that is 
$\mathcal{I}=\{[V_i]\}_{i\in I}$. Without loss of generality, we can suppose that each $V_{i}$ is the 
representative of $[V_i]$ in the vertex space $(\mathcal{X}_v,\mathfrak{S}_v)$.
Let $S_v\in\mathfrak{S}_v$ be the $\sqsubseteq$-maximal element, and notice that $[V_i]\sqsubseteq [S_v]$ for all $i\in I$. Furthermore, note that $[V]\sqsubseteq [\bigvee_{i\in I}V_{i}]$ for all $[V]\in\mathcal{I}$ and let $[V_{\vee}]$ denote $[\bigvee_{i\in I}V_i]$.  Therefore, in this first sub-case, Large Links is satisfied by the family $Y\cup \{[V_\vee]\}$ for the elements $T_{[W]}\in\mathfrak{S}$ and $x,x'\in\mathcal{X}$.

For the second sub-case, suppose that $\sigma\cap T_{[W]}$ is not empty, and let $\{v_1,\dots,v_n\}$ be the finitely many vertices of $\sigma\cap T_{[W]}$ (there can be only finitely many such 
vertices because $\sigma$ is a geodesic).
Analogously to Equation~\eqref{largelinks_won_eq1}, for all $v_i\in \sigma\cap T_{[W]}$ define
\begin{equation*}
\mathcal{I}_{v_i}=\bigl\{[V]\sqsubseteq T_{[W]}\mid v_i\in T_{[V]}\cap \sigma\bigr\},
\end{equation*}
and notice that $\mathcal{I}=\bigcup \mathcal{I}_{v_i}$.
As in the previous case, for each $\mathcal{I}_{v_i}$ consider $[S_{v_i}]$, and notice that $[V]\sqsubseteq [S_{v_i}]$ for all $[V]\in
\mathcal{I}_{v_i}$, for all $i=1,\dots,n$.
Therefore, Large 
Links for an element $T_{[W]}\in \mathfrak{S}_2$ is satisfied considering the set $Y\cup\{[V_\vee^{v_1}],\dots,[V_\vee^{v_n}]\}$. 

Notice that, in both sub-cases, we bounded the cardinality of the sets $Y\cup \{[V_\vee]\}$ and $Y\cup\{[V_\vee^{v_1}],\dots,[V_\vee^{v_n}]\}$ in terms of $\sigma$, that is in terms of $d_T(x,x')$. As $d_{T_{[W]}}(\pi_{T_{[W]}}(x),\pi_{T_{[W]}}(x'))$ is bounded from above by $d_T(x,x')$, we obtained the 
desired bound on the cardinality of these sets.

Combining these bounds with Lemma \ref{lemma_finite_set}, we conclude the proof of Large Links for the case $X\sqsubsetneq T_{[W]}$.

\smallskip
Finally, we prove Large Links for the $\sqsubseteq$-maximal element $\widehat T$. From Lemma \ref{lemma_finite_set} applied with $S=\widehat T$, there 
are only finitely many (and the number depends only on the distance in $\widehat T$ from $x$ to $x'$) 
elements $X\in\mathfrak{S}_2$ such that $d_X(\pi_X(x),\pi_X(x'))$ is big. 
On the other hand, for an equivalence class $[V]\sqsubseteq \widehat T$, the distance $d_{[V]}(\pi_{[V]}(x),\pi_{[V]}(x'))$ can be big
only if the support $T_{[V]}$ intersects the geodesic $\sigma$ connecting $v$ to $v'$ (otherwise, it would be zero). 
Let $S_1,\dots, S_n$ be the $\sqsubseteq$-maximal
elements of all the finitely many edges in $\sigma\cap T_{[V]}$. We have that $[V]\sqsubseteq [S_i]$ for all $i=1,\dots,n$.
Therefore, the set $Y\cup\{S_1,\dots,S_n\}$ is the set of significant elements for the Axiom.

Let $E'$ be the constant that satisfies the Large Links Axiom of the (uniformly) hierarchically hyperbolic vertex spaces (see Definition 
\ref{HHS_definition}), and let $E>\max\{2, E'\}$. Then Large Links is satisfied with this constant $E$. 

\medskip
\noindent
{\bf (Bounded geodesic image)}
Consider $[W]\sqsubsetneq \widehat{T}$, and let $\gamma$ be a geodesic in $\hat T$. 
If $\gamma\cap T_{[V]}=\emptyset$, let $e$ be the last edge in the geodesic connecting $\gamma$ to $T_{[V]}$, and suppose $e^+\in T_{[V]}$. 
Then $\rho^{\hat T}_{[V]}(\gamma)=\mathfrak{c}_W\circ\pi_{V_{e^+}}(\phi_{e^+}(\mathcal{X}_e))$ is a uniformly bounded set. If not, then $\gamma$ 
intersects $\rho^{[V]}_{\widehat T}$.
The cases $[V]\sqsubseteq T_{[W_1]}$, $T_{[W_1]}\sqsubseteq T_{[W_2]}$, and $T_{[W_1]}\sqsubseteq \hat T$, where $T_{[W_1]},T_{[W_2]}\in\mathfrak{S}_2$, are analogous.

\smallskip
Let $[W] \in\mathfrak{S}$, let $[V]\sqsubseteq[W]$, and let $\gamma$ be a geodesic in $\mathcal{C}[W]=\mathcal{C}W_w$ (where $w$ is the favorite 
vertex of $[W]$ and $W_w\in\mathfrak{S}_w$ is the favorite representative). 
Let $V_w$ be the representative of $[V]$ supported in the vertex $w$, so that $\rho_{[W]}^{[V]}=\rho_{W_w}^{V_w}$.
The Bounded
Geodesic Image Axiom in this case follows because it holds in the vertex space $(\mathcal{X}_w,\mathfrak{S}_w)$ (notice that the constant $E$ 
changes according to the quasi-isometry constant of the comparison maps).

\medskip
\noindent
{\bf (Partial realization)}
Notice that two elements $T_{[W_1]}$ and $T_{[W_2]}$ of $\mathfrak{S}_2$ are never orthogonal. 
Consider $k+1$ pairwise orthogonal elements $[V_1],\ldots,[V_k],T_{[W]}\in\mathfrak{S}$,
and let $p_i\in\pi_{[V_i]}(\mathcal{X})\subseteq \mathcal{C}[V_i]$, for $i=1,\dots,k$, and $v_S\in \widehat T_{[W]}$.

By definition of orthogonality, $ T_{[V_i]}\cap T_{[V_j]}\neq\emptyset$ for all $i\neq j$, $T_{[W]}\subseteq T_{[V_i]}$ for all $i=1,\dots, n$, and 
in particular $T_{[W]}\subseteq \bigcap_{i=1}^{k} T_{[V_i]}$. Consider a vertex $v\in T_{[W]}$ that is not a cone point and has distance at most one from 
$v_S$, that is $v\in T\cap T_{[W]}$ and $d_{T_{[W]}}(v,v_S)\leqslant 1$.
As $v\in T_{[V_i]}$ for all $i=1,\dots,k$, without loss of generality we can suppose that $V_i$ is an element of $\mathfrak{S}_v$, by choosing representatives. We have that $V_i\perp V_j$ for all $i\neq j$.
Comparison maps are uniform quasi-isometries, and $p_i\in \pi_{[V_i]}(\mathcal{X})$, therefore the element $\mathfrak{c}_i(p_i)$ is uniformly close to the set $\pi_{V_i}(\mathcal{X})$ for all $i=1,\dots,k$, where $\mathfrak{c}_i\colon \mathcal{C}[V_i]\to\mathcal{C}V_i$ is the comparison map. For $i=1,\dots,k$, let $p_i^v\in \pi_{V_i}(\mathcal{X})$ be a point such that $d_{V_i}\bigl(p_i^v,\mathfrak{c}_i(p_i)\bigr)$ is uniformly bounded.

By Partial realization in the vertex space $(\mathcal{X}_v,\mathfrak{S}_v)$, there 
exists $x\in\mathcal{X}_v$ such that
$d_{V_i}(\pi_{V_i}(x),p_i^v)$ is uniformly bounded for all $i$. As comparison maps are uniform quasi-isometries, we obtain that 
$d_{[V_i]}(\pi_{[V_i]}(x),p_i)$ is uniformly bounded for all $i$. Moreover, $d_{T_{[W]}}(\pi_{T_{[W]}}(x),v_S)=d_{T_{[W]}}(v,v_S)\leqslant 1$.

If $[V_i]\sqsubseteq [U]$, then $[U]$ has a representative $U_v\in\mathfrak{S}_v$ such that $V_i\sqsubseteq U_v$. Therefore 
$d_{[U]}(\pi_{[U]}(x),\rho^{[V_i]}_{[U]})$ is uniformly bounded, because $x$ is a realization point for $\{V_i\}_{i=1}^k$, and comparison maps 
are uniform quasi isometries.

If $[V_i]\sqsubseteq T_{[U]}$, then $\rho_{T_{[U]}}^{[V_i]}=T_{[V_i]}\cap T_{[U]}$ and $\pi_{T_{[U]}}(x)\in
\rho_{T_{[U]}}^{[V_i]}$.  Therefore, $d_{T_{[U]}}\bigl(\pi_{T_{[U]}}(x),\rho_{T_{[U]}}^{[V_i]}\bigr)=0$.
Analogously, for $T_{[W]} \sqsubseteq T_{[U]}$ we have that  $d_{T_{[U]}}\bigl(\pi_{T_{[U]}}(x),\rho_{T_{[U]}}^{T_{[W]}}\bigr)=~0$.
This argument also applies when considering the $\sqsubseteq$-maximal element, therefore proving that $d_{\widehat T}\bigl(\pi_{\widehat T}(x),\rho_{\widehat T}^{T_{[W]}}\bigr)=0$ and $d_{\widehat T}\bigl(\pi_{\widehat T}(x),\rho_{\widehat T}^{[V_i]}\bigr)=0$.

\smallskip
Let now $[V_i]\pitchfork [U]$. Either $T_{[U]}\cap T_{[V_i]}=\emptyset$, in which case the distance $d_{[U]}(\pi_{[U]}(x),\rho^{[V_i]}_{[U]})$ is 
uniformly bounded, or $T_{[U]}\cap T_{[V_i]}\neq\emptyset$, in which case $[U]$ has a representative $U_v\in\mathfrak{S}_v$ that is transverse to $V_i$. 
Therefore, in the latter case the distance $d_{[U]}(\pi_{[U]}(x),\rho^{[V_i]}_{[U]})$ is again uniformly bounded, because it is in the vertex space 
$\mathcal{X}_v$, and comparison maps are uniform quasi-isometries.

If $[V_i]\pitchfork T_{[U]}$ then $\pi_{T_{[U]}}(x)\in\rho_{T_{[U]}}^{[V_i]}$, and therefore $d_{T_{[U]}}\bigl(\pi_{T_{[U]}}(x),\rho_{T_{[U]}}^{[V_i]}\bigr)=~0$.
For the last case, suppose that $T_{[W]}\pitchfork [U]$ for some $[U]\in\mathfrak{S}_1$. If the support of $[U]$ does not intersect $T_{[W]}$, then 
$\pi_{[U]}(x) \in\rho_{[U]}^{T_{[W]}}$. So, suppose that $T_{[W]}$ intersects $T_{[U]}$. 
Again using Lemma \ref{consequence_winning_theorem}, we can conclude.

If $T_{[W]}\pitchfork T_{[U]}$ and $T_{[W]}\cap T_{[U]}\neq\emptyset$, then the subtree $T_{[W]}\cap T_{[U]}=T_{[W\vee U]}$ is strictly contained in $T_{[U]}$. Therefore, $T_{[W]}\cap T_{[U]}$ is coned-off in $\mathcal{C}T_{[U]}=\widehat{T}_{[U]}$. Since $\pi_{T_{[U]}}(x)\in T_{[W]}\cap T_{[U]}$, we obtain that $d_{T_{[U]}}(\pi_{T_{[U]}}(x),\rho^{T_{[W]}}_{T_{[U]}})\leq 2$. On the other hand, if $T_{[W]}\cap T_{[U]}=\emptyset$ then $\pi_{T_{[U]}}(x)=\rho^{T_{[W]}}_{T_{[U]}}=e^+\in T_{[U]}$, where $e$ is the last edge in the geodesic separating $T_{[W]}$ from $T_{[U]}$, and therefore $d_{T_{[U]}}(\pi_{T_{[U]}}(x),\rho^{T_{[W]}}_{T_{[U]}})=0$.
By definition, no element of $\mathfrak{S}_1$ can be nested into an element of $\mathfrak{S}_2$. Therefore, all the relevant cases have been considered.

\medskip
\noindent
{\bf (Uniqueness)}
Suppose $x,y\in\mathcal{X}$ are such that $d_{R}(\pi_{R}(x),\pi_{R}(y))\leqslant K$, for all $R\in\mathfrak{S}$. In particular, we have that
$d_{\widehat T}\bigl(\pi_{\widehat T}(x),\pi_{\widehat T}(y) \bigr)\leqslant K$, that $d_{S}\bigl(\pi_{S}(x),\pi_{S}(y) \bigr)\leqslant K$ for all 
$S\in \mathfrak{S}_2$, and that $d_{[V]}\bigl(\pi_{[V]}(x),\pi_{[V]}(y) \bigr)\leqslant K$ for all $[V]\in\mathfrak{S}_1$.

Suppose that the distance in $\widehat T$ from $\pi_{\widehat T}(x)$ to $\pi_{\widehat T}(y)$ is realized by a path only consisting of vertices of 
$T\subseteq \widehat T$, and let 
\begin{equation*}
v_0=\pi_T(x),v_1,\dots,v_{k-1},\pi_T(y)=v_k,
\end{equation*}
be these vertices, where $k\leqslant K$. In particular, no four consecutive vertices can belong to the same support tree, because this would produced a shorter path in $\widehat T$ joining $x$ to $y$.

We have that $d_{\mathcal{X}}(x,y)\leqslant\sum_{i=0}^kd_{\mathcal{X}_{v_i}}\bigl(\mathfrak{g}_{v_i}(x),\mathfrak{g}_{v_i}(y)\bigr)+k$. Moreover, for all $i=0,\dots, k$ we have that the distance $d_{\mathcal{X}_{v_i}}\bigl(\mathfrak{g}_{v_i}(x),\mathfrak{g}_{v_i}(y)\bigr)$ is uniformly bounded. Indeed, if this is not the case, by Uniqueness in the hierarchically hyperbolic space $(\mathcal{X}_{v_i},\mathfrak{S}_{v_i})$, there exists $V\in\mathfrak{S}_{v_i}$ such that $d_V\bigl(\pi_V(\mathfrak{g}_{v_i}(x)),\pi_V(\mathfrak{g}_{v_i}(y))\bigr)$ is not bounded. By \cite[Lemma 8.18]{BHS2} and Theorem \ref{thmB}, we have that $d_V\bigl(\pi_V(\mathfrak{g}_{v_i}(x)),\pi_V(\mathfrak{g}_{v_i}(y))\bigr)$ and $d_{[V]}\bigl(\pi_{[V]}(x),\pi_{[V]}(y)\bigr)$ coarsely coincide, and therefore the latter is not bounded. This contradicts the fact that 
$d_{[V]}\bigl(\pi_{[V]}(x),\pi_{[V]}(y)\bigr)\leqslant K$, and thus $d_{\mathcal{X}_{v_i}}\bigl(\mathfrak{g}_{v_i}(x),\mathfrak{g}_{v_i}(y)\bigr)\leqslant\zeta=\zeta(K)$ is uniformly bounded, as claimed.
Therefore, $d_{\mathcal{X}}(x,y)\leqslant \zeta'(K)$, for some uniform bound $\zeta'(K)$.

\smallskip
Suppose now that in the geodesic $\sigma$ in $\hat T$ connecting $\pi_{\hat T}(x)$ to $\pi_{\hat T}(y)$ there is a cone point. Therefore, 
there exists an element 
$T_{[W_1]}\in\mathfrak{S}_2$ containing two points $x_1$ and $y_1$ in this geodesic (that, therefore, have distance two in $\hat T$ since $T_{[W]}$ is 
coned-off in $\widehat T$). 
As $T_{[W_1]}\in\mathfrak{S}_2$, we have that $d_{T_{[W_1]}}\bigl(\pi_{T_{[W_1]}}(x_1),\pi_{T_{[W_1]}}(y_1)\bigr)=d_{T_{[W_1]}}(x_1,y_1)\leqslant~K$. Either the geodesic 
$\sigma_1$ in $\mathcal{C}T_{[W_1]}=\widehat{T}_{[W_1]}$ connecting these 
two points only consists of vertices of $T$, or there are cone points, and therefore an element $T_{[W_2]}\in\mathfrak{S}_2$ containing two elements
$x_2,y_2$ of the geodesic $\sigma_1$.

As complexity in $\mathfrak{S}_2$ is finite and nesting coincides with inclusion, this process must end after a finite number of steps (that 
depends only on $K$). Therefore, there
exists a geodesic in $T$ connecting $\pi_{\hat T}(x)$ to $\pi_{\hat T}(y)$, whose length is bounded from above by a function in $K$.
Repeating the argument given before, we conclude that $d_{\mathcal{X}}(x,y)$ is uniformly bounded.

\smallskip
This concludes the proof of hierarchical hyperbolicity of the space $\bigl(\mathcal{X}(\mathcal{T}),\mathfrak{S}\bigr)$.

\section{Applications of Theorem \ref{mainT}}\label{last_section}

In this concluding section, we collect two applications of Theorem \ref{mainT}, that is Corollary \ref{mainC}, and Theorem \ref{mainGraphProducts}.

\subsection{Finite graphs of hierarchically hyperbolic groups}\label{subsection_graph_hhgs}
In this subsection we apply Theorem \ref{mainT} to prove Corollary \ref{mainC}:
\begin{cor52}
Let $\mathcal{G}=\bigl(\Gamma,\{G_v\}_{v\in V},\{G_e\}_{e\in E},\{\phi_{e^\pm}\colon G_e\to G_{e^\pm}\}_{ e\in E}\bigr)$ be a finite graph of hierarchically
hyperbolic groups. Suppose that:
\begin{enumerate}
\item each edge-hieromorphism is hierarchically quasiconvex, uniformly coarsely lipschitz and full;
\item comparison maps are isometries;
\item the hierarchically hyperbolic spaces of $\mathcal{G}$ have the intersection property and clean containers.
\end{enumerate}
Then the group associated to $\mathcal{G}$ is itself a hierarchically hyperbolic group.
\end{cor52}

We begin with the following lemma, in which we use the notation of Section \ref{trees_with_decorations}. 

\begin{lemma}\label{useful_lemma_1}
Let $\mathcal{T}$ be a tree of hierarchically hyperbolic spaces and $\widetilde{\mathcal{T}}$ be the corresponding decorated tree. Then
\begin{enumerate}
\item $\pi_{\widetilde{T}_{[V]_\star}}(\mathcal{X}({\mathcal{T}}))$ is isometric to $\mathcal{C} T_{[V]}$, and quasi-isometric to $\mathcal{C} \widetilde{T}_{[V]_\star}$, for all support trees $T_{[V]}\in\mathfrak{S}_2$; 
\item $\pi_{[V]_{\star}}(\mathcal{X}(\mathcal{T}))$ is isometric to $\pi_{[V]}(\mathcal{X}(\mathcal{T}))$, and quasi-isometric to $\pi_{[V]_{\star}}(\mathcal{X}(\widetilde{\mathcal{T}}))$, for all equivalence classes $[V]\in \mathfrak{S}_1$;
\item $\mathcal{X}(\mathcal{T})$ is hierarchically quasiconvex in $\mathcal{X}(\widetilde{\mathcal{T}})$.
\end{enumerate}
\end{lemma}
\begin{proof}
\begin{enumerate}
\item The first assertion of this item follows from the fact that the projections to hyperbolic spaces for elements in $\mathcal{X}(\mathcal{T})$ are not modified by decorating the tree $\mathcal{T}$.
Furthermore, by the construction of Section \ref{trees_with_decorations}, there exists a constant $C>0$ such that  $\mathcal{C}\widetilde{T}_{[V]_\star}=\mathcal{N}_C
\bigl(\pi_{\widetilde{T}_{[V]_\star}}
(\mathcal{X}(\mathcal{T}))\bigr)$,
and therefore $\pi_{\widetilde{T}_{[V]_\star}}(\mathcal{X}(\mathcal{T}))$ is quasi-isometric to $\mathcal{C}\widetilde{T}_{[V]_\star}$.

\item As the favorite representative of the equivalence class $[V]_\star$ is the same as of the class $[V]$, it follows that $\pi_{[V]_{\star}}(\mathcal{X}(\mathcal{T}))$ is isometric to $\pi_{[V]}(\mathcal{X}(\mathcal{T}))$. The second assertion of this item follows from the equality $\mathcal{X}(\widetilde{\mathcal{T}})=\mathcal{N}_C
\bigl(\mathcal{X}(\mathcal{T})\bigr)$.

\item 
By what was just proved in the previous points, $\pi_U(\mathcal{X}(\mathcal{T}))$ is $k(0)$-quasiconvex in $\pi_U(\mathcal{X}(\widetilde{\mathcal{T}}))$, for all $U\in \mathfrak{S}$, for some fixed number $k(0)$.

Moreover, let $\vec{b}$ be a $\kappa$-consistent tuple such that $b_X\in\pi_X(\mathcal{X}(\mathcal{T}))$ for every $X\in\mathfrak{S}$ and let $x\in\mathcal{X}(\widetilde{\mathcal{T}})$ be a realization point of $\vec{b}$. Since $\mathcal{X}(\widetilde{\mathcal{T}})=\mathcal{N}_C(\mathcal{X}(\mathcal{T}))$ there exists $x'\in\mathcal{X}(\mathcal{T})$ such that $d_{\mathcal{X}(\widetilde{\mathcal{T}})}(x,x')\leq C$, and therefore the proof is complete.
\end{enumerate}
\end{proof}

As already mentioned in Section \ref{trees_with_decorations}, to construct the hierarchically hyperbolic structure of the graph of hierarchically hyperbolic groups $\mathcal{G}$ of Corollary \ref{mainC}, we do not consider directly a decorated tree, because there might not be a non-trivial action of the fundamental group of $\mathcal{G}$ on that hierarchically hyperbolic space.
Instead, we proceed as follows.
Let 
\begin{equation}\label{tree_of}
\mathcal{T}=\Bigl(T, \{H_{ w}\}_{ w\in V},\{ H_{ f}\}_{ f\in E},\{\phi_{ f^\pm}\}\Bigr)
\end{equation}
be the tree of hierarchically hyperbolic groups associated to $\mathcal{G}$, as described in \cite[Section 8.2]{BHS2}. In particular, $T=(V,E)$ is the Bass-Serre tree associated to the finite graph $\Gamma$, each $H_{ w}$ is conjugated in the total group $G$ to
$G_v$, where $w$ maps to $v$ via the quotient map $T\to \Gamma$, analogously $H_{f}$ is conjugated to $G_e$, and the
edge maps $\phi_{ f^\pm}$ agree with these conjugations of edge and vertex groups to give the embeddings in the tree of hierarchically hyperbolic groups. Let $\mathcal{X}(\mathcal{T})$ be the associated metric space, and let $\mathfrak{S}$ denote the index set associated to $\mathcal{X}(\mathcal{T})$, as described in Section \ref{Section4}.

Associated to this, we consider the decorated tree $\widetilde{ \mathcal{T}}$ of hierarchically hyperbolic groups, as described in Section \ref{trees_with_decorations}. By Theorem \ref{mainT}, the metric space $\mathcal{X}(\widetilde{\mathcal{T}})$ admits a hierarchically hyperbolic space structure, that we denote by $\widetilde{\mathfrak{S}}$. By Lemma \ref{useful_lemma}, the metric space $\mathcal{X}(\mathcal{T})$ is hierarchically quasiconvex in $\mathcal{X}(\widetilde{\mathcal{T}})$, and therefore $\bigl(\mathcal{X}(\mathcal{T}), \widetilde{\mathfrak{S}}\bigr)$ is a hierarchically hyperbolic space by \cite[Proposition 5.5]{BHS2}, where the hyperbolic spaces associated to an element $U\in\widetilde{\mathfrak{S}}$ is defined as $\pi_U\bigl(\mathcal{X}(\mathcal{T})\bigr)\subseteq \mathcal{C}U$. From Remark \ref{remark_coarsely_surjective}, we are assuming that every $\pi_U$ is uniformly coarsely surjective, so in fact there is no harm in considering $\mathcal{C}U$ instead of $\pi_U\bigl(\mathcal{X}(\mathcal{T})\bigr)$. As $\widetilde{\mathfrak{S}}$ and $\mathfrak{S}$ coincide as sets of indices (what changes are the hyperbolic spaces associated to each index, as detailed in Section~\ref{trees_with_decorations}), the above substitution is equivalent to equipping the metric space $\mathcal{X}(\mathcal{T})$ with the hierarchically hyperbolic structure given by $\mathfrak{S}$.
That is to say, $\bigl(\mathcal{X}(\mathcal{T}),\mathfrak{S}\bigr)$ is a hierarchically hyperbolic space.

\smallskip
As shown in~\cite[Section 8.1]{BHS2}, the hierarchically hyperbolic structure associated to $\mathfrak{S}$ can be made equivariant, if the
starting hierarchically hyperbolic spaces are hierarchically hyperbolic \emph{groups}. 
We rewiev the construction here, extend it to cover the bigger index set we are using, and use it to prove Corollary~\ref{mainC}.


We recall here the notion of $\mathcal{T}$-\emph{coherent} bijections, where $\mathcal{T}$ is the tree of hierarchically hyperbolic spaces.
A bijection of the index set $\mathfrak{S}$ given in Equation \eqref{newfactorsystem} is said to be $\mathcal{T}$-coherent if:
\begin{itemize}
\item it induces bijections on the sets $\mathfrak{S}_1$ and $\mathfrak{S}_2$;
\item it preserves the relation $\sim$ on $\mathfrak{S}_1$;
\item it induces a bijection $b$ of the underlying tree $T$ that commutes with $f\colon \bigsqcup_{v\in V} \mathfrak{S}_v\to T$, 
where $f$ sends each $V\in \mathfrak{S}_v$ to the vertex $v$. That is,  $fb=bf$.
\end{itemize}
Notice that the composition of $\mathcal{T}$-coherent bijections is $\mathcal{T}$-coherent. Therefore, let $\mathcal{P}_{\mathcal{T}}\leqslant \Aut(\mathfrak{S})$ be the group of $\mathcal{T}$-coherent bijections.

To produce the index set $\mathfrak{S}$ in a $\mathcal{P}_{\mathcal{T}}$-equivariant manner, we proceed as follows. Notice that $G$ acts on $\bigsqcup_{v\in V}\mathfrak{S}_v$, so that for any $V_v\in \mathfrak{S}_v$ we have that $g.V_v\in \mathfrak{S}_{g.v}$. This extends to an action of $\mathfrak{S}_1$ defining $g.[V]=[g.V]$
For any $[W]\in\mathfrak{S}_1$, choose a left transversal $\mathcal{S}_{[W]}$ of the subgroup 
\[\Stab_G([W])=\bigl\{g\in G\mid g[W]=[W]\bigr\},\]
and impose that $e_G\in \mathcal{S}_{[W]}$.
For each $\mathcal{P}_{\mathcal{T}}$-orbit in $\mathfrak{S}_1$ choose a representative $[V]$ of the orbit, a favorite vertex $v$ for $[V]$, and a favorite
representative $V_v\in\mathfrak{S}_v$ for $[V]$. 
For any element $g\in G$, there is a unique element $l\in \mathcal{S}_{[V]}$ such that $g\in l\cdot\Stab_G([V])$. We declare $lv$ to be the favorite vertex of $g[V]$, and $gV_v\in\mathfrak{S}_{l.v}$ to be the favorite representative of the equivalence class $g.[V]$.

This definition is consistent, that is that if $g,\tilde g\in G$, then the favorite vertex of $(g\tilde g).[V]$ coincides with the favorite vertex of $g.\bigl(\tilde g.[V]\bigr)$. Indeed, suppose that $\tilde g\in \tilde l\cdot \Stab_G([V])$, that $g\tilde g\in p\cdot \Stab_G([V])$, and that $g\in l_\star \cdot\Stab_G(\tilde g[V])$, for unique elements $\tilde l,p\in\mathcal{S}_{[V]}$ and $l_\star \in\mathcal{S}_{\tilde g[V]}$. Thus, the favorite vertex of $g\tilde g[V]$ is $p.v$, and its representative is $V_{p.v}\in\mathfrak{S}_{p.v}$. On the other hand, the favorite vertex of $\tilde g[V]$ is $\tilde l. v$, with favorite representative $\tilde l[V]$, and consequently the favorite vertex of $g\bigl(\tilde g[V]\bigr)$ is $(l_\star \tilde l).v$, with favorite representative $V_{(l_\star \tilde l).v}$. It remains to be shown that $p.v=(l_\star\tilde l).v$, that is that $p^{-1}l_\star\tilde l\in \Stab_G(v)$. As $g\in l_\star \cdot\Stab_G(\tilde g[V])$ and $\Stab_G(\tilde g[V])=\tilde g\Stab_G([V])\tilde g^{-1}$, we have that $g\tilde g\in (l_\star \tilde g)\cdot\Stab_G([V])=(l_\star \tilde l)\cdot\Stab_G([V])$.
Therefore, as $g\tilde g$ belongs to a unique coset of $\Stab_G([V])$, we have that $p\cdot \Stab_G([V])=(l_\star \tilde l)\cdot \Stab_G([V])$, that is $p^{-1}l_\star\tilde l\in \Stab_G(v)$.

\smallskip
For any vertex $v\in T$ and any $V_v\in \mathfrak{S}_v$, an element $g\in G$ induces an isometry $g_{V_v}\colon \mathcal{C}V_v\to\mathcal{C}gV_v$. We use these maps to produce isometries $g_{[V]}\colon \mathcal{C}[V]\to\mathcal{C}g[V]$ in the following way. Denote $[U]:=g[V]$ and let $v$ be the favorite vertex of the equivalence class $[V]$, so that $\mathcal{C}[V]=\mathcal{C}V_v$. Let $U_{g.v}:=gV_v\in \mathfrak{S}_{g.v}$. Moreover, by how we defined the favorite representative for $[U]$, we have that $\mathcal{C}[U]=\mathcal{C}U_{l.v}$, where $g\in l\cdot\Stab_G([V])$ for a unique $l\in\mathcal{S}_{[V]}$. We define $g_{[V]}:=\mathfrak{c}\circ g_{V_v}$, where $\mathfrak{c}\colon \mathcal{C}U_{g.v}\to \mathcal{C}U_{l.v}$ is the comparison map between the two representatives $U_{g.v}\in\mathfrak{S}_{g.v}$ and $U_{l.v}\in\mathfrak{S}_{l.v}$. By hypotheses of Corollary \ref{mainC} the comparison map $\mathfrak{c}$ is an isometry. Therefore $g_{[V]}$ is an isometry, and the diagrams of Equation \eqref{coarsely.commuting.diagrams} uniformly coarsely commute.

\smallskip
From the definition of the action of $\mathcal{P}_{\mathcal{T}}$ on $\mathfrak{S}_2$, it follows that $\mathcal{C}g.T_{[U]}=\mathcal{C}T_{g.[U]}$. We are now ready to prove Corollary \ref{mainC}.

\begin{proof42}\upshape
\upshape
Let $\mathcal{T}$ be the tree of hierarchically hyperbolic spaces constructed from the finite graph of hierarchically hyperbolic groups, as done in Equation \eqref{tree_of}.
By Theorem \ref{mainT}, the metric space $\mathcal{X}(\mathcal{T})$ associated to $\mathcal{T}$ admits a hierarchical hyperbolic structure $\mathfrak{S}$. We
choose $\mathfrak{S}$ following the constraints of Subsection \ref{subsection_graph_hhgs}. 

The group $G$ acts on $\mathcal{X}(\mathcal{T})$ in the following way. At the level of the metric space
$g.x=gx\in\mathcal{X}(\mathcal{T})$ for all $x\in \mathcal{X}(\mathcal{T})$. The action at the level of the index set $\mathfrak{S}$ is defined by
$g.[V]=[gV]\in\mathfrak{S}_1$ for all $[V]\in\mathfrak{S}_1$, and $g.T_{[V]}=T_{g.[V]}\in\mathfrak{S}_2$ for all $T_{[V]}\in\mathfrak{S}_2$. At the level of
hyperbolic spaces, the action of $G$ is completely determined by the actions of the hierarchically hyperbolic groups $G_v$ on the hyperbolic spaces associated 
to elements of $\mathfrak{S}_1$, and by the action of $G$ on the Bass-Serre tree for hyperbolic spaces associated to elements of~$\mathfrak{S}_2$.

Therefore $(G,\mathfrak{S})$ is a hierarchically hyperbolic space. Moreover, $G\leqslant \mathcal{P}_{\mathcal{G}}$, because the action is given by 
$\mathcal{T}$-coherent automorphisms. As in \cite[Corollary 8.22]{BHS2}, this action is cocompact and proper. The action of $G$ on $\mathcal{G}$ is cofinite
if and only if the induced actions on $\mathfrak{S}_1$ and $\mathfrak{S}_2$ are cofinite, and this is indeed the case. The action on $\mathfrak{S}_1$
coincides with the action considered in \cite[Corollary 8.22]{BHS2} and therefore is cofinite, and the action on 
\[\mathfrak{S}_2=\{T_{[V]}\mid [V]\in\mathfrak{S}_1\}\]
is cofinite because the action on $\mathfrak{S}_1$ is.

This proves that $G$ is a hierarchically hyperbolic group. It has the intersection property and clean containers because $\bigl(\mathcal{X}(\mathcal{T}),
\mathfrak{S}\bigr)$ has these properties.

\qed
\end{proof42}

\begin{remark}\label{remark_qi}
If supports of equivalence classes are bounded (as for instance in the setting of \cite{BHS2}) then for any equivalence class $[V]$ there will exist a vertex $v\in T_{[V]}$ that is fixed by the subgroup $\Stab_G([V])$. In this case, the vertex $v$ can be chosen to be the favorite vertex of the equivalence class $[V]$. This choice is consistent with the action of the group $G$: if $v$ is the vertex fixed by all elements of $\Stab_G([V])$, then $g.v\in T_{g.[V]}$ will be the vertex fixed by all elements of $\Stab_G(g.[V])=g\Stab_G([V])g^{-1}$. 
Therefore, for each equivalence class $[V]$ the favorite vertex can be chosen so that it is fixed by $\Stab_G([V])$. With this choices, the vertices $g.v$ and $l.v$ involved in the definition of the isometries $g_{[V]}\colon \mathcal{C}[V]\to\mathcal{C}g[V]$, where $g\in l\cdot\Stab_G([V])$, will coincide. In this case, therefore, there is no need to require comparison maps to be isometries and, to be able to apply Theorem~\ref{mainT}, one just needs comparison maps to be uniformly quasi isometries.

On the other hand, if supports of equivalence classes are unbounded, then there might exists elements $g\in \Stab_G([V])$ that do not fix any vertex in $T_{[V]}$. This happens whenever $g$ acts as a translation on the tree $T_{[V]}$, globally (but not point-wise) fixing it.
\end{remark}

Therefore, we obtain
\begin{corollary}\label{corollary_qi}
Let $\mathcal{G}=\bigl(\Gamma,\{G_v\}_{v\in V},\{G_e\}_{e\in E},\{\phi_{e^\pm}\colon G_e\to G_{e^\pm}\}_{ e\in E}\bigr)$ be a finite graph of hierarchically
hyperbolic groups. Suppose that:
\begin{enumerate}
\item each edge-hieromorphism is hierarchically quasiconvex, uniformly coarsely lipschitz and full;
\item comparison maps are uniform quasi-isometries;
\item the hierarchically hyperbolic spaces of $\mathcal{G}$ have the intersection property and clean containers;
\item for all $[V]\in \mathfrak{S}_1$ there exists $v\in T_{[V]}$ such that $g.v=v$ for all elements $g\in \Stab_G([V])$.
\end{enumerate}
Then the group associated to $\mathcal{G}$ is itself a hierarchically hyperbolic group.
\end{corollary}

\subsection{Finite graph of hyperbolic groups}\label{BFsubsection}
Let us briefly comment on some Bestvina-Feighn flavored applications of Corollary \ref{mainC} concerning graphs of hyperbolic groups.

First, let us stress that with Corollary \ref{mainC} we cannot hope to recover the full combination theorem of Bestvina-Feighn. Indeed, consider the graph of groups
associated to the HNN extension where vertex and edge groups are the same free group $F$, one embedding is the identity map $\id_F$, and the other is a hyperbolic
automorphism~$\phi$
\[F\ast_\phi=\langle F,t\mid tft^{-1}=\phi(f)\quad \forall f\in F\rangle.\]
This group is hyperbolic, by means of \cite{BestvinaFeighn}.

As the vertex and the edge groups are hyperbolic, they admit the hierarchically hyperbolic structure $(F,\{F\})$ with intersection property and 
clean containers, and the embeddings $\id_F$ and $\phi$ extend to full, hierarchically quasiconves coarsely lipschitz hieromorphisms. The only equivalence class in the index
set is $[F]$, and its support tree $T_{[F]}$ is equal to the whole Bass-Serre tree associated to the HNN extension.

It can be seen that the comparison maps $\mathfrak{c}\colon\mathcal{C}F_v\to\mathcal{C}F_u$ are not uniform quasi isometries, where $v,u$ are vertices 
in the support $T_{[F]}$ and $F_v$, $F_u$ are representatives of $[F]$, because the automorphism $\phi$ is hyperbolic. 
Therefore, the hypotheses of Corollary \ref{mainC} are not met, and in particular the projection 
$\pi_{[V]}$ (as defined
in Equation \eqref{proj1}) would not be a $(K,K)$-coarsely lipschitz map for any $K$, as required in the definition of hierarchically hyperbolic space.

\smallskip
On the other hand, if also the automorphism $\phi$ is the identity of $F$, that is the HNN extension is the direct product $F\times \Z$, then all the hypotheses of Corollary \ref{mainC}
are met, and the index set produced for the group by Corollary \ref{mainC} is $\{[F],T_{[F]},\widehat T\}$, where $\widehat T$ is the $\sqsubseteq$-maximal element and has a bounded  
associated hyperbolic space, $[F]\perp T_{[F]}$, $\mathcal{C}[F]$ is the free group $F$, and $\mathcal{C}T_{[F]}=\widehat{T}_{[F]}$ is the Bass-Serre tree of the HNN extension, which is isometric to a line. That is, in this case we recover the usual index set for the direct product of the two hyperbolic groups $F$ and $\Z$.

\medskip
Let us now suppose that the groups appearing in Corollary \ref{mainC} are hyperbolic, and that edge groups are (hierarchically) quasiconvex in vertex groups. 
For the sake of simplicity, let us also suppose that the finite graph
of hyperbolic groups $\mathcal{G}$ has two vertices and an edge, that is, we are considering an amalgamated free product. To construct the hierarchically hyperbolic
structures for the vertex groups $G_v$ and $G_w$, we proceed as follows. Let $G_e$ be the edge group, and let $\phi_v(G_e)$ and $\phi_w(G_e)$ be its (hierarchically
quasiconvex) images into the vertex groups.
By \cite[Theorem 1]{Sp1}, the subgroup $\phi_v(G_e)$ induces a hierarchically hyperbolic structure $\mathfrak{S}^0$ on $G_v$, given by cosets of certain 
quasiconvex subgroups (up to finite Hausdorff distance). 
To obtain a full hieromorphisms $\phi_v$, we are forced to induce
on the edge group $G_e$ the hierarchical structure $\mathfrak{S}^0_{\phi_v(G_e)}$. On the other hand, $\mathfrak{S}^0_{\phi_v(G_e)}$ induces on the other
vertex group $G_w$ a new hierarchical structure $\mathfrak{S}^1$, and to make the hieromorphism $\phi_w$ full, we need to enrich the structure of the edge group $G_e$
with all the (possibly new) cosets that appear in $\mathfrak{S}^1_{\phi_w(G_e)}$, and so on. 

If this process stabilizes after a finite number of times, then the groups can be
given hierarchically hyperbolic structures that induce a full, coarsely lipschitz hieromorphism with hierarchically quasiconvex images.
This construction always produces structures with the intersection property and clean containers, 
but it is
unclear whether there is a simpler way to articulate the necessary
hypothesis in this case, than just requiring the comparison maps to be uniformly quasi isometries.


\subsection{Graph products of hierarchically hyperbolic groups}
In this subsection, we prove Theorem \ref{mainGraphProducts} of the Introduction:
\begin{thmGP}
Let $\Gamma$ be a finite simplicial graph, $\mathcal{G}=\{G_v\}_{v\in V}$ be a family of hierarchically hyperbolic groups with the intersection property and clean containers. Then 
the graph product $G=\Gamma\mathcal{G}$ is a hierarchically hyperbolic group with the intersection property and clean containers.
\end{thmGP}
\begin{proof}
Throughout the proof, if $G$ denotes the graph product $\Gamma\mathcal{G}$ and $\Delta$ is a subgraph of $\Gamma$, we denote with  $G_\Delta$ the subgroup of $G$ generated by the family of subgroups $\{G_v\mid v\in \Delta\}$. This is canonically isomorphic to the graph product $\Delta\mathcal{G}_\Delta$, where $\mathcal{G}_\Delta$ is the subfamily of $\mathcal{G}$ indexed by elements in $\Delta$.
Given vertex groups $\{G_v\}_{v\in V}$, we fix once and for all word metrics on them, and we always consider the graph product metric on $\Gamma\mathcal{G}$, so that the (infinite) generating set of the graph product $\Gamma\mathcal{G}$ consists of all vertex-groups elements. In particular, for a full subgroup $H$ of the graph product $G$, that is a subgroup conjugated to a $G_\Delta$ as above, the inclusion map $H\to G$ is an isometric embedding.

\smallskip
We show by induction on the number of vertices that every graph product $G$ of hierarchically hyperbolic groups with the intersection property and clean containers is again a hierarchically hyperbolic group with the intersection property and clean containers, and that for any full subgroup $H$ of $G$, hierarchically hyperbolic group structures (with intersection property and clean containers) can be given to $H$ and $G$ so that the canonical inclusion $H\hookrightarrow G$ is a full, hierarchically quasiconvex hieromorphism, inducing isometries at the level of hyperbolic spaces.

The case $n=1$ is trivial, so let us suppose that $V=\{v,w\}$. 
If the vertices are connected by an edge, then the graph product is the direct product of the two vertex groups, its hierarchically hyperbolic structure is described in Example \ref{example_directproduct}, and it satisfies the inductive statement we want to prove.

On the other hand, if the two vertices are not connected by an edge, then the graph product is the free product of the two vertex groups, and also in this case the inductive statement is satisfied.

Let us suppose that the graph $\Gamma$ has $n$ vertices, that is $\lvert V\rvert=n$, and that the lemma is satisfied by graph products on at most $n-1$ vertices.
If the graph product splits non-trivially as a direct or free product, then either $G= G_{\Delta}\times G_{\Theta}$ or $G=G_{\Delta}\ast G_{\Theta}$, where $\Delta$ and $\Theta$ are proper non-trivial subgraphs of $\Gamma$. In both cases the inductive statement is satisfied, by induction and by
either invoking Example \ref{example_directproduct} or the free product case (as done for graph products on two vertices).
Therefore, suppose that $G$ does not split non-trivially as a direct nor as a free product. Consider any (non-central and non-isolated) vertex $v\in V$ and the splitting 
\begin{equation}\label{splitting}
G\cong G_{\Gamma\setminus\{v\}}\ast_{G_{\link(v)}}(G_{\link(v)}\times G_v).
\end{equation}
We now check that all the hypotheses of Corollary \ref{mainC} are satisfied.

By the inductive hypotheses the groups $G_{\Gamma\setminus\{v\}}$ and $G_{\link(v)}$ admit a hierarchically hyperbolic group structures with the intersection property and clean containers, and we call $\mathfrak{S}_{\Gamma\setminus\{v\}}$ and $\mathfrak{S}_{\link(v)}$ their index sets, respectively. By Lemma~\ref{directproduct_intersection} the direct product $G_{\link(v)}\times G_v$ is a hierarchically hyperbolic group with the intersection property, and it also satisfies clean containers by \cite[Lemma~3.6]{ABD}. Moreover, also by inductive hypotheses, the inclusions $\iota_1 \colon G_{\link(v)}\hookrightarrow G_{\Gamma\setminus\{v\}}$ and $\iota_2\colon G_{\link(v)}\hookrightarrow G_{\link(v)}\times~G_v$ are full, hierarchically quasiconvex hieromorphisms, and $\iota_{i,U}^*$ are isometries for $i=1,2$ and for all $U\in\mathfrak{S}_{\link(v)}$. 

Moreover, $\iota_1$ and $\iota_2$ are isometric embeddings. By choosing inverse isometries for the maps $\iota_{i,U}^*$ for $i=1,2$ and all $U\in\mathfrak{S}_{\link(v)}$, we conclude that the comparison maps, as defined in Definition \ref{comparison_maps}, are again isometries.
Therefore, all of the hypotheses of Corollary \ref{mainC} are satisfied, and we apply it to the graph of groups appearing in Equation \eqref{splitting}. Thus, the group $G$ admits a hierarchically hyperbolic group structure with the intersection property and clean containers. To conclude the proof, it is enough to prove that the embedding $G_{\Delta}\hookrightarrow G$ is a full, hierarchically quasiconvex hieromorphism, and that induces isometries at the level of hyperbolic spaces, where $\Delta$ is any proper subgraph of $\Gamma$. 

\smallskip
Let us first consider the case $\Delta=\Gamma\setminus\{v\}$, and let us show that $G_{\Gamma\setminus\{v\}}$ is hierarchically quasiconvex in $G$.
Recall that the index set $\mathfrak{S}$ constructed in Corollary~\ref{mainC} for $G_{\Gamma}$ is $\mathfrak{S}_1\cup\mathfrak{S}_2\cup \{\widehat{T}\}$, as fully described in Equation~\eqref{S_1} and Equation~\eqref{S_2}.

Any element of $\mathfrak{S}_1$ is an equivalence class $[V]$, equipped with a favourite representative $V_w$ in the Bass-Serre tree $T$ for which $\mathcal{C}[V]=\mathcal{C}V_w$.
On the other, any element of $\mathfrak{S}_2$ is a support tree $T_{[V]}$, and the metric space $\mathcal{C}T_{[V]}$ is the tree $T_{[V]}$ in which all properly contained support trees $T_{[W]}$ are coned-off.

For each $[V]\in\mathfrak{S}_1$, the projection $\pi_{[V]}$, as defined in Equation \eqref{proj1} and Equation \eqref{proj2}, is
\begin{equation*}
\pi_{[V]}(x)=\begin{cases}
\mathfrak{c}_w\circ\pi_{V_w}(x),&\quad \forall x\in \mathcal{X}_v,\ v\in T_{[V]};\\
\mathfrak{c}_{e^+}\circ\pi_{V_{e^+}}(\phi_{e^+}(\mathcal{X}_e)),&\quad 
\forall x\in \mathcal{X}_v,\ v\notin T_{[V]},
\end{cases}
\end{equation*}
where $e=e(v)$ is the last edge in the geodesic connecting $v$ to $T_{[V]}$ such that $e^+\in T_{[V]}$, and the maps $\mathfrak{c}_w$ and $\mathfrak{c}_{e^+}$ denote the appropriate comparison maps to the favorite representative of $[V]$.

Let $x\in\mathcal{X}_v\subseteq\mathcal{X}$ and let $T_{[V]}\in\mathfrak{S}_2$. Then, $\pi_{T_{[V]}}(x)$ is defined as the composition of the closest point projection of $v$ to $T_{[V]}$ in the Bass-Serre tree $T$, with the inclusion of $T_{[V]}$ into the coned-off~$\mathcal{C}T_{[V]}=\widehat{T}_{[V]}$.

To prove that $G_{\Gamma\setminus\{v\}}$ is hierarchically quasiconvex in $G_\Gamma$, we need to check the two conditions of Definition~ \ref{hierarchical_quasiconvexity}.
For each element $T_{[V]}\in\mathfrak{S}_2$ we have that $\pi_{T_{[V]}}(G_{\Gamma\setminus\{v\}})$ is a point in $\mathcal{C}T_{[V]}=\widehat{T}_{[V]}$ and, therefore, it is quasiconvex in~$\mathcal{C}T_{[V]}$.

Suppose that $[V]\in\mathfrak{S}_1$,
and assume that $[V]$ has a representative in $g.\mathfrak{S}_{v}$, where $\mathfrak{S}_{v}$ is the index set associated to the vertex group $G_v$. In particular $[V]=\{V\}$,
and $\pi_{[V]}(G_{\Gamma\setminus\{v\}})\subseteq \pi_{V}(g.G_{\link(v)})$. Since $V\not\in g.\mathfrak{S}_{\link(v)}$, the set $\pi_{V}(g.G_{\link(v)})$ is uniformly bounded, and therefore $\pi_{[V]}(G_{\Gamma\setminus\{v\}})$ is quasiconvex in $\mathcal{C}[V]$. 

On the other hand, assume that the group orbit $G.[V]$ intersects $\mathfrak{S}_{\Gamma\setminus\{v\}}$. Without loss of generality, as the group acts isometrically on the hyperbolic spaces, we can assume that $[V]$ has a representative $\tilde V\in \mathfrak{S}_{\Gamma\setminus\{v\}}$. By definition $\pi_{[V]}(G_{\Gamma\setminus\{v\}})=\mathfrak{c}\circ \pi_{\tilde V}(G_{\Gamma\setminus\{v\}})$, where $\mathfrak{c}$ is the comparison map from $\tilde V$ to the favourite representative of $[V]$. 
By Axiom \eqref{axiom1} of Definition \ref{HHS_definition}, the set $\pi_{\tilde V}(G_{\Gamma\setminus\{v\}})$ is quasiconvex in $\mathcal{C}\tilde V$, and therefore $\pi_{[V]}(G_{\Gamma\setminus\{v\}})$ is quasiconvex in $\mathcal{C}[V]$, being $\mathfrak{c}$ an
isometry.
It follows that for every element $[V]\in\mathfrak{S}_1$, the set $\pi_{[V]}(G_{\Gamma\setminus\{v\}})$ is quasiconvex in $\mathcal{C}[V]$.

To conclude the proof of hierarchical quasiconvexity, consider a consistent tuple $\vec{b}$ in $(G,\mathfrak{S})$ such that $b_{[V]}\in\pi_{[V]}(G_{\Gamma\setminus\{v\}})$ and $b_{T_{[V]}}\in\pi_{T_{[V]}}(G_{\Gamma\setminus\{v\}})$ for every $[V]\in\mathfrak{S}_1$. 
The sets $\pi_{T_{[V]}}(G_{\Gamma\setminus\{v\}})$ are uniformly bounded, being points, for all $T_{[V]}\in\mathfrak{S}_2$. Moreover, $\pi_{[V]}(G_{\Gamma\setminus\{v\}})$ are uniformly bounded for every equivalence class $[V]\in\mathfrak{S}_1$ which has a representative in $g.\mathfrak{S}_{v}$. 

Let $\alpha$ denote the vertex of the Bass-Serre tree in which the subgroup $G_{\Gamma\setminus\{v\}}$ is supported.
Let $i:G_{\Gamma\setminus\{v\}}\to G_\Gamma$ be the hieromorphism defined as follows. At the metric-space level define it to be the natural inclusion. At the level of index sets $i^{\lozenge}(U)=[U]$ and, at the level of hyperbolic spaces, $i^*_U\colon \mathcal{C}U\to\mathcal{C}[U]$ is the comparison map $\mathfrak{c}\colon\mathcal{C}U_{\alpha}\to\mathcal{C}[U]$, which is an isometry. 

For each $[V]\in\mathfrak{S}_1$, we have that 
\begin{equation*}
\pi_{[V]}(G_{\Gamma\setminus\{v\}})=\begin{cases}
\mathfrak{c}_\alpha\circ\pi_{V_{\alpha}}(G_{\Gamma\setminus\{v\}}),&\quad\text{if }\alpha\in T_{[V]};\\
\mathfrak{c}_{e^+}\circ\pi_{V_{e^+}}(\phi_{e^+}(\mathcal{X}_e)),&\quad \text{if }\alpha\not\in T_{[V]}.
\end{cases}
\end{equation*} 
By Theorem \ref{thmB} the set $\pi_{V_{e^+}}(\phi_{e^+}(\mathcal{X}_e))$ is uniformly bounded, and thus $\mathfrak{c}_{e^+}\circ\pi_{V_{e^+}}(\phi_{e^+}(\mathcal{X}_e))$ is uniformly bounded. For each $[V]\in\mathfrak{S}_1$ such that $\alpha\in T_{[V]}$, let $c_{[V]}$ denote $\mathfrak{c}(b_{[V]})$, where the maps
$\mathfrak{c}$ denote the comparison maps (which are isometries) from the favourite representative of $[V]$ to the representative $V_\alpha$ (therefore, the maps $\mathfrak{c}$ change with respect to different equivalence classes).
Consider the consistent tuple 
\[\vec c=\prod_{\substack{[V]\in\mathfrak{S}_1,\\\alpha\in T_{[V]}}} c_{[V]}\]
By induction hypothesis, $G_{\Gamma\setminus\{v\}}$ is a hierarchically hyperbolic group. Therefore, the consistent tuple $\vec{c}$
admits a realization point $z\in G_{\Gamma\setminus\{v\}}$, and thus we obtain that $\pi_{[V]}(z)\asymp b_{[V]}$ for every $[V]\in\mathfrak{S}_1$. Furthermore, since $\pi_{T_{[V]}}(G_{\Gamma\setminus\{v\}})$ is a point, we also have that $\pi_{T_{[V]}}(z)=b_{T_{[V]}}=\pi_{T_{[V]}}(G_{\Gamma\setminus\{v\}})$ for every $T_{[V]}\in\mathfrak{S}_2$. That is, the second condition of hierarchical quasiconvexity is proved, and the inclusion $G_{\Gamma\setminus\{v\}}\hookrightarrow G_{\Gamma}$ is a hierarchically quasiconvex hieromorphism. 

Moreover, for each $V\in\mathfrak{S}_{\Gamma\setminus\{v\}}$ the map $\mathcal{C}V\to\mathcal{C}[V]$ is an isometry. Note that, if an element $[V]\sqsubseteq i^{\lozenge}(U)=[U]$, where $U\in\mathfrak{S}_{\Gamma\setminus\{v\}}$, then $T_{[U]}\subseteq T_{[V]}$. By assumption $\alpha\in T_{[U]}$, and therefore $\alpha\in T_{[V]}$ and there exists $V\in\mathfrak{S}_{\Gamma\setminus\{v\}}$ such that $i^{\lozenge}(V)=[V]$.

Thus, we proved that all induction hypotheses are satisfied by the inclusion $G_{\Gamma\setminus\{v\}}\hookrightarrow G$, that is that the embedding is a full, hierarchically quasiconvex hieromorphism, which induces isometries at the level of hyperbolic spaces. 

\smallskip
To deduce the same for an arbitrary $G_\Delta$, we proceed as follows. If $\Delta=\Gamma\setminus\{u\}$ for some (other) vertex $u\in V$, then the above argument, where in Equation \eqref{splitting} we consider the splitting over the subgroup $G_{\link(u)}$, proves that the inclusion $G_\Delta \hookrightarrow G$ satisfies the desired properties. If not, then $\Delta$ is a proper subgraph of $\Gamma\setminus\{u\}$, for some $u\in V$. Induction proves that the embedding $G_\Delta \hookrightarrow G_{\Gamma\setminus\{u\}}$ satisfies said properties, and again the above argument proves the claim for the inclusion $G_{\Gamma\setminus\{u\}}\hookrightarrow G$. As fullness, hierarchical quasiconvexity, and inducing isometries at the level of hyperbolic spaces, are all properties preserved by composition of hieromorphisms, we conclude that the inclusion $G_\Delta\hookrightarrow G$ satisfies the inductive statement, and the proof is thus complete.
\end{proof}


\begin{thebibliography}{10}

\bibitem{ABD} C. Abbott, J. Behrstock, M.G. Durham, \emph{Largest acylindrical actions and stability in hierarchically hyperbolic groups}. ArXiv preprint (2017): \url{https://arxiv.org/abs/1705.06219};

\bibitem{EminaAlibegovic} E. Alibegovi\'c, \emph{A combination theorem for relatively hyperbolic groups}. Bull. London Math. Soc. 37 no. 3, 459 -- 466, 2005;

\bibitem{BakerCooper} M.Baker, D.Cooper, \emph{A combination theorem for convex hyperbolic manifolds, with applications to surfaces in $3$-manifolds}. J. Topol. 1 (2008), no. 3, 603 -- 642;

\bibitem{BDM} J. Behrstock, C. Dru\c{t}u, L. Mosher; \emph{Thick metric spaces, relative hyperbolicity, and quasi-isometric rigidity}. Math. Ann. 344 (2009), no. 3, 543 -- 595; 

\bibitem{BHS1} J. Behrstock, M. F. Hagen, A. Sisto, \emph{Hierarchically hyperbolic spaces I: curve complexes for cubical groups}. Geom. Topol. 21 (2017), no. 3, 1731 -- 1804;

\bibitem{BHS2} J. Behrstock, M. F. Hagen, A. Sisto, \emph{Hierarchically hyperbolic spaces II: combination theorems and the distance formula}. To appear in Pacific J. Math.;

\bibitem{asdim} J. Behrstock, M. F. Hagen, A. Sisto, \emph{Asymptotic dimension and small-cancellation for hierarchically hyperbolic spaces and groups}. Proc. Lond. Math. Soc. (3) 114 (2017), no. 5, 890 -- 926;

\bibitem{quasiflats} J. Behrstock, M. F. Hagen, A. Sisto, \emph{Quasiflats in hierarchically hyperbolic spaces}. ArXiv preprint (2017): \url{https://arxiv.org/abs/1704.04271};

\bibitem{BestvinaFeighn} M. Bestvina, M. Feighn, \emph{A combination theorem for negatively curved groups}. J. Differential Geom. 35 (1992), no. 1, 85 -- 101;

\bibitem{BKS} M. Bestvina, B. Kleiner, M. Sageev, \emph{Quasiflats in $CAT(0)$ $2$-complexes}. Algebr. Geom. Topol., 16(5): 2663 -- 2676, 2016;

\bibitem{Bo} B. Bowditch, \emph{Relatively hyperbolic groups}. Internat. J. Algebra Comput. 22 (2012), no. 3, 1250016, 66 pp;

\bibitem{Dahm} F. Dahmani, \emph{Combination of convergence groups}. Geom. Topol. 7 (2003) 933 -- 963;

\bibitem{DHS} M.G. Durham, M.F. Hagen, A. Sisto, \emph{Boundaries and automorphisms of hierarchically hyperbolic spaces}. Geom. Topol. 21 (2017) 3659 -- 3758;


\bibitem{G92} S. M. Gersten, \emph{Dehn functions and $\ell_1$-norms of finite presentations}. Algorithms and classification in combinatorial group theory (Berkeley, CA, 1989), 195 -- 224, Math. Sci. Res. Inst. Publ., 23, Springer, New York, 1992;

\bibitem{Gitik} R. Gitik, \emph{Ping-pong on negatively curved groups}. J. Algebra 217 (1999), no. 1, 65 -- 72;

\bibitem{Gr} M. Gromov, \emph{Hyperbolic groups}. Essays in group theory, 75 -- 263, Math. Sci. Res. Inst. Publ., 8, Springer, New York, 1987;

\bibitem{hyprig} T. Haettel, \emph{Hyperbolic rigidity of higher rank lattices}. ArXiv preprint(2017): \url{https://arxiv.org/abs/1607.02004};

\bibitem{Ha1} M. F. Hagen, \emph{Weak hyperbolicity of cube complexes and quasi-arboreal groups}. J. Topol. 7 (2014), no. 2, 385 -- 418;

\bibitem{HaSu} M. F. Hagen, T. Susse, \emph{On hierarchical hyperbolicity of cubical groups}. To appear in Isr. J. Math.;

\bibitem{WiHs} T. Hsu, D. T. Wise, \emph{Cubulating malnormal amalgams}. Invent. Math. 199 (2015), no. 2, 293 -- 331;

\bibitem{huang} J. Huang, \emph{Top dimensional quasiflats in $CAT(0)$ cube complexes}. Geom. Topol. 21 (2017), no. 4, 2281 -- 2352;

\bibitem{VirtualAmalgamation} E. Martinez-Pedroza, A. Sisto, \emph{Virtual amalgamation of relatively quasiconvex subgroups}. Algebr. Geom. Topol. 12 (2012), 1993 -- 2002;

\bibitem{MM1} H. A. Masur, Y. N. Minsky, \emph{Geometry of the complex of curves I: hyperbolicity}. Invent. Math. 138 (1999), 103 -- 149;

\bibitem{MM2} H. A. Masur, Y. N. Minsky, \emph{Geometry of the complex of curves II: Hierarchical structure}. Geom. Funct. Anal. 10 (2000), 902~--~974;

\bibitem{Me} J. Meier, \emph{When is the graph product of hyperbolic groups hyperbolic?} Geom. Dedicata 61 (1996), no. 1, 29 -- 41;

\bibitem{MinasyanOsin} A.Minasyan, D. Osin, \emph{Acylindrical hyperbolicity of groups acting on trees}. Math. Ann. 362 (2015), no. 3 -- 4, 1055 -- 1105;

\bibitem{StrongRelGps} M. Mj, L. Reeves, \emph{A combination theorem for strong relative hyperbolicity}. Geom. Topol. 12 (2008) 1777 -- 1798;

\bibitem{Os1} D. Osin, \emph{Relatively hyperbolic groups: intrinsic geometry, algebraic properties, and algorithmic problems}. Mem. Amer. Math. Soc. 179 (2006), no. 843;

\bibitem{Os2} D. Osin, \emph{Acylindrically hyperbolic groups}. Trans. Amer. Math. Soc. 368 (2016), no. 2, 851 -- 888;

\bibitem{Sa1} M. Sageev, \emph{CAT(0) cube complexes and groups}. Geometric group theory, 7 -- 54, IAS/Park City Math. Ser., 21, Amer. Math. Soc., Providence, RI, 2014;

\bibitem{Sp1} D. Spriano, \emph{Hyperbolic HHS I: factor systems and quasi-convex subgroups}. ArXiv preprint (2017): \url{https://arxiv.org/abs/1711.10931};

\bibitem{Sp2} D. Spriano, \emph{Hyperbolic HHS II: graphs of hierarchically hyperbolic groups}. ArXiv preprint(2018): \url{https://arxiv.org/abs/1801.01850};

\bibitem{KateVokes} K. Vokes, \emph{Hierarchical hyperbolicity of graphs of multicurves}. ArXiv preprint (2017): \url{https://arxiv.org/abs/1711.03080};

\bibitem{Wi} D. T. Wise, \emph{From riches to raags: $3$-manifolds, right-angled Artin groups, and cubical geometry}. CBMS Regional Conference Series in Mathematics, vol. 117, Published for the Conference Board of the Mathematical Sciences, Washington, DC, 2012.


\end{thebibliography}
\end{document}